\documentclass[letterpaper,11pt]{article}
\usepackage{tabularx} 
\usepackage{amsmath}  
\usepackage[shortlabels]{enumitem}
\usepackage{graphicx} 
\usepackage[margin=1in,letterpaper]{geometry} 
\usepackage[final]{hyperref} 
\usepackage{dsfont}

\usepackage{dingbat}
\usepackage{lmodern}
\usepackage{comment}
\usepackage{xfrac}
\usepackage[T1]{fontenc}
\usepackage{amsmath}
\usepackage{subcaption}
\usepackage{graphicx}
\usepackage{chngcntr}
\usepackage{amsthm}
\usepackage[export]{adjustbox}
\usepackage[english]{babel}
\usepackage{amsmath, varwidth}
\usepackage{bbm}
\usepackage{amsfonts}

\usepackage[refpage,noprefix,intoc]{nomencl} 
\usepackage[dvipsnames]{xcolor}
\usepackage{cleveref}
\usepackage[utf8]{inputenc}
\usepackage{csquotes}

\usepackage[normalem]{ulem}

\usepackage{color,xcolor}
\definecolor{darkgreen}{rgb}{0,0.5,0}

\usepackage{etoolbox}

\renewcommand*{\pagedeclaration}[1]{\unskip\dotfill}	
\makenomenclature										
\newtheorem{theorem}{Theorem}[section]
\newtheorem{lemma}[theorem]{Lemma}
\newtheorem{corollary}[theorem]{Corollary}

\newtheorem{proposition}[theorem]{Proposition}

\theoremstyle{remark}
\newtheorem{remark}[theorem]{Remark}



\newcommand{\sqbracket}[1]{\left[#1\right]}

\newcommand{\TT}{\mathbb{T}}
\newcommand{\R}{\mathbb{R}}
\newcommand{\N}{\mathbb{N}}

\newcommand{\bP}{\mathbb{P}}

\newcommand{\bQ}{\mathbb{Q}}
\newcommand{\bE}{\mathbb{E}}
\newcommand{\Z}{\mathbb{Z}}

\newcommand{\cU}{\mathcal{U}}

\newcommand{\cN}{\mathcal{N}}
\newcommand{\cH}{\mathcal{H}}
\newcommand{\cW}{\mathcal{W}}
\newcommand{\cR}{\mathcal{R}}
\newcommand{\cQ}{\mathcal{Q}}

\newcommand{\cF}{\mathcal{F}}
\newcommand{\cs}{c^*}

\newcommand{\gs}{\gamma}

\newcommand{\cproba}[4]{\mathbf{P}_{#1}^{#2}[#3\ | \ #4 ]}
\newcommand{\cprobaa}[4]{\mathbf{P}_{#1}^{#2}\bigl[#3\ \bigl| \ #4 \bigl]}

\newcommand{\proba}[3]{\mathbf{P}_{#1}^{#2}[#3]}
\newcommand{\probaa}[3]{\mathbf{P}_{#1}^{#2}\bigl[#3\bigl]}
\newcommand{\probaaa}[3]{\mathbf{P}_{#1}^{#2}\Bigl[#3\Bigl]}

\newcommand{\prob}[3]{\mathbb{P}_{#1}^{#2}[#3]}

\newcommand{\probbb}[3]{\mathbb{P}_{#1}^{#2}\Bigl[#3\Bigl]}

\newcommand{\parentesis}[1]{(#1)}
\newcommand{\parentesiss}[1]{\bigl(#1\bigl)}

\newcommand{\expec}[3]{\mathbb{E}_{#1}^{#2}[#3]}
\newcommand{\expecdd}[3]{\mathbb{E}_{#1}^{#2}\bigl[#3\bigl]}
\newcommand{\expecddd}[3]{\mathbb{E}_{#1}^{#2}\Bigl[#3\Bigl]}
\newcommand{\expecdddd}[3]{\mathbb{E}_{#1}^{#2}\biggl[#3\biggl]}

\newcommand{\expected}[3]{\mathbf{E}_{#1}^{#2}[#3]}
\newcommand{\expectedd}[3]{\mathbf{E}_{#1}^{#2}\bigl[#3\bigl]}

\newcommand{\cexpectedd}[4]{\mathbf{E}_{#1}^{#2}\bigl[#3 \bigl|   #4 \bigl]}

\newcommand{\indc}{\mathds{1}}

\numberwithin{equation}{section}

\usepackage[numbers, sort&compress]{natbib}

\begin{document}

\title{A shape theorem for BBM in a periodic environment}
\author{Louigi Addario-Berry
\thanks{Department of Mathematics and Statistics, McGill University.
{\footnotesize \href{mailto:louigi@gmail.com}{louigi@gmail.com}.}
}
\and 
Arturo Arellano Arias
\thanks{Department of Mathematics and Statistics, McGill University.
{\footnotesize \href{mailto:arturo.arellanoarias@mail.mcgill.ca}{arturo.arellanoarias@mail.mcgill.ca}.}
}
\and
Jessica Lin
\thanks{Department of Mathematics and Statistics, McGill University.
{\footnotesize \href{jessica.lin@mcgill.ca}{jessica.lin@mcgill.ca}.}
}
}
\date{July 13, 2026}
\maketitle

\begin{abstract}
We consider the long-time behaviour of binary branching Brownian motion (BBM) where the branching rate depends on a periodic spatial heterogeneity. We prove that almost surely as $t\to\infty$, the heterogeneous BBM at time $t$, normalized by $t$, approaches a deterministic convex shape with respect to Hausdorff distance. Our approach relies on establishing tail bounds on the probability of existence of BBM particles lying in half-spaces, which in particular yields the asymptotic speed of propagation of projections of the BBM in every direction. Our arguments are primarily probabilistic in nature, but additionally exploit the existence of a ``front speed'' (or minimal speed of a pulsating traveling front solution) for the Fisher-KPP reaction-diffusion equation naturally associated to the BBM. 
\end{abstract}

\section{Introduction}\label{sec:Intro}

\subsection{Main result}\label{sec:mainresult.intro}
In the classical model of \textit{$d$-dimensional binary branching Brownian motion} (BBM), a particle undergoes a Brownian motion in $\R^{d}$ until a rate-one exponential random time. At that moment (which we refer to as the branching time), the particle produces two offspring and immediately dies; these two offspring subsequently evolve independently in the same manner. The BBM at time $t$ is the set of points in $\mathbb{R}^d$ representing the positions of all living particles at time $t$. It is well-known that the BBM expands linearly over time with velocity $\sqrt{2}$ almost surely (see for example \cite{mckean} and \cite{MR0494541} for $d=1$ and \cite{MR0670523} and \cite{MR3417448} for $d>1$). Furthermore, the BBM at time $t$, rescaled by $(\sqrt{2}t)^{-1}$, approximates the Euclidean ball in the Hausdorff distance almost surely as $t \to \infty$; see \cite{MR0518327}. This is an example of a \textit{shape theorem} for a branching process. Our objective is to obtain a shape theorem when the branching rate of the BBM is spatially periodic rather than constant.

We consider binary branching Brownian motion with a fixed continuous $\mathbb{Z}^d$-periodic, space-dependent branching rate \( g: \mathbb{R}^d \rightarrow (0, \infty) \) or $g$-BBM for short. Initially, there is one particle, denoted by $\emptyset$, at position \( X_0(\emptyset) = x \in \mathbb{R}^d \). This particle undergoes Brownian motion \( (X_s(\emptyset))_{s \geq 0} \) until a random time \( \tau_\emptyset \), at which it dies and produces two offspring. The branching time \( \tau_\emptyset \) is defined by
\begin{equation}\label{eq:intro:BranchT}
    \tau_\emptyset := \inf \left\{ t \geq 0 : \sigma_\emptyset < \int_0^t g(X_s(\emptyset)) \, ds \right\},
\end{equation}
where \( \sigma_\emptyset \) is a rate-one exponential random variable independent of \( (X_s(\emptyset))_{s \geq 0} \). The two offspring particles, positioned at $X_{\tau_\emptyset}(\emptyset)$ where their parent died, evolve independently in the same manner. Each particle is named using a word from the dictionary $\mathcal{U} := \bigcup_{n=0}^\infty \{0,1\}^n$. If a particle is named $v \in \mathcal{U}$, its two offspring are named with the concatenated words $v0$ and $v1$. 

Let $\cN_t\subset \cU$ denote the names of the particles that are alive at time $t$; for $v\in \cN_t$, let $X_t(v)$ denote the position of the particle $v$ at time $t$; and set $\mathcal{X}_t=\{X_t(v): v \in \cN_t\}$. We refer to $(\mathcal{X}_t)_{t \ge 0}$ as a \emph{$g$-BBM}. For $s\leq t$, we will also denote by $X_s(v)$ the position of the unique living ancestor of $v$ at time $s$. For $x\in \R^{d}$, we let \( \mathbf{P}_x \) denote the law of the $g$-BBM starting with a single particle at position $x$.

The main results of this work are the following shape theorems. The first result is that the $g$-BBM, rescaled by $t^{-1}$, converges almost surely to a specific convex set $\cW$, with respect to the Hausdorff distance.
\begin{theorem}\label{thm:shapeIntroHaussdorf}
    Let $g:\R^d\rightarrow [0,+\infty)$ with $g\not \equiv 0$ be $\Z^d$-periodic and $\alpha$-H\"older continuous for some $\alpha \in (0,1)$. Let $(\mathcal{X}_t)_{t \ge 0}$ be a $g$-BBM. There exists a convex set $\cW\subset \R^d$, with $0\in \mathrm{int}(\cW)$ such that, for all $\epsilon\in (0,1)$ and $x\in \R^d$, 
    \begin{equation*}
        \proba{x}{}{\exists T>0 \ \mathrm{s.t.} \ \forall t\geq T, d_{\mathrm{H}}(\cW, t^{-1}(\mathcal{X}_t-x))<\epsilon} = 1.
    \end{equation*}
\end{theorem}

Theorem \ref{thm:shapeIntroHaussdorf} is a fairly direct consequence of our second main result, which is the following quantitative estimate for the shape of the convex hull of a $g$-BBM. 

\begin{proposition}\label{prop:QuantShapeTheorem}
    Let $g:\R^d\rightarrow [0,+\infty)$ with $g\not \equiv 0$ be $\Z^d$-periodic and $\alpha$-H\"older continuous for some $\alpha \in (0,1)$, and let $(\mathcal{X}_t)_{t \ge 0}$ be a $g$-BBM. 
     Let $H_t$ be the convex hull of $\mathcal{X}_t$ and let $\cW$ be the convex set appearing in Theorem \ref{thm:shapeIntroHaussdorf}. Fix $\epsilon \in (0,1)$. There exists $C=C(\epsilon)>0$, $\Gamma=\Gamma(\epsilon)>0$ and $T_0=T_0(\epsilon, d, \|g\|_{\infty})$ such that, for every $t>T_0$,
    \begin{equation}\label{eq:quantShapeThm}
        \inf_{x\in \R^d}\mathbf{P}_x[(1-\epsilon)\cW\subset t^{-1}(H_t-x) \subset(1+\epsilon)\cW] \geq 1-Ce^{-t\Gamma(\epsilon)}.
    \end{equation}
\end{proposition}
While our main results require that $g$ is H\"older continuous, we highlight that none of our estimates depend on $\|g\|_{C^{0,\alpha}}$. The assumption of H\"older continuity is used for the relevant PDE tools we invoke to construct the shape $\mathcal{W}$ (namely the existence of principal eigenvalues associated to certain differential operators). As we discuss in Remark \ref{r.holder}, if $g$ is merely continuous (and not necessarily H\"older continuous), one can identify a candidate shape $\mathcal{W}$ using approximation arguments. We conjecture that Theorem \ref{thm:shapeIntroHaussdorf} holds if $g$ is merely continuous (and not H\"older continuous); this is straightforwardly seen to be true if, for some $\alpha \in (0,1)$, one can construct a monotone sequence $(g_n)$ of $\alpha$-H\"older continuous functions such that  $g_{n}\rightarrow g$ uniformly.

In the next two subsections, we first describe the convex set $\cW$ appearing in Theorem \ref{thm:shapeIntroHaussdorf} and in Proposition \ref{prop:QuantShapeTheorem}, as well as the connections between our results and the literature on reaction-diffusion equations. We then provide a high-level description of our proof techniques.

\subsection{Connection to F-KPP equations}
It is well-known that there is a connection between classical BBM and the homogeneous Fisher-KPP (F-KPP) reaction-diffusion equation, and we now describe this connection in the case of $g$-BBM and a heterogeneous version of the Fisher-KPP equation. Fix a measurable function $f:\R^d \rightarrow [0,1]$, and define
\begin{equation}\label{eq:multiIntro}
    q(t, x) :=1- \mathbf{E}_x \bigl [ \prod_{v \in \mathcal{N}_t} f(X_t(v)) \bigl].
\end{equation}
Under the same conditions for the branching rate $g$ as in Theorem \ref{thm:shapeIntroHaussdorf}, it can be shown using the techniques in \cite{mckean} that $q$ is a classical solution of the F-KPP equation
\begin{equation}\label{eq:F-KPPintro}
    \begin{cases}
        \partial_t q = \frac{1}{2} \Delta q + g(x)(q-q^2) & \text{in } (0, \infty) \times \mathbb{R}^d, \\
        q(0, x) = q_0(x) & \text{in } \mathbb{R}^d,
    \end{cases}
\end{equation}
with $q_0(x)=1-f(x)$.

The long-time behaviour of solutions $q$ of \eqref{eq:F-KPPintro} has been widely studied in the theory of F-KPP equations. In order to describe this, we first discuss the known properties of the so-called \emph{front speed} $c^{*}: S^{d-1}\rightarrow (0, \infty)$. For each $e\in S^{d-1}$, the front speed $c^{*}(e)$ is defined according to a variational formula involving the principal eigenvalue of an associated linear elliptic PDE (see \eqref{eq:prelimCriticalSpeed}, below). The quantity $c^*(e)$ is well-studied in the PDE literature, and it is known as the critical/minimal speed of a pulsating traveling front solution in direction $-e$ (see for example \cite[Section 1.1]{MR3682669}).
For rather general initial conditions, the function $c^*$ controls the long-term behaviour of solutions $q$ of \eqref{eq:F-KPPintro} with front-like initial data. More precisely, for $e\in S^{d-1}$ and $c\in \R$, let
\begin{equation*}
    \cH_{e,c}:=\{x\in \R^d: x\cdot e = c\},\
    \cH_{e,c}^+:=\{x\in \R^d: x\cdot e> c\}  \text{ and } \cH_{e,c}^-:=\{x\in \R^d: x\cdot e \leq c\}.
\end{equation*}
Then it is known that if $q_{0}$ satisfies that 
\begin{equation*}
    \lim_{x\cdot (-e)\to -\infty} q_{0}(x)=1\quad\text{and}\quad \text{$q_{0}(x)=0$ for all $x\cdot (-e)$ sufficiently large,}
\end{equation*}
then for any compact $K\subseteq \cH^{+}_{-e,0}$
\begin{equation*}
\lim_{t\to\infty} \inf_{x\in [c^{*}(e)(-e)-K]t} q(t,x)=1\quad\text{and}\quad \lim_{t\to\infty} \sup_{x\in [c^{*}(e)(-e)+K]t} q(t,x)=0.
\end{equation*}
This implies that
\begin{equation}\label{e.hypersol}
    q(t,xt)\xrightarrow[]{t\to\infty} \indc_{\cH^{-}_{-e,c^{*}(e)}}(x)
\end{equation}
locally uniformly in $x \in \R^d$, away from the interface $\cH_{-e, c^{*}(e)}$; see \cite{linzlatos,MR3682669}. 

Using property \eqref{e.hypersol} of $c^{*}(e)$, we now sketch a key aspect of our proof. Consider $q$ solving \eqref{eq:F-KPPintro} with $q_{0}(\cdot)=1-\indc_{\cH^{-}_{e,0}}=1-\indc_{\left\{x\cdot e \le 0\right\}}=\indc_{\cH^{-}_{-e,0}}$. According to the McKean formula, 
\begin{equation}\label{e.qtx}
q(t,x)=1-\mathbf{E}_x \bigl [ \prod_{v \in \mathcal{N}_t} \indc_{\left\{X_{t}(v)\cdot e\leq 0\right\}}\bigr].
\end{equation}
Suppose that the convergence \eqref{e.hypersol} holds for this solution $q$, and moreover that it holds uniformly in space away from the interface $\cH_{-e, c^{*}(e)}$, rather than locally uniformly. Then for every $\epsilon>0$, we have $\sup_{x \in \cH^{+}_{-e,(1+\epsilon)c^{*}(e)t}} q(t,x)\xrightarrow[]{t\to\infty} 0$. By \eqref{e.qtx}, this is equivalent to the statement that
\begin{equation*}
\lim_{t\to\infty} \inf_{x\in \cH^{-}_{e,-(1+\epsilon)c^{*}(e)t}} \mathbf{E}_x \bigl [ \prod_{v \in \mathcal{N}_t} \indc_{\left\{X_{t}(v)\cdot e\leq 0\right\}}\bigr]=1
\end{equation*}
This suggests (but does not prove, since we assumed without justification that \eqref{e.hypersol} holds for the function $q$ from \eqref{e.qtx}) that for all $t$ sufficiently large, for all $x\in \R^{d}$
\begin{equation*}
\proba{x}{}{\forall v\in \mathcal{N}_{t}: (X_{t}(v)-x)\cdot e< (1+\epsilon)c^{*}(e)t}=1. 
\end{equation*}
An analogous argument suggests that for all $t$ sufficiently large, for all $x\in \R^{d}$
\begin{equation*}
\proba{x}{}{\exists v\in \mathcal{N}_{t}: (X_{t}(v)-x)\cdot e> (1-\epsilon)c^{*}(e)t}=1.
\end{equation*}
We shall in fact prove precise (and quantitative) versions of these bounds, which are critical to our proof, and which are stated in the following propositions. 
\begin{proposition}\label{propo:introUpper}
Assume the hypotheses of Theorem \ref{thm:shapeIntroHaussdorf} and let $\epsilon\in (0,1)$. Then there exist $C>0$ and $\Gamma>0$ such that, for every $e\in S^{d-1}$ and $t\geq 0$,
    \begin{equation}\label{eq:introUpper}
        \sup_{x\in \R^d}\proba{x}{}{\exists v\in \mathcal{N}_t : (X_t(v)-x)\cdot e \geq (1+\epsilon)t\cs(e)} \leq Ce^{-t\Gamma}.
    \end{equation}
\end{proposition}

\begin{proposition}\label{propo:introLower}
    Assume the hypotheses of Theorem \ref{thm:shapeIntroHaussdorf} and let $\epsilon\in (0,1)$. Then there exist $C>0$ and $\Gamma>0$ such that, for every $e\in S^{d-1}$, there exists $t_0=t_0(e)$ such that, for every $t\geq t_0$,
    \begin{equation}
        \sup_{x\in \R^d}\proba{x}{}{\forall v\in \mathcal{N}_t : (X_t(v)-x)\cdot e \leq (1-\epsilon)t\cs(e)} \leq Ce^{-t\Gamma}.
    \end{equation}
\end{proposition}

We prove Proposition \ref{propo:introUpper} and Proposition \ref{propo:introLower} using completely probabilistic arguments; in particular, our proofs do not rely on \eqref{e.hypersol} or on the McKean formula. Aside from some properties of $c^{*}(e)$ which arise from the definition \eqref{eq:prelimCriticalSpeed}, and properties of the associated principal eigenfunction of the linearization of \eqref{eq:F-KPPintro}, our arguments are self-contained. We give a more detailed overview of the proofs of Proposition \ref{propo:introUpper} and Proposition \ref{propo:introLower} in the following section.

We obtain the following corollary, which is an analogous, quantitative version of \eqref{e.hypersol} for solutions of \eqref{eq:F-KPPintro} with initial data $q_0(\cdot)=\indc_{\cH^{-}_{-e,0}}$. The proof consists in expressing the solution in terms of the $g$-BBM via McKean's formula, followed by applying Proposition \ref{propo:introUpper} and Proposition \ref{propo:introLower}.
\begin{corollary}\label{Prop:PDEintroheaviside}
    Fix $e\in S^{d-1}$. Let $q$ be the solution of \eqref{eq:F-KPPintro} with initial data $q_0(\cdot)=\indc_{\cH^{-}_{-e,0}}$. For every $\epsilon\in (0,1)$, there exist $C>0$ and $\Gamma>0$ such that:
    \begin{enumerate}[(i)]
        \item for all $t>0$,
        \begin{equation*}
            \sup_{x\in (1+\epsilon)\cH^+_{-e,\cs(e)}} q(t,xt) \leq Ce^{-\Gamma t},
        \end{equation*}
        \item for all $t$ sufficiently large, 
        \begin{equation*}
            \sup_{x\in (1-\epsilon)\cH^-_{-e,\cs(e)}} \bigl (1-q(t,xt)\bigl) \leq Ce^{-\Gamma t}.
        \end{equation*}
    \end{enumerate}
\end{corollary}

The bounds established in Corollary \ref{Prop:PDEintroheaviside}(i) have been derived in the literature through different methods by using the comparison principle for parabolic equations with the existence of pulsating traveling fronts (see for example an argument by Nadin \cite[proof of Proposition 2.10]{MR2555178} for general front-like initial data). A similar bound for a solution $q$ started from compactly-supported initial data appears in various other references in the PDE literature (see for example \cite[proof of Theorem 3.6]{Roquejoffre2014KPPII}) and separately, in the book of Freidlin \cite[Chapter 7.3, Lemma 3.1]{Freidlin+1985} using large deviation techniques. To our knowledge, a quantitative estimate as in Corollary \ref{Prop:PDEintroheaviside}(ii) has not appeared in the literature.

Equipped with the (front speed) function $c^{*}$, we can now define the associated \emph{Wulff shape} $\mathcal{W}$ of $\cs$, which we denote by $\cW(c^{*})$, by
\begin{equation}\label{wulffintro}
    \cW(\cs):=\bigcap_{e\in S^{d-1}}\{x\in \R^d: x\cdot e\leq \cs(e)\}.
\end{equation}
The set $\mathcal{W}=\mathcal{W}(c^{*})$ will be exactly the convex set appearing in Theorem \ref{thm:shapeIntroHaussdorf}.

The set $\mathcal{W}(c^{*})$ has arisen in the study of Fisher-KPP reaction-diffusion equations, and is known as the asymptotic spreading set. It can be defined as the unit ball with respect to the metric given by the \emph{spreading speed} function $w: S^{d-1}\rightarrow (0, \infty)$ which satisfies the following: if $q_{0}: \R^{d}\rightarrow [0,1]$ is compactly-supported, and $q_{0}\neq 0$, then for every $\epsilon>0$, locally uniformly in $x$, 
\begin{align}\label{e.spreadingdef}
\lim_{t\to\infty} q(t,x+t(1+\epsilon)w(e)(-e))=0\quad\text{and}\quad \lim_{t\to\infty} q(t,x+t(1-\epsilon)w(e)(-e))=1. 
\end{align}

Moreover, the $\tfrac{1}{2}$ level-set of $q$, scaled by $\tfrac{1}{t}$, asymptotically approaches $\mathcal{W}(c^{*})$ in Hausdorff distance as $t\to \infty$.
In the above, we refer to $w(e)$ as the {\em spreading speed} in direction $-e$. It turns out that we can express
\begin{equation*}
    \mathcal{W}(c^{*})=\left\{s(-e): 0\leq s\leq w(e), e\in S^{d-1}\right\},
\end{equation*}
and this representation is a consequence of the Freidlin-G\"artner formula, which relates $c^{*}$ and $w$:
\begin{equation}\label{eq:FGintro}
    w(e):=\inf_{\substack{e'\in S^{d-1}\\e'\cdot e>0}}\frac{\cs(e')}{e'\cdot e}.
\end{equation}

\subsection{Proof overview, main ideas and challenges}

We begin with a brief overview of the proofs of Proposition \ref{propo:introUpper} and Proposition \ref{propo:introLower}. 
Through this overview, we will assume that the initial particle of the $g$-BBM is located at the origin. Due to the spatial heterogeneity of the model, this assumption requires justification, which is provided by the uniformity of the estimates in  Propositions~\ref{propo:introUpper} and~\ref{propo:introLower}.

One key property of the homogeneous multidimensional BBM is that the process obtained by projecting the BBM onto any one-dimensional subspace is itself a one-dimensional BBM. This property has been used, among others, by Biggins \cite{MR0518327} to obtain a shape theorem; by Mallein \cite{MR3417448} to determine the maximum displacement of the BBM; and by Stasi\'nski, Berestycki, and Mallein \cite{MR4291461} to study the derivative martingale of multidimensional BBM. This technique does not suit our setting, as the spatial heterogeneity means that such projections are no longer branching processes (in fact, they are no longer even Markov processes as the branching rate cannot be recovered by observing only the projected process). Instead, for a given direction $e\in S^{d-1}$, we check if particles lie within half-spaces moving with speed strictly larger/smaller than the quantity $\cs(e)$ introduced above, that is, the half-spaces
\begin{equation}\label{intro:halfspace}
    \begin{split}
        (1+\epsilon)t\cH^+_{e,\cs(e)}&:=\{x\in \R^d: x\cdot e\geq (1+\epsilon)t\cs(e)\},\\  \ (1-\epsilon)t\cH^+_{e,\cs(e)}&:=\{x\in \R^d: x\cdot e\geq (1-\epsilon)t\cs(e)\}.
    \end{split}
\end{equation}

Proposition \ref{propo:introUpper} asserts that at time $t$, the probability of finding a particle within $(1+\epsilon)t\cH^+_{e,\cs(e)}$ decays exponentially in $t$. This result can be proved via large deviations techniques, similar to those used in \cite{MR553200} (see also \cite[Chapter 7.3]{Freidlin+1985}); however, our proof of Proposition \ref{propo:introUpper} parallels the one-dimensional approach in \cite{MR4492971}, proceeding via a many-to-one lemma. 

Our approach for Proposition \ref{propo:introLower} leverages the branching property of the $g$-BBM. We fix $T>0$ large, and consider a multitype discrete-time branching process formed by particles that stay in the moving half-space $(1-\epsilon)(nT)\cH^+_{e,\cs(e)}$ at integer multiples of $T$. Via a large deviation principle, it can be shown that $T$ can be chosen sufficiently large such that the expected number of offspring in the multitype discrete-time branching process is strictly larger than one. This, in turn, implies that such a branching process survives with positive probability. The above reasoning, combined with a classical cut-off argument employed in the analysis of the homogeneous BBM, yields the bound in Proposition \ref{propo:introLower}. While this line of proof is fairly standard, the heterogeneity of the environment requires us to obtain uniform bounds independent of the initial position of the process. For this we use a generalization of the G\"artner-Ellis theorem (Theorem \ref{thm:GEgenralization}) to accommodate multiple families of measures; a Poissonian thinning of the BBM to facilitate comparison with a homogeneous BBM; and survival results for multitype branching processes (Lemma \ref{lemma:SurvivalProbability}).

We now explain how we use Proposition \ref{propo:introUpper} and Proposition \ref{propo:introLower} to prove the quantitative shape theorem (Proposition \ref{eq:quantShapeThm}).  For this, we additionally use the following straightforward geometric result, which says that arbitrary Wulff shapes can be well-approximated from within and from without by intersections of finitely many half-spaces. For future use, we state the result with a more general function $\mathfrak{c}^*$ whose range is bounded away from $0$ and from $\infty$.

\begin{proposition}\label{prop:wulffapprox}
    Fix $a \in (0,1)$, let $\mathfrak{c}^*:S^{d-1}\rightarrow (a,a^{-1})$, and let $\cW(\mathfrak{c}^*)$ be the Wulff shape of $\mathfrak{c}^*$. Then for $\epsilon\in (0,1)$ there exist finite sets $\cR,\cQ\subset S^{d-1}$ such that 
    \begin{enumerate}[(i)]
        \item $ \bigcap_{r\in \cR} \cH_{r,\mathfrak{c}^*(r)}^{-}\subset(1+\epsilon)\cW(\mathfrak{c}^*)$, and
        \item for every compact convex set $K\subset B(0,2a^{-1})$, if for every $q\in \cQ$, $K\cap \cH^+_{q,\mathfrak{c}^*(q)} \neq \emptyset$, then 
        \begin{equation*}
            (1-\epsilon)\cW(\mathfrak{c}^*)\subset K.
        \end{equation*}
    \end{enumerate}
\end{proposition}
The proof of Proposition~\ref{prop:wulffapprox} appears in Appendix~\ref{append:ApproxWulff}. We first apply (i) combined with Proposition \ref{propo:introUpper}, to prove that for $t$ large and with high probability, $H_t/t\subset (1+\epsilon)\cW$. Then, we apply (ii) combined with Proposition \ref{propo:introLower} to prove that for $t$ large and with high probability, $(1-\epsilon)\cW \subset H_t/t$. Together, these two bounds prove the quantitative shape theorem; the details appear in Section \ref{s.shape}.

Finally, to prove Theorem~\ref{thm:shapeIntroHaussdorf}, we use that by Carath\'eodory's theorem (see \cite[Theorem 1.3]{MR1242986}), any point $x \in \cW$ can be written as a convex combination of at most $d+1$ extreme points of $\cW$; that is, there exists $d' \in [d+1]:=\left\{1, 2, \ldots, d+1\right\}$, extreme points $\xi_1,\ldots,\xi_{d'}$ of $\cW$, and $\alpha_1,\ldots,\alpha_{d'}>0$ with $\sum_{i \in [d']} \alpha_i=1$, such that $x=\sum_{i\in[d']} \alpha_i \xi_i$. Roughly speaking, we can then use Proposition~\ref{prop:QuantShapeTheorem} and the Markov property of the $g$-BBM to show that with very high probability, there exists $v \in \cN_t$ such that for each $i \in [d']$, we have 
\[X_{(\alpha_1+\ldots+\alpha_i)t}(v)-X_{(\alpha_1+\ldots+\alpha_{i-1})t}(v)\approx \alpha_it \xi_i\, .
\]
If this occurs, then by a telescoping sum we have $t^{-1}X_t \approx t^{-1}\sum_{i \in [d']} \alpha_i t \xi_i = x$. We then use the compactness of $\cW$ and the Borel-Cantelli lemma to obtain the almost sure statement of Theorem~\ref{thm:shapeIntroHaussdorf}.

\subsection{Literature review and related results}
For $d=1$, the convex hull is always an interval formed by the leftmost and rightmost particles of the BBM. For homogeneous media, the median location of the rightmost particle was determined to have position $(1+o(1))\sqrt{2}t$ by McKean \cite{mckean}, and to have position $m(t)=\sqrt{2}t-3/(2\sqrt{2})\log t+O(1)$ by Bramson \cite{MR0494541} (see also Roberts \cite{MR3127890}); by symmetry, a corresponding result holds for the leftmost particle. 
This result is stronger than a shape theorem since, in addition to determining the speed of expansion, it also establishes a logarithmic correction. For the periodic case in $d=1$, an analogous result was established by Lubetzky, Thornett, and Zeitouni \cite{MR4492971}. Here, the authors identify two quantities $v^*$ and $\lambda^*$ such that the rightmost particle is located near $m(t)=v^*t-(3/2\lambda^*)\log t$ and establish the convergence in distribution of the maximum along subsequences. We also mention the works of \cite{MR4079432, MR4585411}, which prove invariance principles for the maximal particle in branching random walks/BBM in certain random environments.

In dimension $d>1$, for a homogeneous environment, a shape theorem was proven by Biggins \cite{MR0518327, Biggins2}; he showed that under conditions on the logarithmic moments of homogeneous branching random walks and of age-dependent branching processes, the convex hull normalized by $t$ converges to a convex set $\cW$ with $0\in \mathrm{int}(\cW)$. His results imply that a homogeneous BBM at time $t$, normalized by $t$, converges to the ball $B_{\sqrt{2}}$ in the Hausdorff distance. A result on logarithmic corrections was obtained by Mallein \cite{MR3417448}, who showed that the farthest particle is located approximately at distance $R(t)=\sqrt{2}t+(d-4)/(2\sqrt{2})\log t$ from the origin. The asymptotic local structure near extremal particles has been determined recently by Kim, Lubetzky and Zeitouni \cite{MR4564433}, Berestycki, Kim, Lubetzky, Mallein, and Zeitouni \cite{BKLMZ}, and Kim and Zeitouni \cite{MR4998359}.

While it is believed that many of the aforementioned results may extend to multidimensional BBM in a periodic environment, few of these conjectures have been rigorously proven. In \cite{MR4162842}, Hebbar, Koralov, and Nolen derive precise asymptotics for the transition kernel of BBM in a periodic environment, specifically, the function $m$ such that for any Borel set $A$, and for all $x\in \R^d$,
\begin{equation*}
    \expectedd{x}{}{\sum_{u\in \mathcal{N}_s}\mathbbm{1}_{\{X_s(u)\in A \}}}=\int_A m(x,y,s) dy.
\end{equation*}
The work \cite{MR4162842} also considers models which allow for killing and for diffusions with drift in periodic media.
Their results imply that there exists a convex function $\Phi$ such that $m(t,x,y)\asymp t^{-d/2}e^{-t\Phi(\frac{y-x}{t})}$. This suggests that the BBM at time $t$ is concentrated within the set $t\{\Phi\leq 0\}$. (In fact, it can be shown that $\{\Phi\leq 0\}=\cW$ for $\mathcal{W}$ in our Theorem \ref{thm:shapeIntroHaussdorf}.) Additionally, \cite{MR4162842} obtain exact asymptotics for the moments of the number of particles within regions where the population is growing. In their analysis, they use the same change of measure we do in our proof of Proposition \ref{propo:introLower}.

On the PDE side, the spreading properties of the F-KPP equation have been a rich subject of interest. For $d=1$, the spreading result \eqref{e.spreadingdef} for the homogeneous F-KPP follows from the original works of Fisher \cite{Fisher} and Kolmogorov, Petrovskii, and Piskounov \cite{KPP}, and in the periodic setting from \cite{MR3463416}. In these works, the authors identify the asymptotic spreading speed, and notably, \cite{MR3043938} and \cite{MR3463416} further provide the logarithmic correction of the level sets for solutions of F-KPP equations with compactly supported initial data.

For the multidimensional F-KPP equation, the spreading result \eqref{e.spreadingdef} in the homogeneous setting was obtained in \cite{AW}. The location of the level sets, up to the logarithmic correction, was obtained using probabilistic methods by G\"artner in \cite{MR0670523} for radially symmetric and compactly supported initial data. In the periodic setting, the asymptotic set of spreading of the F-KPP equation \eqref{eq:F-KPPintro} was obtained by G\"artner and Freidlin in \cite{MR553200}, assuming sufficient smoothness hypotheses for the function $g$ describing the periodic environment (see also \cite{Fpaper} and Chapter 7.3 of \cite{Freidlin+1985} for a more detailed presentation of these results). Alternatively, \cite{Weinberger} (using a dynamical systems approach, in a discrete setting) and \cite{MR3682669} (using a PDE approach) identify the asymptotic spreading speed that generalizes to other nonlinearities. The spreading properties of space-time periodic F-KPP reaction-diffusion equations was considered by \cite{MR2555178}. The logarithmic correction for the level sets of solutions of the F-KPP equation is obtained by Shabani \cite{shabani2019logarithmic} for solutions with compactly supported initial data.

Other types of shape theorems have been obtained for other branching systems in heterogeneous environments. For recurrent branching random walks in $\Z^d$ evolving in random environments, a shape theorem was obtained for the set of sites visited up to time $t$ by Comets and Popov in \cite{MR2365644} and \cite{MR2303944}. Interesting examples are introduced in \cite{MR2365644} and \cite{MR2303944}, where the shape has flat faces and corners. The techniques required for their analysis are quite different from the ones involved here due to the randomness in the environment. 

Spreading results for F-KPP in more general media have also been obtained. In one-dimension, the works  \cite[Chapter 7.4]{Freidlin+1985} and \cite{nolen} examine the long-time behaviour of solutions of F-KPP in stationary-ergodic random media, demonstrating the existence of a deterministic speed of a front (and a CLT in the case of \cite{nolen} under additional assumptions on the initial data). For higher dimensions, the spreading speed estimates of \cite{LS,Z1,Z2} yield the existence of a deterministic Wulff shape. An approach based on generalized principal eigenvalues is also developed in \cite{MR4493578} to prove the existence of an asymptotic spreading set for a wide range of heterogeneous environments. 

\subsection{Structure of the paper}

The preliminaries for this work are presented in Section \ref{sec:prelim}, covering the first branching time (Section \ref{sec:BranchingTime}), PDE aspects (Section \ref{section:FKPPEigenvalue}), a many-to-one lemma tailored for the $g$-BBM (Section \ref{section:manytoone}), and an interpolation result for the $g$-BBM (Section \ref{sec:interpolation}). Building upon these, the half-space estimates are derived in Section \ref{sec:ShapeThm}: Section \ref{sec:rightestimate} proves the upper bound (Proposition \ref{propo:introUpper}), while Section \ref{sec:lowerestimate} establishes the lower bound (Proposition \ref{propo:introLower}) using a large deviation principle (Section \ref{section:largedeviationprinciple}), survival results for multitype branching processes (Section \ref{section:embeddedsupercriticalbranching}), and a cut-off argument (Section \ref{section:cutoffargument}). These results are then combined with the approximation result (Proposition \ref{prop:wulffapprox}) to prove the quantitative shape theorem (Proposition \ref{prop:QuantShapeTheorem}) in Section \ref{s.shape}. In the same section, the quantitative shape theorem is then used to prove the shape theorem (Theorem \ref{thm:shapeIntroHaussdorf}). Supporting details are provided in the appendices: Appendix \ref{append:survival} contains the proof of a survival result used in Section \ref{section:embeddedsupercriticalbranching}; Appendix \ref{section:gartnerellis} establishes a uniform large deviation principle used in Section \ref{section:largedeviationprinciple}; and Appendix \ref{append:ApproxWulff} provides the proof of the Wulff shape approximation result (Proposition \ref{prop:wulffapprox}).

\section{Prerequisites}\label{sec:prelim}

It will be convenient to consider an auxiliary $d$-dimensional standard Brownian motion $(B_t)_{t\geq 0}$ defined on a filtered probability space $(\Omega,\mathcal{F},\{\cF_t\}_{t \ge 0},\bP)$ so that $(B_t)_{t\geq 0}$ is adapted to $\{\cF_t\}_{t \ge 0}$. For $x\in \R^d$, we let $\bP_x$ be the measure on $(\Omega,\cF)$ under which $B$ is a Brownian motion with $\bP_x[B_0=x]=1$, and let $\bE_x$ denote the expectation operator associated to $\bP_x$.

Throughout this paper, for $k\in \mathbb{N}$, let $C^k_{p}(\mathbb{R}^d)$ denote the set of $k$-times continuously differentiable real-valued $\Z^d$-periodic functions; also, let $C^0_p(\R^d)$ denote the set of continuous real-valued $\Z^{d}$-periodic functions. Similarly, for $\alpha\in (0,1]$ and $k \in \N_0$, we let $C^{k,\alpha}_{p}(\R^{d})$ denote the space of $\Z^{d}$-periodic $C^{k}(\R^{d})$ functions for which all second order partial derivatives are $\alpha$-H\"older continuous. Notice that a function $f\in C^{k, \alpha}_{p}(\R^{d})$ can always be identified with a function $\dot f\in C^{k, \alpha}(\mathbb{T}^{d})$ where $\mathbb{T}^{d}$ is the unit torus in $\R^{d}$; for ease of notation, we will usually write $f$ instead of $\dot f$. Also, for a function $f\in C^{k, \alpha}_{p}(\mathbb{R}^{d})$ we write
Throughout this paper, for $k\in \mathbb{N}$, let $C^k_{p}(\mathbb{R}^d)$ denote the set of $k$-times continuously differentiable real-valued $\Z^d$-periodic functions; also, let $C^0_p(\R^d)$ denote the set of continuous real-valued $\Z^{d}$-periodic functions. Similarly, for $\alpha\in (0,1]$ and $k \in \N_0$, we let $C^{k,\alpha}_{p}(\R^{d})$ denote the space of $\Z^{d}$-periodic $C^{k}(\R^{d})$ functions for which all second order partial derivatives are $\alpha$-H\"older continuous. Notice that a function $f\in C^{k, \alpha}_{p}(\R^{d})$ can always be identified with a function $\dot f\in C^{k, \alpha}(\mathbb{T}^{d})$ where $\mathbb{T}^{d}$ is the unit torus in $\R^{d}$; for ease of notation, we will usually write $f$ instead of $\dot f$. Also, for a function $f\in C^{k, \alpha}_{p}(\mathbb{R}^{d})$ we write
\begin{equation*}
    \bar{f}:=\int_{\TT^d} f(x)dx,
\end{equation*}
to denote its average in the torus.

For a vector $v=(v_1,\ldots,v_d)\in \R^d$, let $|v|$ denote its Euclidean norm and let $\|v\|_\infty=\max(|v_i|,i \in [d])$, where we will use the convention that $[d]:=\left\{1,2, \ldots, d\right\}$. Also, for a function $f:\R^d \to \R$ we write $\|f\|_\infty:= \sup(|f(x)|,x \in \R^m)$. We denote by $L^\infty(\R^d)$ the space of Borel measurable functions $f$ such that $\|f\|_\infty<\infty$. Finally, for a function $f \in C_p^{0,\alpha}(\R^{d})$ with $\alpha\in (0,1]$, we denote its $\alpha$-H\"older norm by $\|f\|_{C^{0,\alpha}}:=\|f\|_\infty + \sup_{x\neq y\in \R^d}|f(x)-f(y)|/|x-y|^\alpha$.  

Several results in this section rely on the Feynman-Kac formula, which provides a probabilistic representation for classical solutions of the parabolic equation
\begin{equation}
    \begin{cases}\label{eq:FK}
        \partial_t q = \frac{1}{2} \Delta q + K(x)q & \text{in } (0, \infty) \times \mathbb{R}^d, \\
        q(0, x) = q_0(x) & \text{in } \mathbb{R}^d.
    \end{cases}
\end{equation}
For $K,q_0:\R^d\to \R$ continuous and bounded, the Feynman-Kac formula states that if $q$ is a classical solution of \eqref{eq:FK}, bounded on $[0,T]\times \R^d$ for every $T>0$, with bounded derivatives up to second order on $x$ and first order on $t>0$, then it has the following probabilistic representation:
\begin{equation}\label{eq:FeynKac}
    q(t,x):=\mathbb{E}_x\Bigl[q_0(B_{t})e^{\int_0^t K(B_s)ds}\Bigl].
\end{equation}
For a more detailed exposition, see \cite[Section 2.2, Theorem 2.2]{Freidlin+1985}. Since it will be useful below, we note that the same representation formula can be used for \emph{time-independent} problems; in other words, if $u$ is a classical solution of 
\begin{equation*}
    0=\frac{1}{2}\Delta u+K(x)u,
\end{equation*}
under the same hypotheses on $K$ above, and $u$ has bounded derivatives up to the second order, then for every $t\geq 0$,  $u(x)=\mathbb{E}_x\bigl[q_0(B_{t})\mathrm{exp}\{\int_0^t K(B_s)ds\bigl]\}$.

\subsection{The branching time}\label{sec:BranchingTime}
In this section,  we present a few results regarding the first branching time $\tau_\emptyset$ defined in \eqref{eq:intro:BranchT}. Observe that by the definition of $\tau_\emptyset$, for every $x\in \R^d$ and $t>0$,
\begin{equation*}
    \proba{x}{}{\tau_\emptyset>t} = \expec{x}{}{e^{-\int_{0}^t g(B_s)ds}},
\end{equation*}
where $B_s$ is a standard Brownian motion started from $x \in \R^d$.

Our first task is to understand the behaviour of $\int_{0}^{t}g(B_{s})\, ds$ as $t\to\infty$. In order to do so, we begin with a preparatory lemma, which establishes an estimate for the $L_\infty$ norms of mean-zero solutions of Poisson's equation.
\begin{lemma}\label{lemma:LinfinityBound}
    There exists $C=C(d)>0$ such that the following holds. 
    Fix $f\in C_p(\R^d)$ with $\bar{f}=0$, and fix $\varphi\in C^2_p(\R^d)$ with $\bar{\varphi}=0$, and suppose that
    \begin{equation}\label{eq:PoissonPeriodic}
        -\frac{1}{2}\Delta \varphi = f \  \mathrm{in} \ \R^d.
    \end{equation}
    Then $\|\varphi\|_{\infty}+\|\nabla\varphi\|_{\infty}\leq C \|f\|_{\infty}$.
\end{lemma}

\begin{proof}
    We present an argument which is similar in spirit to an estimate appearing in  \cite[Proposition 7.6]{CCKW}, but which is appropriately adapted to our setting. Let $u:\R^d \times [0,\infty) \to \R$ solve $u_{t}-\Delta u=0$ on $(0, \infty)\times \R^{d}$, with initial condition $u(0,x)=f(x)$. In this case, we know that 
        \begin{equation*}
            u(t,x)=\int_{\R^{d}}(2\pi t)^{-d/2}e^{-\frac{|x-y|^2}{2t}}f(y)\, dy, 
        \end{equation*}
        and by the periodicity of $f$, we conclude that $x\mapsto u(t,x)$ is also $\Z^{d}$-periodic.
        We next introduce the resolvent operator $R_{1}: C_{p}(\R^{d})\rightarrow C_{p}(\R^{d})$ defined by
        \begin{equation*}
          R_{1}f(x):= \int_0^\infty e^{-t} u(t,x)\, dt.
        \end{equation*}
        
    We make the following claims:
    \begin{enumerate}
            \item \textbf{Claim 1:} The function $q:=R_{1}f$ is the unique solution of $q - \tfrac{1}{2}\Delta q = f$. Hence $R_{1}=\left(\text{Id}-\frac{1}{2}\Delta\right)^{-1}$. 
            \item \textbf{Claim 2:} There exists $C=C(d)>0$ such that, $\|\nabla R_{1} \|_{\mathrm{op}}\leq C$.
            \item \textbf{Claim 3:} The function $w(x):= \int_0^\infty u(t,x)\, dt$ solves \eqref{eq:PoissonPeriodic} and $\bar{w}=0$.
        \end{enumerate}
    We begin by proving the desired estimates on $\varphi$, assuming these three claims. First, we note that solutions to Poisson's equation \eqref{eq:PoissonPeriodic} are unique up to the addition of constants, and therefore, subject to condition that the solution is mean 0, we conclude from Claim 3 that $w\equiv \varphi$. Equipped with the representation formula of $\varphi$ given in Claim 3, we now recall that by the  exponential ergodicity of Brownian motion on the torus (see, \cite[Theorem 3.3.2]{MR2839402}), there exists $K, \rho>0$, depending on dimension $d$, such that $\sup_{x\in \R^d}|u(t,x)-\bar{f}|\leq K\| f\|_{\infty}e^{-\rho t}$. Since $\bar{f}=0$, this implies that 
    \begin{equation*}\label{eq:Linfty}
            \|\varphi \|_\infty = \sup_{x\in \mathbb{T}^d}
            \Big| \int_0^\infty u(t,x) dt \Big| \leq \int_0^\infty \sup_{x\in \TT^d}|u(t,x)| dt \leq K\|f\|_{\infty}\int_0^\infty e^{-\rho t} dt = C \|f\|_\infty.
        \end{equation*}

       To bound $\|\nabla \varphi\|_\infty$, we simply use that $\varphi =  \left(\mathrm{Id}-\frac{1}{2}\Delta\right)^{-1}\left(\mathrm{Id}-\frac{1}{2}\Delta\right)\varphi = R_{1}(\varphi + g)$, and from Claim 2, it follows that
        \begin{equation*}\label{eq:LinftyGrad}
            \|\nabla \varphi\|_\infty  = \|\nabla R_{1}( \varphi +g)\|_\infty \leq C\left(\|\varphi\|_{\infty} + g\|_\infty \right)\leq C\|g\|_{\infty}.
        \end{equation*}
        This combined with \eqref{eq:Linfty} yields the desired estimates on $\varphi$. 
        
        We now discuss the proofs of the claims. For Claim 1 , one can consider the modified quantity $R_{1}^{\delta}f(x):=\int_{\delta}^{\infty}e^{-t}u(t,x)\, dt$. Using the solution properties of $u$, one can check that $R_{1}^{\delta}f-\Delta (R_{1}^{\delta}f)=u(x,\delta)$. Sending $\delta\to 0$, and using the stability property of solutions of PDEs (for example, in the viscosity sense), we may conclude that Claim 1 holds. 
        
 For Claim 2, one can check that by differentiating underneath the integral, there exists $C=C(d)>0$ such that 
 \begin{equation*}
     |\nabla R_{1}f|\leq C\|f\|_{\infty}\int_{0}^{\infty}e^{-t}t^{-1/2}\, dt\leq C\|f\|_{\infty}, 
 \end{equation*}
which yields the claim. 

 Finally, for Claim 3, we first note that by the priorly mentioned exponential ergodicity estimate, the integral in the definition of $w$ is well-defined. A similar approximation argument as in the proof of Claim 1 shows that $-\Delta w=g$. The claim that $\overline{w}=0$ is a simple consequence of Fubini's theorem and the fact that $\bar{g}=0$. This completes the proofs of all three claims. 
         
\end{proof}

\begin{lemma}\label{lemma:ErgodicityBM}
    For all $\epsilon>0$ there exists a constant $C=C(d,\|g\|_{\infty})>0$ such that for all $t$ sufficiently large (depending on $d,\|g\|_{ \infty}$, and $\epsilon$),
 
    \begin{equation}\label{eq:ergo1}
        \sup_{x\in \R^d} \mathbb{P}_{x}\left[{\bigg|\frac{1}{t}\int_0^t g(B_s)ds - \bar{g} \bigg|>\epsilon}\right]\leq C(t\epsilon^2)^{-1}.
    \end{equation}
    
\end{lemma}
The above supremum in fact tends to zero exponentially quickly, assuming only that $g$ is bounded, but the above bound is sufficient for our purposes, and has a somewhat shorter proof.
\begin{proof}[Proof of Lemma~\ref{lemma:ErgodicityBM}]
    Let $\hat{g}:\R^d\to \R$ be the zero-mean $\Z^d$-periodic function defined by
    \begin{equation*}
        \hat{g}(x):=g(x)-\bar{g}.
    \end{equation*}
    From the solution theory of the Poisson equation in periodic media (see for example \cite[Chapter 
 6.3]{MR1881888}, there exists a solution $\varphi\in C^{2,\alpha}_p(\R^d)$ of the equation $(1/2)\Delta \varphi = \hat{g}$, which is moreover unique up to an additive constant. By It\^o's formula, 
    \begin{equation}
        \begin{split}
            \varphi(B_t)-\varphi(B_0)& = \frac{1}{2}\int_0^t \Delta \varphi(B_s)ds + \sum_{i=1}^d \int_0^t \frac{\partial \varphi}{\partial x_i}(B_s) dB^i_s\\
            & = \int_0^t \hat{g}(B_s) ds + \sum_{i=1}^d \int_0^t \frac{\partial \varphi}{\partial x_i}(B_s) dB^i_s.
        \end{split}
    \end{equation}
    Fix $\epsilon>0$ and $x\in \R^d$. Since $\varphi$ is bounded, for $t \ge 4\|\varphi\|_\infty/\epsilon$ we have 
    \begin{equation}\label{eq:ergo2}
        \begin{split}
            \mathbb{P}_{x}{}\left[{\frac{1}{t}\bigg|\int_0^t \hat{g}(B_s) ds\bigg|>\epsilon}\right] & = \mathbb{P}_{x}{}\left[{\frac{1}{t}\bigg|\varphi(B_t)-\varphi(B_0) - \sum_{i=1}^d \int_0^t \frac{\partial \varphi}{\partial x_i}(B_s) dB^i_s\bigg|>\epsilon}\right]\\
            & \leq \mathbb{P}_{x}{}\left[{\frac{1}{t} \bigg|\sum_{i=1}^d \int_0^t \frac{\partial \varphi}{\partial x_i}(B_s) dB^i_s\bigg|>\frac{\epsilon}{2}}\right]\\
            & \leq \mathbb{P}_{x}{}\left[{\sup_{\ell\in [0,t]}\bigg| \sum_{i=1}^d \int_0^\ell \frac{\partial \varphi}{\partial x_i}(B_s) dB^i_s\bigg|>\frac{t\epsilon}{2}}\right].
        \end{split}
    \end{equation} 
    Viewed as a function of $\ell$, the inner sum in the final term of the estimate is a continuous martingale with finite second moment. Thus, by Doob's maximal inequality and It\^o's isometry,
    
 \begin{equation}\label{eq:Doobsmax1}
        \begin{split}
            \mathbb{P}_{x}{}\left[{\frac{1}{t}\bigg|\int_0^t \hat{g}(B_s) ds\bigg|>\epsilon}\right] & \leq \frac{4}{t^2\epsilon^2}\expecdddd{x}{}{\Bigl| \sum_{i=1}^d \int_0^t \frac{\partial \varphi}{\partial x_i}(B_s) dB^i_s\Bigl|^2}\\
            & = \frac{4}{t^2\epsilon^2}\sum_{i=1}^d\expecdddd{x}{}{  \int_0^t \Bigl|\frac{\partial \varphi}{\partial x_i}(B_s) \Bigl|^2 ds}\\
            & \leq \frac{4td\|\nabla\varphi\|_{\infty}^2}{t^2\epsilon^2}= \frac{4d\|\nabla\varphi\|_{\infty}^2}{t\epsilon^2}.
        \end{split}
    \end{equation}
    By Lemma \ref{lemma:LinfinityBound}, there exists a constant $C>0$ such that $\|\nabla\varphi\|_{\infty}\leq C\|\hat{g}\|_{\infty} \leq 2\|g\|_{\infty}$. Plugging this back into \eqref{eq:Doobsmax1} completes the proof.
\end{proof}

As a consequence of this result, we obtain uniform bounds for the tail of $\tau_\emptyset$.

\begin{lemma}\label{lemma:tailBranchingTime}
There exist positive constants $\theta$ and $\Theta$, such that, for every $t\geq 0$,
    \begin{equation*}
        \theta^{-1}e^{-t\Theta}\leq \inf_{x\in \TT^d}\proba{x}{}{\tau_\emptyset>t}\leq \sup_{x\in \TT^d}\proba{x}{}{\tau_\emptyset>t} \leq \theta e^{-t\Theta}.
    \end{equation*}
\end{lemma}
\begin{proof}
     Recall that $\proba{x}{}{\tau_\emptyset > t}=\expec{x}{}{e^{-\int_0^tg(B_s)ds }}$. On the other hand, by \cite[Theorem 4.11.1(vi)]{MR1326606}, there exists a strictly positive function $\varphi \in C^{2,\alpha}_p(\R^d)$, and $\Gamma\in \R$ such that $\tfrac{1}{2}\Delta \varphi - g \varphi = \Gamma\varphi$ in $\R^d$. 
     Clearly $-g-\Gamma$ is bounded from above, so by the Feynman-Kac formula \eqref{eq:FeynKac}, 
     \begin{equation}\label{eq:FK0.1}
         \varphi(x)= \expecdd{x}{}{\varphi(B_t)e^{-\Gamma t-\int_0^tg(B_s)ds }}\quad\text{for every $t\geq 0$}.
     \end{equation}
     Letting $\theta= \max_{x \in \TT^d}\varphi(x)/\min_{x \in \TT^d}\varphi(x)$, we observe that
     \begin{equation*}
         \theta^{-1}e^{\Gamma t} \leq \inf_{x\in \TT^d} \expecdd{x}{}{e^{-\int_0^tg(B_s)ds }} \leq \sup_{x\in \TT^d} \expecdd{x}{}{e^{-\int_0^tg(B_s)ds }} \leq \theta e^{\Gamma t}.
     \end{equation*}
     To conclude, we must prove that $\Gamma<0$. Suppose for the purposes of contradiction that $\Gamma\geq 0$, and fix $\epsilon \in (0,\bar{g})$. Then by \eqref{eq:FK0.1},
     \begin{equation*}
         \begin{split}
             \varphi(x) & \leq  \expecdd{x}{}{\varphi(B_t)e^{-\int_0^tg(B_s)ds }}\\
             & \leq \mathbb{E}_{x}\left[{}{\varphi(B_t)e^{-\int_0^tg(B_s)ds }\mathbbm{1}_{\{ \int_0^t g(B_s)ds>t \bar{g}- t\epsilon \}} }\right] +\left(\max_{y\in \TT^d} \varphi(y)\right)\mathbb{P}_{x}\left[{\int_0^t g(B_s)ds\leq t \bar{g}- t\epsilon }\right] \\
             & \leq \left(\max_{y\in \TT^d} \varphi(y)\right)\left[e^{-t \bar{g}+ t\epsilon }+ \mathbb{P}_{x}{}\Big[{\int_0^t g(B_s)ds\leq t \bar{g}- t\epsilon }\Big]\right].
         \end{split}
     \end{equation*}
     By Lemma \ref{lemma:ErgodicityBM}, the right-hand side of the expression above goes to $0$ as $t\to\infty$, while the left-hand side is strictly positive; this is a contradiction, and hence we must have that $\Gamma<0$. 
\end{proof}

\subsection{The eigenvalue problem and $\cs$}\label{section:FKPPEigenvalue}

For each $e\in S^{d-1}$, the front speed $c^{*}(e)$ of the $g$-BBM has a variational formulation which can be expressed in terms of a family of eigenvalue problems for an associated linear differential operator. More precisely, for each fixed $\lambda\in \R$ and $e\in S^{d-1}$, we seek $\mu\in \R$ and a $\Z^d$-periodic function $\psi \in C_p^{2,\alpha}(\R^d)$ with $\psi \not\equiv 0$ satisfying 
\begin{equation}\label{eq:eigenPrelim}
         \frac{1}{2}\Delta \psi + \lambda e \cdot \nabla \psi +\parentesis{\frac{1}{2}\lambda^2 + g} \psi = \mu \psi \ \ \text{ in } \R^d. 
\end{equation}
The parameter $\lambda$ arises naturally in looking for exponential-type solutions of the linearization of~\eqref{eq:F-KPPintro}. Notice that the pair $(e,\lambda)$ produces the same eigenvalue problem as the pair $(-e,-\lambda)$. Due to this symmetry, most of the time it will be sufficient to study properties of the eigenvalue problem \eqref{eq:eigenPrelim} for pairs $(e,\lambda)$ with $e\in S^{d-1}$ and $\lambda\geq 0$.

A key quantity in our analysis is the \textit{principal eigenvalue} of \eqref{eq:eigenPrelim}, which is the function $\gamma:S^{d-1}\times \R \to \R$ defined by 
\begin{equation}\label{e.gadef}
    \gamma(e,\lambda)=\gamma(e,\lambda;g):=\sup\{\text{$\mu\in \R:$ there exists  $\psi \in C_p^{2,\alpha}(\R^d),\psi \not \equiv 0$ s.t.\ $\mu$ and $\psi$ satisfy \eqref{eq:eigenPrelim}}\}.
\end{equation} 
The quantity $\gamma(e,\lambda)$ plays a crucial role in characterizing the long-term behavior of the $g$-BBM; the following proposition collects a few of its properties. 
\begin{proposition}\cite[Theorem 4.11.1, parts (vi) and (viii)]{MR1326606}\label{prop:PrincipalEigenEx}
    Let $g\in C^{0,\alpha}_p(\R^d)$. Then $\gamma(e,\lambda)\in \R$ and there exists a strictly positive function $\psi(\cdot\ ;e, \lambda)\in C_p^{2,\alpha}(\R^d)$ such that $\gamma(e,\lambda)$ and $\psi(\cdot\ ;e, \lambda)$ satisfy \eqref{eq:eigenPrelim}, and such that $\int_{\TT^d} \psi(x;e, \lambda) dx=1$.
\end{proposition}

\begin{remark}
    To help make the connection with \cite{MR1326606}, we note that the linear operator appearing in the left-hand side of \eqref{eq:eigenPrelim} plays the role of $L$ in \cite[Theorem 4.11.1]{MR1326606}, and $\gamma(e,\lambda)$ plays the role of $\lambda_0$ in that theorem.
\end{remark}

    For fixed $e \in S^{d-1}$ and $\lambda\in \R$, fix a strictly positive function $\psi(\cdot\ ;e, \lambda)$ as in the proposition above. Since $\psi(\cdot\ ;e, \lambda)$ is $\Z^d$-periodic and twice differentiable, it achieves its minimum and maximum. In particular, 
    \begin{equation}\label{eq:BoundEigenfunction}
        \inf_{x\in \R^d}\psi(x\ ;e, \lambda) > 0 \ \mathrm{and} \ \sup_{x\in \R^d}\psi(x\ ;e, \lambda) < +\infty.  
    \end{equation}
    Furthermore, for the same strictly positive eigenfunction $\psi(\cdot\ ;e, \lambda)$, 
    \begin{equation}\label{eq:BoundEigenfunctionQuotient}
        \inf_{x,y\in \R^d}\frac{\psi(x ;e, \lambda)}{\psi(y ;e, \lambda)} \geq \frac{\inf_{x\in \R^d}\psi(x ;e, \lambda)}{\sup_{y\in \R^d}\psi(y ;e, \lambda)} > 0 \ \mathrm{and} \ \sup_{x,y\in \R^d}\frac{\psi(x ;e, \lambda)}{\psi(y ;e, \lambda)} \leq \frac{\sup_{x\in \R^d}\psi(x ;e, \lambda)}{\inf_{y\in \R^d}\psi(y ;e, \lambda)} < +\infty.  
    \end{equation}

We will express the front speed $\cs(e)$ of the $g$-BBM in terms of the principle eigevalue $\gamma:S^{d-1}\times \R\rightarrow \R$, defined by \eqref{e.gadef}. We first recall the following properties of $\gamma$. 

\begin{proposition}\label{Prop:PropertiesGamma}
    Fix $\alpha > 0$ and let $g\in C^{0,\alpha}_p(\R^d)$ be such that $g \ge 0$ and $g\not\equiv 0$. Then $\gamma(e,\lambda)=\gamma(e,\lambda;g)$ satisfies the following properties.
    \begin{enumerate}[(i)]
        \item The function $\gs(e,\cdot\ )$ belongs to $C^2(\R)$, is strictly convex, and for $\lambda\in \R$,
        \begin{equation}\label{eq:BoundGamma}
             \gs(e,\lambda) \in \left[\Big(\min_{x\in \R^d}g(x)\Big)+\frac{\lambda^2}{2}, \Big(\max_{x\in \R^d}g(x)\Big)+\frac{\lambda^2}{2}\right].
        \end{equation}
        \item The function $\gs(e,\cdot\ )$ is strictly positive and satisfies 
        \begin{equation}\label{eq:limGamma}
            \lim_{\lambda \rightarrow 0^+ }\frac{\gs(e,\lambda)}{\lambda}=+ \infty,\lim_{\lambda \rightarrow + \infty }\frac{\gs(e,\lambda)}{\lambda}=+\infty, \lim_{\lambda \rightarrow 0^- }\frac{\gs(e,\lambda)}{\lambda}=- \infty \text{ and } \lim_{\lambda \rightarrow - \infty }\frac{\gs(e,\lambda)}{\lambda}=-\infty.
        \end{equation}
        \item For all $e\in S^{d-1}$, there exists a unique $\lambda_e\in (0,\infty)$ such that 
        \begin{equation*}
            \inf_{\lambda>0} \frac{\gs(e,\lambda)}{\lambda}=\frac{\gs(e,\lambda_e)}{\lambda_e}>0.
        \end{equation*}
    \end{enumerate}
\end{proposition}

\begin{proof}[Proof of Proposition \ref{Prop:PropertiesGamma}(i)]
        The fact that $\gamma(e,\cdot)$ belongs to $C^2(\R)$ and is strictly convex is a direct consequence of \cite[Theorem 8.2.10 (i) (iii)]{MR1326606}. The argument in \cite[Lemma 2.1]{MR4596348} verbatim shows~\eqref{eq:BoundGamma}.
\end{proof}
\begin{proof}[Proof of Proposition \ref{Prop:PropertiesGamma}(ii)]
         By \eqref{eq:BoundGamma}, since $g\geq 0$, to prove that $\gamma(e,\cdot)$ is strictly positive it suffices to show that $\gamma(e,0)>0$. Let us prove the stronger assertion that $\gamma(e,0)\geq \bar{g}$. Suppose for the purposes of contradiction that $\gamma(e,0)<\bar{g}$. In particular, since $\bar{g}>0$, there is $\epsilon\in (0,1)$ such that $\gamma(e,0)<(1-\epsilon)\bar{g}$. By Proposition~\ref{prop:PrincipalEigenEx}, there exists a strictly positive function $\psi(\cdot)=\psi(\cdot\ ;e,0)\in C^{2,\alpha}_p(\R^d)$ such that $\tfrac{1}{2}\Delta \psi + g \psi = \gamma(e,0)\psi$. By the Feynman-Kac formula \eqref{eq:FeynKac}, since $\gamma(e,0)<(1-\epsilon) \bar{g}$, and by Lemma \ref{lemma:ErgodicityBM}, for every $t>0$ sufficiently large and $x\in \R^d$,
        \begin{equation*}
            \begin{split}
                \psi(x)&= \expecdd{x}{}{\psi(B_t)e^{-t\gamma(e,0)+\int_0^tg(B_s)ds }}\\ 
                & \geq \expecdd{x}{}{\psi(B_t)e^{-t(1-\epsilon)\bar{g}+\int_0^tg(B_s)ds }}\\
                &\geq \inf_{y\in \R^d}\psi(y)\cdot e^{t\epsilon\bar{g}/2}\cdot \expecddd{x}{}{e^{-t(1-\epsilon/2)\bar{g} + \int_0^tg(B_s)ds} ; t^{-1} \int_0^t g(B_s)ds >  (1-\epsilon/2)\bar{g}} \\
                &\geq \inf_{y\in \R^d}\psi(y)\cdot e^{t\epsilon\bar{g}/2}\cdot \mathbb{P}_{x}\Bigl[{t^{-1} \int_0^t g(B_s)ds >  (1-\epsilon/2)\bar{g}}\Bigl]\\
                & \geq \inf_{y\in \R^d}\psi(y)\cdot e^{t\epsilon\bar{g}/2}\cdot\bigl(1-Ct^{-1} (\epsilon/2\overline{g})^{-2}\bigl).
            \end{split}
        \end{equation*}
        The right-hand side of the expression above goes to $+\infty$ as $t\uparrow \infty$, but the left-hand side is bounded from above, which gives a contradiction. Thus, we must have $\gamma(e,0)\geq \bar{g}$.
        The limits in \eqref{eq:limGamma} follow directly from the fact that $\gamma(e,\cdot)$ is strictly positive and from \eqref{eq:BoundGamma}.
     \end{proof}
     \begin{proof}[Proof of Proposition \ref{Prop:PropertiesGamma}(iii)]
     This follows from the fact that $\gamma(e,\cdot)$ is strictly positive and strictly convex, and by the limits in~\eqref{eq:limGamma}. \qedhere
\end{proof}

\begin{remark}\label{r:derivativecs}
    Observe that $\tfrac{\gamma(e,\lambda)}{\lambda}$ denotes the slope of the line passing through the origin and $\gamma(e,\lambda)$. Given the convexity and differentiability of $\gamma$, this property, together with (iii), implies that ${\tfrac{\gamma(e,\lambda_e)}{\lambda_e} = \partial_\lambda \gamma(e,\lambda_e)}$.
\end{remark}

In light of the previous proposition, for $g$ as in the proposition we can define a function $\cs:S^{d-1}\rightarrow \R$, by setting, for every $e\in S^{d-1}$, 
\begin{equation}\label{eq:prelimCriticalSpeed}
    \cs(e)=\cs(e;g):= \min_{\lambda>0} \frac{\gs(e,\lambda)}{\lambda}=\frac{\gs(e,\lambda_e)}{\lambda_e}.
\end{equation}
Observe that $c^{*}(e)$ only depends on the principal eigenvalue $\gamma(e,\lambda)$ for the linearized eigenvalue problem \eqref{eq:eigenPrelim}. The quantity $c^{*}(e)$ defined in \eqref{eq:prelimCriticalSpeed} has several equivalent formulations (see, for example, \cite[Remark 1.2]{BHN}). One of those characterizations is that $c^*(e)$ is the minimal speed for which there exists a pulsating traveling front solution (in the direction $-e$) associated to the Fisher-KPP reaction-diffusion equation
\begin{equation*}
    u_{t}-\Delta u=f(x,u).
\end{equation*}
We now use this fact in order to justify that $\min_{S^{d-1}} c^{*}(e)>0.$ 

Let $b:=\|g\|_{\infty}$ and let $f:\R^d\times \R\to \R$ be given by 
\begin{equation}\label{eq:AuxNonlinearity}
    f(x,r):=r(g(x)-br).
\end{equation}
We will use auxiliary results from \cite{MR2155900}, who study properties of $c^{*}(e; f)$ for $f$ of this form. 
Observe that $c^{*}(e; f)$ is defined according to the principal eigenvalue for a linearized version of the above reaction-diffusion equation. In other words, we seek $\mu\in \R$ and a function $\psi \in C_p^{2,\alpha}(\R^d)$ with $\psi \not\equiv 0$ such that, setting
$f_{r}(x,0):=\lim_{r\to 0^+}r^{-1}f(x,r)=g(x),$ we have 
\begin{equation}\label{eq:eigenPrelimAux}
         \frac{1}{2}\Delta \psi + \lambda e \cdot \nabla \psi +\parentesis{\frac{1}{2}\lambda^2 + f_r(x,0)} \psi = \mu \psi \ \ \text{ in } \R^d.
\end{equation}
This shows that the eigenvalue problem \eqref{eq:eigenPrelimAux} is equivalent to \eqref{eq:eigenPrelim}, which implies that $c^{*}(e; g)=c^{*}(e; f).$ 
 
We now recall a result from \cite{MR2155900}, which we have stated in a simplified manner for the setting we consider here, which guarantees that $\cs=\cs(\cdot; g)$ is strictly positive. 
\begin{lemma}\cite[Corollary 1.4]{MR2155900}\label{l.c*low}
  Let $g\in C_p^{1,\alpha}(\R^d)$ such that $g\geq 0$ and $g\not \equiv 0$.  Let $f$ be defined by \eqref{eq:AuxNonlinearity}. Then,
    \begin{equation}\label{e.c*low}
        \min_{e\in S^{d-1}}\cs(e;g)=\min_{e\in S^{d-1}}\cs(e;f)\geq \sqrt{2\bar{g}}>0.
    \end{equation}
\end{lemma}

The hypotheses on the reaction term in  \cite{MR2155900} are clearly satisfied by the specific choice of $f$ in \eqref{eq:AuxNonlinearity}, and since $f_{r}(x,0)=g(x)$, it follows that $\int_{[0,1]^{d}} g(x)\, dx\geq \bar{g}>0.$ One of the conclusions of \cite[Corollary 1.4]{MR2155900} is the ``hypothesis for conservation,'' which guarantees that $c^{*}(e; f)$ is well-defined, and agrees with the formula for $c^{*}$ given in \eqref{eq:prelimCriticalSpeed}.

While this lemma essentially finishes the claim that $\min_{e\in S^{d-1}}c^{*}(e,g)>0$, we now use an approximation argument to remove the additional regularity assumption on the branching rate function $g$. We also present an upper bound on $c^{*}(\cdot; g)$ which we will use in the sequel.  

\begin{proposition}\label{p.c*pos}
Let $g\in C^{0,\alpha}_p(\R^d)$ be such that $g \ge 0$ and $g\not\equiv 0$. Then the function $\cs=\cs(\cdot;g):S^{d-1}\rightarrow [0,\infty)$ is bounded away from zero and bounded from above.
\end{proposition}

\begin{proof}
    To prove that $\cs$ is bounded from above, for $v\in \R^d$ let $\Gamma(v):=\gs(v/|v|,|v|)$, where $\gamma$ is the principal eigenvalue defined in \eqref{e.gadef}. By \cite[Chapter 8, Theorem 2.10]{MR1326606}, $\Gamma \in C^2(\R^d)$. In particular, for $\lambda>0$, $\gs(\cdot\ ,\lambda):S^{d-1}\rightarrow \R$ is continuous. By \eqref{eq:prelimCriticalSpeed}, 
    \begin{equation*}
        \cs(e)= \min_{\lambda>0}\frac{\gs(e,\lambda)}{\lambda} \leq \frac{\gs(e,1)}{1} \leq \max_{e'\in S^{d-1}}\gs(e',1)<\infty,
    \end{equation*}
    which shows that $\cs$ is bounded from above. 

    For the lower bound, we first recall that by a standard mollification argument, we can construct a non-negative function $g_\text{aux}\in C^2_p(\R^d)$ such that $g\geq g_{\text{aux}}$ in $\R^d$ and $g_\text{aux}\not \equiv 0$. We claim that $\cs(\cdot\ ;g)\geq \cs(\cdot\ ;g_\text{aux})$. In fact, by \eqref{eq:prelimCriticalSpeed}, this will follow if we verify that for every $(e,\lambda) \in S^{d-1}\times \R$, $\gamma(e,\lambda;g)\geq \gamma(e,\lambda;g_{\text{aux}})$. 
    
    To prove this claim, we recall a 
  min-max formula for the principal eigenvalue of the eigenvalue problem \eqref{eq:eigenPrelim}. Specifically, by  \cite[Proposition 2.9]{MR2491804}, for $e\in S^{d-1}$ and $\lambda\in \R$,
    \begin{equation}
        \begin{split}
            \gamma(e,\lambda; g) &= \max_{ \substack{\phi>0,\\ \phi\in C^2_p(\R^d) } } \min_{x\in \R^d} \frac{ \frac{1}{2}\Delta \phi + \lambda e \cdot \nabla \phi +\parentesis{\frac{1}{2}\lambda^2 + g} \phi}{\phi(x)}\\
            &\geq \max_{ \substack{\phi>0,\\ \phi\in C^2_p(\R^d) } } \min_{x\in \R^d} \frac{ \frac{1}{2}\Delta \phi + \lambda e \cdot \nabla \phi +\parentesis{\frac{1}{2}\lambda^2 + g_\text{aux}} \phi}{\phi(x)}\\
            &=  \gamma(e,\lambda; g_\text{aux}).
        \end{split}
    \end{equation}
    Hence, $\cs(\cdot\ ;g)\geq \cs(\cdot\ ;g_\text{aux})$ in $S^{d-1}$. 

    We conclude by applying Lemma \ref{l.c*low} to $f_{\text{aux}}(x,u)$ given by \eqref{eq:AuxNonlinearity} with $g_{\text{aux}}$. Then, since $g_{\text{aux}}$ is sufficiently regular, it follows from Lemma \ref{l.c*low} and the prior claim that  
    \begin{equation}
        \min_{e\in S^{d-1}}\cs(e;g) \geq \min_{e\in S^{d-1}}\cs(e;g_\text{aux}) \geq \sqrt{2 \bar{g}_\text{aux}}>0. 
    \end{equation}
\end{proof}

\begin{lemma}\label{lemma:UniformBoundsLambda}
    Let $\lambda_e>0$  be such that $\cs(e)=\frac{\gs(e,\lambda_e)}{\lambda_e}$ as in 
     \eqref{eq:prelimCriticalSpeed} and let $\psi(\cdot\ ;e,\lambda_e)$ the eigenfunction associated to the principal eigenvalue of the eigenvalue problem \eqref{eq:eigenPrelim},  normalized so that $\int_{\TT^d} \psi(x;e, \lambda) dx=1$. Then,
\begin{equation}\label{eq:BoundsLambdaQuotient}
        \begin{split}
            0<\inf_{e\in S^{d-1}} \lambda_e \ &\mathrm{and} \ \sup_{e\in S^{d-1}} \lambda_e < +\infty,\\
            \inf_{e\in S^{d-1}}\inf_{x\in \R^d}\psi(x ;e, \lambda_e) > 0 \ &\mathrm{and} \ \sup_{e\in S^{d-1}}\sup_{x\in \R^d}\psi(x ;e, \lambda_e) < +\infty.
        \end{split}
    \end{equation}
\end{lemma}

\begin{proof}
    From \cite[Chapter 8, Theorem 2.10]{MR1326606}, we know that $\gamma(e,\lambda)$ is differentiable in both of the variables $e$ and $\lambda$, and thus, by Proposition \ref{Prop:PropertiesGamma}(ii), $\inf_{(e,\lambda) \in S^{d-1}\times \R}\gamma(e,\lambda)>0$.  Thus, from \eqref{eq:prelimCriticalSpeed} and Proposition \ref{p.c*pos},
    \begin{equation*}
        \lambda_e = \frac{\gamma(e,\lambda_e)}{\cs(e)} \geq \frac{\inf_{(e,\lambda) \in S^{d-1}\times \R} \gamma(e,\lambda) }{\sup_{e\in S^{d-1}} \cs(e)} > 0.
    \end{equation*}
    For an upper bound on $\lambda_e$, as above and from Proposition \ref{eq:BoundGamma}(i),
    \begin{equation*}
        \lambda_e = \frac{\gamma(e,\lambda_e)}{\cs(e)} \geq \frac{(\min_{x\in \R^d} g(x)) + \frac{\lambda_e^2}{2}}{\sup_{e\in S^{d-1}} \cs(e)}.
    \end{equation*}
    Solving a quadratic equation, from this inequality we deduce that 
    \begin{equation*}
        \lambda_e \leq \sup_{e\in S^{d-1}} \cs(e) + \sqrt{\big(\sup_{e\in S^{d-1}} \cs(e)\big)^2 - 2\big(\min_{x\in \R^d} g(x)\big)}.
    \end{equation*}
    Since the previous bounds do not depend on the choice of $e\in S^{d-1}$, we obtain the uniform bounds for $\lambda_e$ in \eqref{eq:BoundsLambdaQuotient}.

    For the uniform bounds on the quotient, let us consider the auxiliary function $\phi:\R^d\times \R^d\to \R$ defined by $\phi(x,\zeta)=\psi(x;e,\lambda)$, where $\lambda>0$ and $e\in S^{d-1}$ are uniquely determined by the relation $\zeta= \lambda e$. We know from Proposition~\ref{prop:PrincipalEigenEx} that for fixed $\zeta\in \R^d$,  $\phi(\cdot,\zeta)$ is twice continuously differentiable, and, in the proof of \cite[Chapter 8, Theorem 2.10, eq. (2.8)]{MR1326606}, it is proved that for fixed $x\in \R^d$, $\phi(x,\cdot)$ is continuously differentiable. Thus, by continuity and periodicity, $\phi$ achieves its maximum and minimum in the compact set $[0,1]^d\times B(0,a)$, where $B(0,a)\subset \R^d$ denotes the closed ball centered at the origin with radius $a$. That is, there exist $(x_1, \zeta_1),(x_2, \zeta_2) \in [0,1]^d\times B(0,a)$ such that 
    \begin{equation*}
        \psi(x_1 ;e_1, \lambda_1) = \phi(x_1, \zeta_1) = \min_{(x,\zeta) \in [0,1]^d\times B(0,a)} \phi(x,\zeta),
    \end{equation*}
    and 
    \begin{equation*}
        \psi(x_2 ;e_2, \lambda_2) =  \phi(x_2, \zeta_2) = \max_{(x,\zeta) \in [0,1]^d\times B(0,a)} \phi(x,\zeta) ,
    \end{equation*}
    where, for $i=1,2$, $\lambda_i\geq 0$ and $e_i\in S^{d-1}$ are such that $\zeta_i=\lambda_i e_i$. We know that the functions $\psi(\cdot  ;e_i,\lambda_i)$ are bounded away from zero and bounded from above. Thus, we deduce that 
    \begin{equation*}
        \inf_{e\in S^{d-1}}\inf_{x\in \R^d}\psi(x ;e, \lambda_e) \geq \min_{(x,\zeta) \in [0,1]^d\times B(0,a)} \phi(x,\zeta) > 0,
    \end{equation*}
    and
     \begin{equation*}
     \sup_{e\in S^{d-1}}\sup_{x\in \R^d}\psi(x ;e, \lambda_e) \leq  \max_{(x,\zeta) \in [0,1]^d\times B(0,a)} \phi(x,\zeta) < +\infty.
    \end{equation*}
    This gives the second part of \eqref{eq:BoundsLambdaQuotient}, which finishes the proof.
\end{proof}

\begin{remark}\label{r.holder}In the case when $g$ is merely continuous (and not H\"older continuous), we note that, as shown in \cite[Theorem 4.4.3]{MR1326606}, the principle eigenvalue $g\mapsto \gamma(e,\lambda; g)$ is Lipschitz continuous in $g$, with respect to the $\sup$-norm. In other words, if $g_{n}\xrightarrow[n\to\infty]{}g$ uniformly, then $\sup_{\lambda>0}|\gamma(e, \lambda; g_{n})-\gamma(e, \lambda; g)|\xrightarrow[n\to\infty]{} 0$, in which case by \eqref{eq:prelimCriticalSpeed}, we may set $c^{*}(e;g):=\lim_{n\to\infty}c^{*}(e; g_{n})$. This is how one can define $c^{*}(e; g)$ in the case when $g$ is merely continuous, but not necessarily H\"older continuous. Equipped with a definition of $c^{*}(e; g)$ for every direction $e\in \mathbb{S}^{d}$, one can then define $\mathcal{W}$ for such $g$ as in \eqref{wulffintro}. \end{remark}

\subsection{Many-to-one lemma}\label{section:manytoone}

We introduce the many-to-one lemma associated with the $g$-BBM. This name reflects the property of equation \eqref{eq:manyTO}, where the probability of a functional depending on the paths of all living particles is expressed in terms of the path of a single particle.  Recall that we denoted by $\mathbb{P}_x$ the underlying abstract measure for the Brownian motion $(B_s)_{s\geq 0}$ started at $x\in \R^d$.
 
\begin{lemma}\cite[Lemma 1 and Section 4.1]{MR3606740}\label{lemma:ManyToOneLemma}
    Let $\cF^B_t=\sigma(B_s; s\in [0,t])$ and $F:C([0,t])\rightarrow [0,\infty)$ be a given function. If $F((B_s)_{s \in [0,t]})$ is $\cF^B_t$-measurable, then
    \begin{equation}\label{eq:manyTO}
        \mathbf{E}_{x}\left[ \sum_{v\in N_t} F\big((X_s(v))_{s \in [0,t]}\big)   \right]= \mathbb{E}_{x}{}\left[{e^{\int_0^t g(B_s) ds}F((B_s)_{s \in [0,t]})  }\right].
    \end{equation}
\end{lemma}

The term $ \exp(\int_0^tg(B_s)ds)$ on the right-hand side of \eqref{eq:manyTO} poses a problem when studying functionals of the form above. To eliminate it, we consider a directional change of measure. Fix $e\in S^{d-1}$ and $\lambda>0$. Let $\phi:\R^d\rightarrow \R^d$ be defined by 
\begin{equation}\label{e.phidef}
   \phi(x)=\phi(x;e,\lambda)=(\phi^1(x),...,\phi^d(x)):= \frac{\nabla \psi(x;e,\lambda)}{\psi(x;e,\lambda)} + \lambda e,
\end{equation}
where $\psi(\cdot \ ;e,\lambda)$ is the strictly positive principal eigenfunction introduced in \eqref{eq:eigenPrelim}.

Let $(\hat{\Omega}, \hat{\cF}, \{\hat{\cF_t}\}_{t\geq 0}, \bP^{e,\lambda})$ denote a filtered space, with $ \{\hat{\cF_t}\}_{t\geq 0}$ right-continuous and complete, and $(\hat{W}_t)_{t\geq 0}$ a $d$-dimensional Brownian motion adapted to $ \{\hat{\cF}_{t}\}_{t\geq 0}$. For $x\in \R^d$, let $(Y_t)_{t\geq 0}$ be a $\hat{\cF}_t$-adapted continuous process which is a strong solution of the SDE 
\begin{equation}\label{eq:SDETiling2.0}
    dY_t = \phi(Y_t;e,\lambda)dt + d\hat{W}_t,\, \, Y_0=x.
\end{equation}
Observe that since $\psi(\cdot; e, \lambda)\in C^{2,\alpha}_{p}(\R^{d})$ is strictly positive, $\phi$ is bounded and Lipschitz continuous. Therefore, the existence and uniqueness of such a process $(Y_{t})_{t\geq 0}$ is a consequence of the well-posedness for SDEs with Lipschitz continuous coefficients (see for example \cite[Theorem 8.3]{MR3497465}). For $x\in \R^d$, we denote by $\bP_x^{e,\lambda}$ the measure on $(\hat{\Omega},\hat{\cF})$ under which $\prob{x}{e,\lambda}{Y_0=x}=1$, and let $\bE_x^{e,\lambda}$ denote the expectation associated to $\bP_x^{e,\lambda}$.

The following lemma is the main result of this section. The change of measure appearing in the lemma is used in \cite{MR4162842} to obtain exact asymptotics of the transition kernel of branching diffusions in periodic environments. Our exposition closely follows that of a similar, one-dimensional version, appearing in \cite[Lemma 2.3]{MR4492971}.

\begin{lemma}\label{Lemma:DirectionalChangeOfMeasure}
Let $\cF^B_t=\sigma(B_s; s\in [0,t])$ and $F:C([0,t])\rightarrow [0,\infty)$ be a given function. If $F((B_s)_{s \in [0,t]})$ is $\cF^B_t$-measurable, then
    \begin{equation}\label{eq:changemeasure}
        \expecddd{x}{e,\lambda}{\frac{\psi(x;e,\lambda)}{\psi(Y_t;e,\lambda)}e^{-\lambda e\cdot (Y_t-Y_0)+t\gs(e,\lambda) }F((Y_s)_{s\in [0,t]})} = \expecddd{x}{}{e^{\int_0^t g(B_s)ds}F((B_s)_{s\in [0,t]})},
    \end{equation}
where $\gs(e,\lambda)$ is the principal eigenvalue introduced in Proposition~\ref{prop:PrincipalEigenEx}.
\end{lemma}

Before proceeding with the proof, we use Girsanov's theorem to construct an auxiliary solution of \eqref{eq:SDETiling2.0} on the space $(\Omega,\cF)$, which will be used in the proof of Lemma \ref{Lemma:DirectionalChangeOfMeasure}. For each $\omega\in \Omega$, let
\begin{equation}\label{WTilting}
    W_t=W_t(\omega):= B_t(\omega)- \int_0^t \phi(B_s(\omega)) ds.
\end{equation}
It is clear that $(W_t)_{t\geq 0}$ is $\cF_t$-adapted and has continuous trajectories. Let
\begin{equation}\label{eq:Zt1}
    Z_t := \exp\left(\sum_{k=1}^d\int_0^t \phi^k(B_s)\ dB^k_s - \frac{1}{2}\int_0^t |\phi(B_s)|^2 ds\right).
\end{equation}
Since $\phi$ is bounded, we have that for every $t>0$, 
\begin{equation*}
    \begin{split}
        \expecddd{x}{}{\exp\Big(\frac{1}{2}\int_0^t |\phi(B_s)|^2 ds\Big)}<\infty.
    \end{split}
\end{equation*}
Thus, by Novikov's theorem (see \cite[Corollary 3.5.13]{MR1121940}), $(Z_t)_{t\geq 0}$ is a martingale under $\bP_x$.  

Consider the completion of the natural filtration of $(B_t)_{t\geq 0}$. Let us denote this filtration by $(\cF^B_{t})_{t\geq 0}$, so $\cF^B_t$ is the completion of $\sigma(B_s;s\in [0,t])$. Since $\phi$ is continuous, one can show that there exists a measure $\bQ_x$ on $\cF_\infty^B:=\sigma(B_s;s\geq 0)$ such that, for every $A\in \cF^B_t$,
\begin{equation}\label{eq:measuretilting}
    \bQ_{x}[A]=\expec{x}{}{\mathbbm{1}_A Z_t}
\end{equation}
(see for example \cite[Corollary 3.5.2]{MR1121940}). It follows that if $G:C([0,t])\rightarrow [0,\infty)$ is such that $G((B_s)_{s\in [0,t]})$ is $\cF^B_t$-measurable, then
\begin{equation}\label{eq:functionalchange}
    \bQ_x[G((B_s)_{s\in [0,t]})] = \expec{x}{}{Z_tG((B_s)_{s\in [0,t]})}.
\end{equation}
Finally, by Girsanov's theorem, the process $(W_t)_{t\geq 0}$ defined by \eqref{WTilting} is a Brownian motion under the measure $\bQ_x$, with $\bQ_x[W_0=x]=1$. In particular, by rearranging \eqref{WTilting}, we observe that $(B_t)_{t\geq 0}$ is a weak solution of \eqref{eq:SDETiling2.0} under the measure $\bQ_x$. 

Now return to the abstract probability measure $\mathbb{P}^{e,\lambda}_x$ introduced right after \eqref{eq:SDETiling2.0}. Since $(B_t)_{t\geq 0}$ is a weak solution of \eqref{eq:SDETiling2.0} under the measure $\bQ_x$, and by the uniqueness of weak solutions for SDEs with Lipschitz coefficients (see \cite[Theorem 8.5]{MR3497465}), if $G:C([0,t])\rightarrow [0,\infty)$ is such that $G((B_s)_{s\in [0,t]})$ is $\cF^B_t$-measurable and lies in $L^1(\mathbb{P}_{x}^{e,\lambda})$, then $G((B_s)_{s\in [0,t]}) \in L^1(\bQ_x)$ and 
\begin{equation}\label{eq:WeakSolUniq}
    \expecdd{x}{e,\lambda}{G((Y_s)_{s\in [0,t]})} = \bQ_x[G((B_s)_{s\in [0,t]})].
\end{equation}

With this in mind, we proceed with the proof of Lemma \ref{Lemma:DirectionalChangeOfMeasure}. We will prove the identity \eqref{eq:changemeasure} for $(B_t)_{t\geq 0}$ under the auxiliary measure $\bQ$, and then use \eqref{eq:WeakSolUniq} to obtain the result for the process $(Y_t)_{t\geq 0}$ in \eqref{eq:SDETiling2.0} under the abstract measure $\mathbb{P}^{e,\lambda}$.

\begin{proof}[\textbf{Proof of Lemma \ref{Lemma:DirectionalChangeOfMeasure}}]
    Recall that $\psi(\cdot\ ;e, \lambda)$ denotes the eigenfunction associated with the principal eigenvalue of the eigenvalue problem \eqref{eq:eigenPrelim} introduced in Proposition~\ref{prop:PrincipalEigenEx}. Since $\lambda >0$ and $e\in S^{d-1}$ are fixed, let us write $\psi(x):= \psi(x;e,\lambda)$, $\gs(\lambda):=\gs(e,\lambda)$, and let $\phi_{[s]}=\phi(B_s;e,\lambda)$. We claim that under the measure $\bP_x$,
    \begin{equation}\label{eq:claim}
        \sum_{k=1}^d\int_0^t \phi^k_{[s]}\ dB^k_s - \frac{1}{2}\int_0^t |\phi_{[s]}|^2 ds = \log\frac{\psi(B_t)}{\psi(B_0)} + \lambda e\cdot(B_t-B_0) - t \gs(\lambda) + \int_0^t g(B_s) \ ds .
    \end{equation}
    Let us write $\psi_{[s]}=\psi(B_s)$, $\nabla\psi_{[s]}=\nabla\psi(B_s)$, $\partial_k \psi_{[s]}=(\tfrac{\partial \psi}{\partial x_k})(B_s)$ and $\Delta\psi_{[s]}=\Delta \psi(B_s)$. Since $(B_t)_{t\geq 0}$ is a Brownian motion under the measure $\bP_x$, we apply It\^o's formula to get 
    \begin{equation}\label{eq:ItoFormulaLog}
        \begin{split}
            \log\frac{\psi(B_t)}{\psi(B_0)}=\log\psi(B_t)- \log\psi(B_0) = \sum_{k=1}^d\int_0^t \frac{\partial_k \psi_{[s]}}{\psi_{[s]}}  \ dB_s^k + \frac{1}{2}\int_0^t \bigg(\frac{\Delta \psi_{[s]}}{\psi_{[s]}}-\frac{|\nabla \psi_{[s]} |^2}{\psi^2_{[s]}}\bigg)\ ds.
        \end{split}
    \end{equation}
    On the other hand, from the definition of $\phi_{[s]}$, observe that
    \begin{equation*}
        \begin{split}
            \sum_{k=1}^d\int_0^t \phi^k_{[s]}\ dB^k_s - \frac{1}{2}\int_0^t |\phi_{[s]}|^2 ds & = \sum_{k=1}^d\int_0^t \frac{\partial_k \psi_{[s]}}{\psi_{[s]}}dB_s^k + \lambda e\cdot (B_t-B_0)\\ 
            & \ \ \ \ \ \ \ \ \ \ \ \ -\frac{1}{2}\int_0^t \frac{|\nabla \psi_{[s]} |^2}{\psi^2_{[s]}} \ ds  -\frac{1}{2}\int_0^t \bigg({2\lambda \frac{e\cdot \nabla \psi_{[s]}}{\psi_{[s]}}+\lambda^2 }\bigg) ds.
        \end{split}
    \end{equation*}
    Combining this with \eqref{eq:ItoFormulaLog}, we get
    \begin{equation*}
        \begin{split}
             \sum_{k=1}^d\int_0^t \phi^k_{[s]}\ dB^k_s - \frac{1}{2}\int_0^t |\phi_{[s]}|^2 ds & = \log\frac{\psi(B_t)}{\psi(B_0)} + \lambda e\cdot(B_t-B_0)\\
             & \ \ \ \ \ \ \ \ \ \ \ \ \ \ \ \ -\frac{1}{2}\int_0^t \bigg(2\lambda \frac{e\cdot \nabla \psi_{[s]}}{\psi_{[s]}}+\lambda^2 + \frac{\Delta \psi_{[s]}}{\psi_{[s]}} \bigg) ds. 
        \end{split}
    \end{equation*}
    Recall that $\psi$ solves the eigenvalue problem \eqref{eq:eigenPrelim} with $\mu=\lambda$. By using the expression of the eigenvalue problem to rewrite the final integral in the expression above, we obtain \eqref{eq:claim}; this finishes the proof of the claim.

    Now we plug \eqref{eq:claim} into \eqref{eq:Zt1}, the definition of $Z_t$, to obtain
    \begin{equation*}
        Z_t = \frac{\psi(B_t)}{\psi(x)}\exp\Big(\lambda e\cdot(B_t-B_0) - t\gs(\lambda) +\int_0^t g(B_s) ds \Big).
    \end{equation*}
    Let $G$ be the non-negative $\cF^B_t$-measurable function given by
    \begin{equation*}
        G((B_s)_{s\in [0,t]}):= \frac{\psi(x)}{\psi(B_t)}\exp\bigl(-\lambda e\cdot (B_t-B_0)+t\gs(\lambda) \bigl)F((B_s)_{s\in [0,t]}).
    \end{equation*}  
    By plugging this into \eqref{eq:functionalchange}, we obtain
    \begin{equation}\label{eq:aux}
        \begin{split}
            \bQ_x\Big[\frac{\psi(x)}{\psi(B_t)}e^{-\lambda e\cdot (B_t-B_0)+t\gs(e,\lambda) }F((B_s)_{s\in [0,t]})\Big] &= \expecdd{x}{}{Z_tG((B_s)_{s\in [0,t]})}\\
            &= \mathbb{E}_{x}\left[\exp\Big(\int_0^t g(B_s) ds \Big)F((B_s)_{s\in [0,t]})\right],
        \end{split}
    \end{equation}
    which shows \eqref{eq:changemeasure} for the solution of \eqref{eq:SDETiling2.0} under the measure $\bQ$. Finally, by \eqref{eq:WeakSolUniq}, we conclude that \eqref{eq:changemeasure} holds for the solution $(Y_t)_{t\geq 0}$ of \eqref{eq:SDETiling2.0} under the abstract measure $\mathbb{P}^{e,\lambda}_x$.
\end{proof}

When applying Lemma \ref{Lemma:DirectionalChangeOfMeasure}, we will specify the choice of $e\in S^{d-1}$ by saying that we apply Lemma \ref{Lemma:DirectionalChangeOfMeasure} {\em in the direction} $e$. 

\subsection{Interpolation Results}\label{sec:interpolation}
Throughout this work, we perform several arguments which rely on the $g$-BBM evaluated at a set of discrete times. In this subsection, we present a general ``interpolation lemma'' which allows us to pass information from discrete times to continuous times. 

For $E\subset \R^d$ and $\epsilon>0$, let $E_\epsilon:=\cup_{x\in E}B(x,\epsilon)$. Given two sets $E,D\subset \R^d$, the {\em Hausdorff distance} between $E$ and $D$ is given by
\begin{equation*}
    d_\mathrm{H}(E,D):=\inf\{ \epsilon>0 :\ E\subset D_\epsilon \text{ and } D\subset E_\epsilon \}.
\end{equation*}
We now prove the following result, which we later refer to as the interpolation lemma.

\begin{lemma}\label{lemma:interpolation}
    For $t \geq 0$, let $\mathcal{X}_t:=\{X_t(v):v\in \cN_t\}\subset \R^d$. Fix $T>0$ and $\kappa>0$. For every $x\in \R^d$, 
    \begin{equation}\label{eq:interpolation2}
        \probaaa{x}{}{\exists N\in \N \text{ s.t. } \forall t\geq NT, \sup_{\ell \in [t,t+T] } d_\mathrm{H}(\mathcal{X}_t,\mathcal{X}_\ell)\leq \kappa t }=1
    \end{equation}
\end{lemma}
\begin{proof}
    Let $b=\|g\|_{\infty}$ and $\kappa':=\kappa/8$. For $n \in \N$, let $A_{n}\subset \Omega$ be defined by
    \begin{equation*}
        A_n:=\Bigl\{\sup_{\ell \in [nT,(n+1)T]} d_\mathrm{H}(\mathcal{X}_{nT},\mathcal{X}_\ell)> \kappa'nT \Bigl \}.
    \end{equation*}
Recall that for $v\in \mathcal{N}_t$ and $s\leq t$, we denote $X_s(v)$ as the position of the living ancestor of $v$ at time $s$. If $A_n$ occurs, there is a particle $v$ in $\cN_{(n+1)T}$ such that $\sup_{\ell \in [nT,(n+1)T]} |X_\ell(v)-X_{nT}(v)|> \kappa'nT $. Hence, by the Markov inequality, the many-to-one lemma (Lemma \ref{lemma:ManyToOneLemma}), and the Markov property for Brownian motion,
    \begin{equation}\label{eq:interpolatio1}
        \begin{split}
            \proba{x}{}{A_n} &\leq \probaaa{x}{}{\exists v\in \cN_{(n+1)T}: \sup_{\ell \in [nT,(n+1)T]} |X_\ell(v)-X_{nT}(v)|> \kappa'nT }\\
            & = \expecddd{x}{}{e^{\int_0^{(n+1)T}g(B_s)ds}\mathbbm{1}_{\{\sup_{\ell \in [nT,(n+1)T]} |B_\ell-B_{nT}|> \kappa'nT\}}  }\\
            & \leq e^{(n+1)Tb}\mathbb{P}_{0}\left[{\sup_{\ell \in [0,T]} |B_\ell|> \kappa'nT }\right]=  e^{(n+1)Tb}\mathbb{P}_{0}\left[{\sup_{\ell \in [0,1]} |B_\ell|> \kappa'nT^{1/2} }\right].
        \end{split}
    \end{equation}
     For $i=1, \ldots, d$, let $B^i_{t}$ be the $i$'th component of the Brownian motion $B_{t}$. Note that since $\sum_{i=1}^d \sup_{\ell\in [0,1]}|B_\ell^i| \geq \sup_{\ell\in [0,1]}|B_\ell|$, we have
     \begin{equation*}
         \Bigl\{ \sup_{\ell \in [0,1]}|B_\ell|> \kappa'n T^{1/2}\Bigl\}\subset  \Bigl\{ \sum_{i=1}^d \sup_{\ell\in [0,1]}|B_\ell^i|> \kappa'n T^{1/2}\Bigl\}\subset \bigcup_{i=1}^d\Bigl\{\sup_{\ell\in [0,1]} |B^i_\ell |>d^{-1}\kappa'n T^{1/2}\Bigl \}.
     \end{equation*}
     Combining this with the reflection principle and a standard tail estimate of the normal distribution, we obtain
     \begin{equation}\label{e.gaussbd}
         \begin{split}
             \probbb{0}{}{\sup_{\ell \in [0,1]} |B_\ell|> \kappa'n T^{1/2} } &\leq 2d \probbb{0}{}{ \sup_{s\in [0,1]}B_s^1> d^{-1}\kappa'n T^{1/2} }\\
             & =4d\prob{0}{}{ B_1^1> d^{-1}\kappa'n T^{1/2} }\\
             & \leq 4d^2 \frac{\exp\parentesiss{\frac{-\kappa'^2n^{2}}{2d^2}T  }}{\kappa'n T^{1/2}\sqrt{2\pi} }.
         \end{split}
     \end{equation}
     Plugging this back into \eqref{eq:interpolatio1}, we observe that $\sum_{n\in \N}\proba{x}{}{A_n}<\infty$, and hence, by the Borell-Cantelli lemma,
     \begin{equation}\label{eq:Inter1}
        \probaaa{x}{}{\exists N\in \N \text{ s.t. } \forall n\geq N, \sup_{\ell \in [nT,(n+1)T] } d_\mathrm{H}(\mathcal{X}_\ell,\mathcal{X}_{nT})\leq \kappa' nT}=1.
    \end{equation}
    To prove \eqref{eq:interpolation2}, observe by the triangle inequality that, on the event in \eqref{eq:Inter1}, for $n\geq N$ and $\ell \in [0,T]$,
    \begin{equation*}
        \begin{split}
            d_\mathrm{H}(\mathcal{X}_{nT+\ell},\mathcal{X}_{(n+1)T+\ell}) 
            & \leq d_\mathrm{H}(\mathcal{X}_{nT+\ell},\mathcal{X}_{nT}) + d_\mathrm{H}(\mathcal{X}_{nT},\mathcal{X}_{(n+1)T}) + d_\mathrm{H}(\mathcal{X}_{(n+1)T},\mathcal{X}_{(n+1)T+\ell})\\
            & \leq 2\kappa'nT+ \kappa'(n+1) T\\
            & \leq \kappa'n T (2+\tfrac{n+1}{n})\\
            & < (\kappa/2) n T,
        \end{split}
    \end{equation*}
    which yields \eqref{eq:interpolation2}.
\end{proof} 

\section{Half-Space Estimates}\label{sec:ShapeThm}

In this section, we present the proofs of Proposition \ref{propo:introUpper} and Proposition \ref{propo:introLower}. We recall the half-space notation introduced in the introduction. For $e\in S^{d-1}$ and $c\in \R$, let
\begin{equation*}
    \cH_{e,c}:=\{x\in \R^d: x\cdot e = c\},\
    \cH_{e,c}^+:=\{x\in \R^d: x\cdot e> c\}  \text{ and } \cH_{e,c}^-:=\{x\in \R^d: x\cdot e \leq c\}.
\end{equation*}

Throughout this section, we will use the notation $\gamma(e,\lambda)$, first introduced in Section \ref{section:FKPPEigenvalue}, to denote the principal eigenvalue of the eigenvalue problem \eqref{eq:eigenPrelim}.

\subsection{Proof of Proposition \ref{propo:introUpper}: the Upper Estimate}\label{sec:rightestimate}

Proposition \ref{propo:introUpper} is immediate from the following result, which gives more detailed information about the constants in the estimate. 
\begin{proposition}\label{lemma:UpperBound}
    Fix $e\in S^{d-1}$. 
    Let $\lambda_e>0$  be such that $\cs(e)=\frac{\gs(e,\lambda_e)}{\lambda_e}$ as in 
     \eqref{eq:prelimCriticalSpeed} and let $\psi(\cdot\ ;e,\lambda_e)$ the eigenfunction associated to the principal eigenvalue of the eigenvalue problem \eqref{eq:eigenPrelim},  normalized so that $\int_{\TT^d} \psi(x;e, \lambda) dx=1$. Then, for every $t>0$ and $f:[0,\infty)\to \R$,
    \begin{equation}\label{eq:upper-half-estimate}
        \sup_{x\in \R^{d}}\proba{x}{}{\exists v\in \mathcal{N}_t : (X_t(v)-X_{0}(v))\cdot e> \cs(e)t + f(t)} \leq \frac{\max_{y\in \R^d}\psi(y;e,\lambda_{e}) }{ \min_{y\in \R^d}\psi(y;e,\lambda_{e}) } e^{-\lambda_e f(t)}.
    \end{equation}
\end{proposition}

\begin{proof}

    Let $\psi(\cdot\ ;e,\lambda_e)$ be the positive eigenfunction associated to $\gs(e,\lambda_e)$ in \eqref{eq:eigenPrelim}. Since $e$ is fixed, let us write $\psi(x)=\psi(x;e,\lambda_e)$ and $\gs(\lambda_e)=\gs(e,\lambda_e)$. Upon setting $\phi(x)=\phi(x,e; \lambda_{e})$ as in \eqref{e.phidef}, let $\bP^{e,\lambda_e}_x$ be the measure defined in \eqref{eq:measuretilting} and \eqref{eq:WeakSolUniq}, and $\bE^{e,\lambda_e}_x$ the corresponding expected value.  Let $(Y_s)_{s\geq 0}$ be the solution of \eqref{eq:SDETiling2.0} under $\bP^{e,\lambda_e}_x$.
   
   We apply \eqref{eq:changemeasure} in the direction $e$, taking $\lambda=\lambda_e$ and 
    \begin{equation*}
        F((B_s)_{s\in [0,t]}) = \mathbbm{1}_{\{ (B_t-B_0)\cdot e > t\cs(e)+f(t)\}},
    \end{equation*}
    which is $\cF^B_t$-measurable and bounded. We obtain
    \begin{equation*}
        \mathbb{E}_{x}\Big[e^{\int_0^t g(B_s)ds}\mathbbm{1}_{\{(B_t-B_0)\cdot e > t\cs(e)+f(t)\}}\Big] =  \mathbb{E}_{x}^{e,\lambda_e}\Big[\mathbbm{1}_{\{ (Y_t-Y_0)\cdot e > t\cs(e)+f(t)\}}\frac{\psi(x)}{\psi(Y_t)}e^{-\lambda_e e\cdot(Y_t-Y_0)+t\gs(\lambda_e)} \Big].
    \end{equation*}
    We will estimate the right hand side of the above expression. Since $ \lambda_e = \gs(\lambda_e)/\cs(e)$, and $\gamma(\lambda_{e})$ and $c^*(e)$ are positive by Proposition \ref{Prop:PropertiesGamma} and Proposition \ref{p.c*pos}, on the event $\{ (Y_t-Y_0)\cdot e > t\cs(e)+f(t)\}$,  
    \begin{equation*}
        -\lambda_e e\cdot(Y_t-Y_0)+t\gs(\lambda_e) = \gs(\lambda_e)\left(-\frac{e\cdot (Y_t-Y_0)}{\cs(e)} +t \right)< \gamma(\lambda_e)\left(-\frac{t\cs(e)+f(t)}{\cs(e)}+t\right) = - \lambda_e f(t).
    \end{equation*}
    Plugging this back into the prior display, and using that $\psi$ bounded away from zero and bounded from above,
    \begin{equation}\label{e.upper1}
    \begin{aligned}
        \begin{split}
             \mathbb{E}_{x}\Big[e^{\int_0^t g(B_s)ds}\mathbbm{1}_{\{ (B_t-B_0)\cdot e > t(1+\epsilon)\cs(e)\}}\Big] & \leq \mathbb{E}_{x}^{e,\lambda_e}\Big[\mathbbm{1}_{\{ (Y_t-Y_0)\cdot e > t\cs(e)+f(t)\}}\frac{\psi(x)}{\psi(Y_t)}e^{- \lambda_e f(t)}\Big]\\
            & \leq \frac{\max_{y\in \R^d}\psi(y;e,\lambda_{e}) }{ \min_{y\in \R^d}\psi(y;e,\lambda_{e}) }e^{- \lambda_e f(t)}\expecdd{x}{e,\lambda_e}{\mathbbm{1}_{\{ (Y_t-Y_0)\cdot e > t\cs(e)+f(t)\}}}\\
            & \leq \frac{\max_{y\in \R^d}\psi(y;e,\lambda_{e}) }{ \min_{y\in \R^d}\psi(y;e,\lambda_{e}) }e^{- \lambda_e f(t)}.
        \end{split}
        \end{aligned}
    \end{equation}
    On the other hand, by the Markov inequality and the many-to-one lemma (Lemma \ref{lemma:ManyToOneLemma}),
    \begin{equation*}
        \begin{split}
            \proba{x}{}{\exists v\in \mathcal{N}_t: (X_t(v)-x)\cdot e> t\cs(e)+f(t)} &= \probaa{x}{}{\sum_{v\in \mathcal{N}_t}\mathbbm{1}_{\{ (X_t(v)-x)\cdot e > t\cs(e)+f(t)\}} \geq 1}\\
            & \leq \expectedd{x}{}{\sum_{v\in \mathcal{N}_t}\mathbbm{1}_{\{ (X_t(v)-x)\cdot e> t\cs(e)+f(t)\}} }\\
            & = \mathbb{E}_{x}\Big[e^{\int_0^t g(B_s)ds}\mathbbm{1}_{\{ (B_t-B_0)\cdot e > t\cs(e)+f(t)\}}\Big] .
        \end{split}
    \end{equation*} 
    Combining this with \eqref{e.upper1}, we deduce \eqref{eq:upper-half-estimate}.
\end{proof}

\begin{proof}[Proof of Proposition \ref{propo:introUpper}]
     This follows directly from the previous proposition by taking $f(t)=\epsilon \cs(e) t$ and observing that, by Lemma \ref{lemma:UniformBoundsLambda} and Proposition \ref{p.c*pos},
    \begin{equation*}
        \Gamma := \min_{e\in S^{d-1}} \lambda_e \cs(e) >0 \ \ \mathrm{ and }\ \ C:=\sup_{e\in S^{d-1}} \frac{\max_{y\in \R^d}\psi(y;e,\lambda_{e}) }{ \min_{y\in \R^d}\psi(y;e,\lambda_{e}) }<+\infty. \qedhere
    \end{equation*}
\end{proof}
    
\subsection{Proof of Proposition \ref{propo:introLower}: the Lower Estimate}\label{sec:lowerestimate}

In this section, we prove that the probability that at time $t$, every living particle belongs to the half-space $t(1-\epsilon)\mathcal{H}^{-}_{e,c^*(e)}=\{z \in \R^d: z\cdot e < t(1-\epsilon)c^*(e)\}$, decays exponentially on $t$.

\begin{proposition}\label{lemma:DirectionalLowerBound}
    Fix $\epsilon\in (0,1)$ and $e\in S^{d-1}$. There exist $C=C(e)>0$, $\Gamma=\Gamma(\epsilon,e)\in (0,1)$, and $T_{0}=T_{0}(\epsilon, d, \|g\|_{\infty})>0$ such that for all $t\geq T_{0}$,
    \begin{equation}\label{eq:LowerHalfspaceEstimate}
        \sup_{x\in \R^d}\proba{x}{}{\forall v\in \mathcal{N}_t, (X_t(v)-X_0(v))\cdot e\leq (1-\epsilon)\cs(e) t} \leq Ce^{-\Gamma(\epsilon,e)t}.
    \end{equation}
\end{proposition}

The proof of Proposition \ref{lemma:DirectionalLowerBound} will require several ingredients, which we develop throughout this section.

\subsubsection{A Large Deviation Principle}\label{section:largedeviationprinciple}

Let 
\begin{equation}
    \hat{Y}_t:= \frac{1}{t}e\cdot (Y_t-Y_0), \label{eq:hatytdef}
\end{equation}
where $(Y_t)_{t\geq 0}$ is the solution of \eqref{eq:SDETiling2.0}. Observe from \eqref{eq:SDETiling2.0} that we are inherently defining $\hat{Y}_t$ in terms of the eigenvalue problem \eqref{eq:eigenPrelim} that involves the branching rate function $g$. Unless otherwise stated, from now on we assume that $g$ satisfies the hypothesis of Theorem \ref{thm:shapeIntroHaussdorf}, namely, that $g \in C^{0,\alpha}_p(\R^d)$, $g\geq 0$ and $g\not \equiv 0$.

For $x\in \R^d$, let $\mu_t^x$ be the measure on $(\R,\mathcal{B}(\R))$ defined by 
\begin{equation}\label{eq:measuresShape}
    \mu^x_t(A):=\prob{x}{e,\lambda}{\hat{Y}_t\in A}.
\end{equation}
In this section, we prove that the family of measures $(\mu^x_t,t\geq 0)_{x\in \R^d}$ satisfies a large deviation principle. This will be a consequence of a generalization of the G\"artner-Ellis theorem which we state below.

We recall that for a function $I:\R\rightarrow [-\infty,\infty]$, $\zeta\in \R$ is an \textit{exposed point} of $I$ if there exists $\eta\in \R$ such that
\begin{equation}\label{eq:exposinghyperplaneS}
     I(\lambda)   >  I(\zeta) + \eta (\lambda-\zeta)\quad\text{for all $\lambda\neq \zeta$} .
\end{equation}
The number $\eta\in \R$ that satisfies \eqref{eq:exposinghyperplaneS} is called a \textit{strictly supporting hyperplane} of $I$ at $\zeta$. Equipped with this vocabulary, we can now state a uniform version of the G\"artner-Ellis theorem. 

\begin{theorem}[G\"artner-Ellis]\label{thm:GEgenralization}
    Consider a family of probability measures $\{\mu_t^x,t\geq 0\}_{x\in \mathcal{X}}$ on $(\R, \mathcal{B}(\R) )$ indexed by a set $\mathcal{X}$. For $\eta\in \R$, let
    \begin{equation}\label{eq:expmoment}
        \Lambda^x_t(\eta):=\log\bigg[ \int_\R e^{\eta \zeta}\,  d\mu_t^x(\zeta) \bigg].
    \end{equation}
    Assume there exists a function $\Lambda:\R\rightarrow (-\infty,+\infty] $ such that, for every $x\in \mathcal{X}$ and $\eta \in \R$,
    \begin{equation}\label{eq:limlogaritmicS}
        \lim_{t\rightarrow \infty}\frac{1}{t}\Lambda_t^x(t\eta)= \Lambda(\eta).
    \end{equation}
    Let $D_\Lambda := \{\eta\in \R: \Lambda(\eta)<\infty\}$ and suppose that $0 \in \mathrm{int}(D_\Lambda)$, and that the convergence above is uniform in $\mathcal{X}$ on $D_\Lambda$, meaning that for every $\eta\in D_{\Lambda}$,
    \begin{equation}\label{eq:GEuniform}
        \lim_{t\rightarrow \infty}\sup_{x\in \mathcal{X}}\ \bigg| \frac{1}{t}\Lambda_t^x(t\eta) -  \Lambda(\eta) \bigg| =0.
    \end{equation}
    
    Let $\Lambda^*:\R\rightarrow [-\infty,\infty]$ be the Legendre transform of $\Lambda$, given by
    \begin{equation*}\label{eq:legendreS}
        \Lambda^*(\zeta) = \sup_{\eta\in \R} \sqbracket{\eta \zeta-\Lambda(\eta)}.
    \end{equation*}
    Let $\mathcal{L}\subset \R$ be the set of exposed points of $\Lambda^*$ for which there exists a strictly supporting hyperplane $\eta \in\mathrm{int}(D_\Lambda)$. Then,
    \begin{enumerate}[(a)]
        \item for every closed set $\mathcal{C}\subset \R$, $\limsup_{t\rightarrow \infty} \sup_{x\in \mathcal{X}}\frac{1}{t}\log \mu_t^x(\mathcal{C})\leq -\inf_{\zeta\in \mathcal{C}}\Lambda^*(\zeta)$,
        \item for every open set $\mathcal{O}\subset \R$, $\liminf_{t\rightarrow \infty}\inf_{x\in \mathcal{X}}\frac{1}{t}\log \mu_t^x(\mathcal{O})\geq -\inf_{\zeta\in \mathcal{O}\cap \mathcal{L}}\Lambda^*(\zeta)$.
    \end{enumerate}
    Moreover, if:
    \begin{enumerate}
        \item $\Lambda$ is lower semicontinuous on $\R$,
        \item $\Lambda$ is differentiable on $\mathrm{int}(D_\Lambda)$,
        \item $\lim_{\eta\rightarrow \partial D_{\Lambda}}| \Lambda^\prime(\eta)|=\infty $,
    \end{enumerate}
    then we can replace the set $\mathcal{O}\cap \mathcal{L}$ by $\mathcal{O}$ in the lower bound in (b) above.
\end{theorem}

The theorem above is a consequence of the fact that, with the extra assumption \eqref{eq:GEuniform}, every bound in the proof of the G\"artner-Ellis theorem can be taken uniformly in $ \mathcal{X}$. This is shown in Appendix \ref{section:gartnerellis}, following the lines of the proof of the G\"artner-Ellis Theorem given in \cite[Chapter 2.3]{MR2571413}. 

For the following lemma, we proceed as in \cite[Lemma 2.6]{MR4492971}. 

\begin{lemma}\label{lemma:logmoment}
    Suppose that the branching rate function $g$ satisfies that $g \in C^{0,\alpha}_p(\R^d)$, $g\geq 0$ and $g\not \equiv 0$. Then, for every $\eta \in \R$ and $e\in S^{d-1}$, there exist $\theta_1,\theta_2\in \R$ such that, for every $x\in \R^d$, 
    \begin{equation}\label{eq:bound0}
        \frac{\theta_1}{t}\leq \frac{1}{t} \log \expec{x}{e,\lambda_e}{e^{\eta t \hat{Y}_t}} -\gamma(e,\lambda_e + \eta)+\gamma(e,\lambda_e) \leq \frac{\theta_2}{t},
    \end{equation}
    where $\hat{Y}_{t}$ is defined in \eqref{eq:hatytdef} and $\lambda_{e}$ and $\gamma$ are as in \eqref{eq:prelimCriticalSpeed}.
\end{lemma}
\begin{remark}
    A non-quantitative version of Lemma \ref{lemma:logmoment} appears in \cite[Lemma 2.1]{Freidlin+1985}. The proof differs substantially from ours. While we do not require quantitative bounds on the rate of convergence for our arguments, they may be useful for other arguments in the future. Also, the proof we present here is self-contained and only relies on estimates already established in this paper. 
\end{remark}
\begin{proof}
    Through this proof, $e\in S^{d-1}$ is fixed, so let us write $\psi(x;\lambda)=\psi(x;e,\lambda)$ and $\gamma(\lambda)=\gamma(e,\lambda)$ for simplicity. Fix $x\in \R^d$. We apply Lemma \ref{Lemma:DirectionalChangeOfMeasure} in the direction $e$ with $\lambda=\lambda_e$ and taking the nonnegative $\cF_t^B$-measurable function
    \begin{equation*}
        F((B_s)_{s\in [0,t]})=\frac{\psi(B_t;\lambda)}{\psi(x;\lambda)}e^{-t\gamma(\lambda)+\lambda e \cdot (B_t-B_0)}e^{\eta e\cdot (B_t-B_{0})},
    \end{equation*}
    where $(B_t)_{t\geq 0}$ is a Brownian motion that starts at $x$ under a measure $\bP_x$ (as in Lemma \ref{Lemma:DirectionalChangeOfMeasure}). We choose this $F$ so there is a cancellation in the left-hand side of \eqref{eq:changemeasure}. Thus, by applying Lemma \ref{Lemma:DirectionalChangeOfMeasure}, we get
    \begin{equation*}
        \begin{split}
            \frac{1}{t}\log \expec{x}{e,\lambda_e}{e^{ \eta t \hat{Y}_t}}&= \frac{1}{t}\log \expec{x}{e,\lambda_e}{e^{ \eta e\cdot(Y_t-Y_0) }}\\
            & = \frac{1}{t}\log \mathbb{E}_{x}\bigg[\frac{\psi(B_t;\lambda_e)}{\psi(x;\lambda_e)}e^{\lambda_e e\cdot(B_t-x)+\int_0^t g(B_s)\ ds -\gamma(\lambda_e)t } e^{\eta e\cdot (B_t-x)}\bigg]\\
            & = -\gamma(\lambda_e)-\frac{(\lambda_e+\eta)e\cdot x}{t} + \frac{1}{t}\log \mathbb{E}_{x}\bigg[\frac{\psi(B_t;\lambda_e)}{\psi(x;\lambda_e)}e^{(\lambda_e +\eta)e\cdot B_t+\int_0^t g(B_s)\ ds }\bigg].
        \end{split}
    \end{equation*}
    Since $\psi(\cdot\ ; \lambda_e)$ is bounded away from zero and bounded from above, there exist $C_1,C_2\in \R$, depending on $e$, such that
    \begin{equation}\label{eq:bound1}
        \frac{C_1}{t}\leq \frac{1}{t}\log \expec{x}{e,\lambda_e}{e^{\eta t \hat{Y_t}}} + \gamma(\lambda_e) +\frac{(\lambda_e+\eta)e\cdot x}{t} - \frac{1}{t}\log \expecdd{x}{}{e^{e\cdot B_t(\lambda_e +\eta)+\int_0^t g(B_s) ds }}\leq \frac{C_2}{t}.
    \end{equation}
    Let us estimate the term involving $\mathbb{E}_x$ in the expression above. We again apply Lemma \ref{Lemma:DirectionalChangeOfMeasure} in the direction $e$, but now taking $\lambda=\lambda_e + \eta$ and 
    \begin{equation*}
        F((B_s)_{s\in [0,t]}) = e^{(\lambda_e+\eta)e\cdot B_t}.
    \end{equation*}
    With this choice of $e$ and $\lambda$, let $(Y_t^\prime)_{t\geq 0}$ be the solution of \eqref{eq:SDETiling2.0} under the measure $\mathbb{P}_x^{e,\lambda_e+\eta}$ introduced in \eqref{eq:SDETiling2.0}. We obtain
    \begin{equation*}
        \begin{split}
            \log \mathbb{E}_{x}\bigg[e^{\int_0^t g(B_s) ds+(\lambda_e +\eta)e\cdot B_t }\bigg]&=\log \mathbb{E}_{x}^{e,\lambda_e+\eta}\bigg[\frac{\psi(x;\lambda_e+\eta)}{\psi(Y_t^\prime;\lambda_e+\eta)} e^{-(\lambda_e+\eta)e\cdot (Y_t^\prime-x)+\gamma (\lambda_{e}+\eta)t}e^{(\lambda_e+\eta)e\cdot Y^\prime_t} \bigg]\\
            & = \log \mathbb{E}_{x}^{e,\lambda_e+\eta}\bigg[\frac{\psi(x;\lambda_e+\eta)}{\psi(Y_t^\prime;\lambda_e+\eta)} e^{(\lambda_e+\eta)e\cdot x+\gamma(\lambda_e + \eta)t} \bigg]\\
            &= (\lambda_e +\eta)e\cdot x + \gamma(\lambda_e+\eta)t+\log \expecddd{x}{e,\lambda_e+\eta}{\frac{\psi(x;\lambda_e+\eta)}{\psi(Y_t^\prime;\lambda_e+\eta)}}.
        \end{split}
    \end{equation*}
    Since $\psi(\cdot\ , \lambda_e + \eta)$ is bounded away from zero and bounded from above, there exist $C_1^\prime,C_2^\prime\in \R,$ depending on $e$ and $\eta$, such that
    \begin{equation*}
        \frac{C_1^\prime}{t}\leq \frac{1}{t}\log \expecdd{x}{}{e^{(\lambda_e +\eta)e\cdot B_t+\int_0^t g(B_s) ds }} - \frac{(\lambda_e +\eta)e\cdot x}{t} - \gamma(\lambda_e+\eta) \leq \frac{C_2^\prime}{t}.
    \end{equation*}
    Combining this with \eqref{eq:bound1}, we obtain \eqref{eq:bound0}
\end{proof}
With this in mind, we obtain the corresponding large deviation principle as a direct consequence of the G\"artner-Ellis theorem (Theorem \ref{thm:GEgenralization}). This is an analogous result to the one-dimensional large deviation principle \cite[Lemma 2.6]{MR4492971}.

\begin{proposition}\label{proposition:largeDeviation}
    Fix $x\in \R^d$ and $e\in S^{d-1}$. Let $\lambda_e>0$  be such that $\cs(e)=\frac{\gs(e,\lambda_e)}{\lambda_e}$ as in 
     \eqref{eq:prelimCriticalSpeed}. Let $(Y_t)_{t\geq 0}$ be the solution of \eqref{eq:SDETiling2.0} under the measure $\bP_x^{e,\lambda_e}$. Let $\hat{Y}_t:=(1/t)e\cdot(Y_t-Y_0)$. Let $I_e:\R\rightarrow (-\infty,+\infty]$ be the function defined by
    \begin{equation}\label{eq:goodratefunction}
        I_e(\zeta)=\sup_{\eta \in \R} \{ \eta \zeta - [\gamma(e,\lambda_e+\eta)-\gamma(e,\lambda_e)]\}=[\gamma(e, \lambda_{e}+\cdot)-\gamma(e, \lambda_{e})]^{*}(\zeta).
    \end{equation}
    For a set $A\subset \R$, let $I_e(A):=\inf_{\zeta\in A}I_e(\zeta)$. The following holds:
    \begin{enumerate}[(a)]
        \item For every closed set $\mathcal{C}\subset \R$, $\limsup_{t\rightarrow \infty}\sup_{x\in \R^d}\frac{1}{t}\log \bP_x^{e,\lambda_e}[\hat{Y}_t \in \mathcal{C}]\leq -I_e(\mathcal{C})$.
        \item For every open set $\mathcal{O}\subset \R$, $\liminf_{t\rightarrow \infty}\inf_{x\in \R^d}\frac{1}{t}\log  \bP_x^{e,\lambda_e}[\hat{Y}_t \in \mathcal{O}]\geq -I_e(\mathcal{O})$.
    \end{enumerate}

\end{proposition}
\begin{proof}
     The conclusion follows from the G\"artner-Ellis Theorem (Theorem \ref{thm:GEgenralization}) if the family of measures $\{\mu^x_t,t\geq 0\}_{x\in \R^d}$ defined in \eqref{eq:measuresShape} satisfies the hypotheses. Consider the function $\Lambda:\R\rightarrow\R$ given by $\Lambda(\cdot) =\gamma(e,\lambda_e+\cdot)-\gamma(e,\lambda_e)$. By Lemma \ref{lemma:logmoment}, we know that for any $\eta \in \R$,
    \begin{equation*}
        \lim_{t\rightarrow \infty} \sup_{x\in \R^d}\ \bigg |\log\Big[ \int_\R e^{\eta \zeta} d\mu_t^x(\zeta) \Big] - \Lambda(\eta) \bigg| =\lim_{t\rightarrow \infty} \sup_{x\in \R^d}\ \bigg| \frac{1}{t} \log \expec{x}{e,\lambda_e}{e^{\eta t \hat{Y}_t}} - \Lambda(\eta) \bigg|=0.
    \end{equation*}
    This implies that, in the notation of Theorem \ref{thm:GEgenralization}, $D_\Lambda=\R$, and in particular, $0\in \mathrm{int}(D_\Lambda)$. It also implies that $|t^{-1}\Lambda_t^x(t\eta)-\Lambda(\eta)|\to 0$ as $t \to \infty$ uniformly over $x\in \R$, so \eqref{eq:GEuniform} holds. This already implies that (a) and (b) in Theorem~\ref{thm:GEgenralization} hold.
    
    All that remains is to show that the function $\Lambda$ satisfies properties 1., 2.\ and 3.\ of Theorem \ref{thm:GEgenralization}. From Proposition~\ref{prop:PrincipalEigenEx}(i), we know that $\gamma(e,\cdot)$ is differentiable on $\R$, and hence $\Lambda$ is differentiable and so also lower semicontinous on $\mathrm{int}(D_\Lambda)= \R$. For the purposes of contradiction, suppose that $|\partial_{\eta}\gamma(e, \eta)|\not\rightarrow +\infty$ as $|\eta|\rightarrow +\infty$. Then there exists $C>0$, and a sequence $(\eta_n)_{n\geq 0}$ such that $|\eta_n|\rightarrow +\infty$ and $|\partial_{\eta}\gamma(e, \eta_n)|\leq C$. Without loss of generality, we can further assume that $\eta_n \uparrow +\infty$ (the case $\eta_n \downarrow -\infty$ follows by a symmetric argument). By the convexity and differentiability of $\gamma(e;\cdot)$, $\partial_\eta \gamma(e,\cdot)$ is increasing. It follows that $\gamma(e;\cdot)$ grows at most linearly, with slope $C$, which contradicts the quadratic lower bound in Proposition~\ref{prop:PrincipalEigenEx}(i). Hence, $|\partial_{\eta}\gamma(e, \eta)|\rightarrow +\infty$ as $|\eta|\rightarrow +\infty$, and thus, $|\Lambda^\prime(\eta)|\rightarrow +\infty$ as $|\eta|\rightarrow +\infty$. This verifies that $\Lambda$ satisfies the hypothesis of the G\"artner-Ellis theorem, and we conclude that (a) and (b) hold.
\end{proof}

In the following lemma, we list a few properties of the function $I_e$ introduced in \eqref{eq:goodratefunction}. The proof is given in Appendix \ref{section:gartnerellis}, as it is purely a convex analysis argument.

\begin{lemma}\label{lemma:propI} 
The following holds:
\begin{enumerate}[(i)]
    \item The function $I_e(\zeta)$ is finite for every $\zeta\in \R$, and $I_e$ is strictly convex and differentiable. Moreover, for every $\zeta\in \R$, there exists $\zeta^*\in \R$ such that
    \begin{equation*}
        I_e(\zeta)=\sup_{\eta \in \R} \{ \eta \zeta - [\gamma(e,\lambda_e+\eta)-\gamma(e,\lambda_e)]\}= \eta \zeta^* - [\gamma(e,\lambda_e+\zeta^*)-\gamma(e,\lambda_e)]. 
    \end{equation*}
    \item For any $\zeta\in \R$, $I_e(\zeta)\geq 0$ and $I_e(\zeta)=0$ if and only if $\zeta=c^*(e)$, where $c^{*}(e)$ is the front speed defined in \eqref{eq:prelimCriticalSpeed}. Moreover, $I_e$ is strictly decreasing on the interval $(-\infty,c^*(e))$ and strictly increasing on the interval $(c^*(e),+\infty)$.
    \item For any constant $\beta>0$, there exists $\delta=\delta(\beta)\in (0,1)$ such that for any $\kappa\in (0,\delta)$, 
    \begin{equation*}
        I_e((1-\kappa)c^*(e)) < \kappa \beta\quad\text{and}\quad I_e((1+\kappa)c^*(e)) < \kappa \beta. 
    \end{equation*}
\end{enumerate} 
\end{lemma}

As a direct consequence of the large deviation principle (Proposition \ref{proposition:largeDeviation}), we obtain the following estimate on the expected number of particles who have traveled a distance approximately $c^{*}(e)$, in direction $e$.

\begin{lemma}\label{lemma:lowerestimate}
    Let $I_e$ be defined by \eqref{eq:goodratefunction}. Fix $\epsilon\in (0,1)$ and $\alpha \in (0,1)$. There exist constants $C=C(e)>0$ and $T=T(e)>0$ such that for every $t>T$,
    \begin{equation}\label{eq:lowerdirectionalestimate}
        \inf_{x\in \R^d}\mathbb{E}_{x}\Big[\sum_{v\in \mathcal{N}_t} \mathbbm{1}_{\{(1-\alpha \epsilon)tc^*(e)\geq (X_t(v)-x)\cdot e\geq (1-\epsilon)tc^*(e) \}} \Big] \geq Ce^{ t[\alpha\epsilon\gamma(e,\lambda_e)-2I_{e}((1-\alpha\epsilon)c^*(e))] }.
    \end{equation}
    Furthermore, there is $\epsilon_0=\epsilon_0(e)\in (0,1)$ such that for $\epsilon\in (0,\epsilon_0)$ and for any $\alpha \in (0,1)$, 
    $$\alpha\epsilon\gamma(e,\lambda_e)-2I_{e}((1-\alpha\epsilon)c^*(e)) >0.$$ In particular, for any $\epsilon\in (0,\epsilon_0)$, the expected value in \eqref{eq:lowerdirectionalestimate} grows to $+\infty$ as $t\uparrow +\infty$.
\end{lemma}

\begin{proof}
    Fix $e\in S^{d-1}$, $\epsilon>0$, and $\alpha \in (0,1)$. Let us introduce the functional $$F((B_s)_{s\in [0,t]}) = \mathbbm{1}_{\{(1-\alpha \epsilon)tc^*(e)\geq (B_t-B_0)\cdot e\geq (1-\epsilon)tc^*(e) \}},$$ which is $\cF^B_t$-measurable. By the many-to-one lemma (Lemma \ref{lemma:ManyToOneLemma}),
    \begin{equation*}
        \mathbb{E}_{x}\bigg[\sum_{v\in \mathcal{N}_t} \mathbbm{1}_{\{(1-\alpha \epsilon)tc^*(e)\geq (X_t(v)-x)\cdot e\geq (1-\epsilon)tc^*(e) \}} \bigg] = \mathbb{E}_{x}\bigg[{e^{\int_0^t g(B_s) ds}F((B_s)_{s\in [0,t]})}\bigg],
    \end{equation*}
    where $(B_s)_{s\geq 0}$ is a Brownian motion started at $x$ under the measure $\bP_x$. To estimate the right-hand side of the expression above, we apply Lemma \ref{Lemma:DirectionalChangeOfMeasure} in the direction $e$, taking $\lambda=\lambda_e$ and the functional $F$ as above. Since $\psi(\cdot \ ;e,\lambda_e)$ is bounded away from zero and bounded from above, then there exists a positive constant $C=C(e)>0$ such that
    \begin{equation*}
        \begin{split}
             \mathbb{E}_{x}\bigg[e^{\int_0^t g(B_s) ds}F((B_s)_{s\in [0,t]})\bigg] &= \mathbb{E}_{x}^{e,\lambda_e}\bigg[\frac{\psi(x;e,\lambda_e)}{\psi(Y_t;e,\lambda_e)}  e^{-\lambda_e e \cdot (Y_t-x) +t\gamma(e,\lambda_e) } F((Y_s)_{s\in [0,t]}) \bigg]\\
             & \geq C \ \expecddd{x}{e,\lambda_e}{ e^{-\lambda_e e \cdot (Y_t-x) +t\gamma(e,\lambda_e) }F((Y_s)_{s\in [0,t]}) },
        \end{split}
    \end{equation*}
    where $\bE_{x}^{e,\lambda_e}$ is the expected value of the measure introduced in \eqref{eq:measuretilting} and \eqref{eq:WeakSolUniq}, and $(Y_s)_{s\geq 0}$ is the solution of \eqref{eq:SDETiling2.0} under $\bP^{e,\lambda}_x$. Recall that by \eqref{eq:prelimCriticalSpeed}, $\lambda_e c^*(e) = \gamma(e,\lambda_e)$. Hence, on the event $$\{(1-\alpha \epsilon)t c^*(e) \geq (Y_t-x) \cdot e \geq (1-\epsilon)t c^*(e) \},$$ we have $e^{-\lambda_e e \cdot (Y_t-x) +t\gamma(e,\lambda_e)} \geq e^{t\alpha\epsilon \gamma(e,\lambda_e)}$. Combining this with the previous bound, we obtain
    \begin{equation*}
    \begin{split}
             \mathbb{E}_{x}\bigg[e^{\int_0^t g(B_s) ds}F((B_s)_{s\in [0,t]})\bigg] &\geq Ce^{t\alpha\epsilon\gamma(e,\lambda_e)}\bE_{x}^{e,\lambda_e}[ F((Y_s)_{s\in [0,t]}) ]\\
             &=Ce^{t\alpha\epsilon\gamma(e,\lambda_e)}\prob{x}{e,\lambda_e}{(1-\alpha \epsilon)tc^*(e)\geq (Y_t-x)\cdot e\geq (1-\epsilon)tc^*(e)}.
 \end{split}
 \end{equation*}
    Now we estimate the probability that appears on the right-hand side of the expression above, for large values of $t$. Let $\mathcal{O}$ be the open interval $\mathcal{O}:= ((1-\epsilon)c^*(e),(1-\alpha\epsilon)c^*(e))$. Recall the notation $I_e(\mathcal{O})=\inf_{x\in \mathcal{O}} I_e(x)$. From Lemma \ref{lemma:propI}(ii), we know that $I_e$ is strictly decreasing on $(-\infty,c^*(e)]$ {and strictly positive on $(-\infty, c^{*}(e))$}. This implies that $I_e(\mathcal{O}) = I_e((1-\alpha \epsilon)c^*(e))>0$. By Proposition \ref{proposition:largeDeviation}(b), we can find $T=T(e)>0$ such that for every $t>T$, for every $x\in \R^{d}$
    \begin{equation*}
       \frac{1}{t}\log \prob{x}{e,\lambda_e}{(1-\alpha \epsilon)tc^*(e)\geq (Y_t-Y_0)\cdot e\geq (1-\epsilon)tc^*(e)} \geq -2 I_e((1-\alpha \epsilon)c^*(e)),
    \end{equation*}
    and hence
    \begin{equation*}
\inf_{x\in \R^{d}} \prob{x}{e,\lambda_e}{(1-\alpha \epsilon)tc^*(e)\geq (Y_t-Y_0)\cdot e\geq (1-\epsilon)tc^*(e)}\geq e^{-2tI_e((1-\alpha \epsilon)c^*(e))}.
    \end{equation*}
    Combining this with our previous estimate, we prove the first claim of the lemma.

    The second claim follows directly from Lemma \ref{lemma:propI}(iii). Letting  $\beta=\gamma(e,\lambda_e)/2$ in Lemma \ref{lemma:propI}(iii), there is $\delta=\delta(\beta)\in (0,1)$ such that for any $\kappa\in (0,\delta)$, $$I_e((1-\kappa)c^*(e))< \kappa \gamma(e,\lambda_e)/2.$$ Hence, for any $\epsilon\in (0,\delta[$ and any $\alpha\in (0,1)$, by taking $\kappa=\alpha \epsilon$ above, we observe that $$2I_e((1-\alpha\epsilon)c^*(e))< \alpha\epsilon\gamma(e,\lambda_e),$$ which shows the claim.
    
\end{proof}

\subsubsection{The embedded supercritical branching process}\label{section:embeddedsupercriticalbranching}

Recall that for  $0\leq s\leq t$ and $v\in \mathcal{N}_t$, we denote by $X_s(v)$ the position of the living ancestor of $v$ at time $s$. For two particles $u,v$, we write $u\leq v$ if $u$ is an ancestor of $v$, with the convention that any particle is an ancestor of itself. 

We now define a discrete-time spatial branching process, which will be constructed inductively. Fix $\epsilon>0$ and $T_0>0$. Let
\begin{equation}\label{eq:FirstGeneration}
    I_1^{T_0,\epsilon} := \{v\in \mathcal{N}_{T_0}: (X_{T_0}(v)-X_0(v))\cdot e \geq (1-\epsilon)c^*(e) T_0 \}.
\end{equation}
We let $I_1^{T_0,\epsilon}$ be the first generation of the discrete-time spatial branching process. Observe that $I_1^{T_0,\epsilon}\subset \cN_{T_0}$. Suppose that for $n\geq 1$, we have defined $I_n^{T_0,\epsilon}$, the $n$'th generation of the discrete-time spatial branching process, and that $I_n^{T_0,\epsilon}\subset \cN_{nT_0}$. We will define the $(n+1)$-generation $I_{n+1}^{T_0,\epsilon}$ as a subset of $\cN_{(n+1)T_0}$. For $v\in I_n^{T_0,\epsilon}$, the descendants of $v$ in the discrete-time branching process are the particles $u\in \cN_{(n+1)T_0}$ such that $v\leq u$ and
\begin{equation*}
    (X_{(n+1)T_0}(u)-X_{nT_0}(u))\cdot e\geq (1-\epsilon)c^*(e) T_0.
\end{equation*}
In other words, 
\begin{equation}\label{eq:Generation}
    I_{n+1}^{T_0,\epsilon} := \{v \in \mathcal{N}_{(n+1)T_0}: v_{nT_0} \in I_{n}^{T_0,\epsilon},  (X_{(n+1)T_0}(v)-X_{nT_0}(v))\cdot e\geq (1-\epsilon)c^*(e) T_0\}.
\end{equation}
Observe that for $k\in \N$, if $v\in I_{k}^{T_0,\epsilon}$, it follows that
\begin{equation}\label{eq:generationEmbed}
    (X_{k T_0}(v)-X_0(v))\cdot e =\sum_{j=0}^{k-1}(X_{(j+1) T_0}(v)-X_{j T_0}(v))\cdot e \geq (1-\epsilon)c^*(e) k T_0.
\end{equation}

By Lemma \ref{lemma:lowerestimate}, we can take $T_0=T_0(\epsilon)$ sufficiently large such that $\inf_{x\in \R^d}\expected{x}{}{\# I_1^{T_0,\epsilon}}>1$, which suggests that $(\#I_k^{T_0,\epsilon},k\geq 0)$ behaves like the generation size of a supercritical multitype branching process; in particular, we expect that the discrete-time branching process survives with positive probability. The following lemma is an application of results from \cite[Chapter 3]{R-381-PR}; its proof is given in Appendix \ref{append:survival}. 

\begin{lemma}\label{lemma:SurvivalProbability}
    Fix $\epsilon>0$ and suppose there is $T_{0}=T_0(\epsilon)>0$ such that $\inf_{x\in \R^d}\expected{x}{}{\# I_1^{T_0,\epsilon}}>1$. Then,
    \begin{equation*}
        \inf_{x\in \R^d}\proba{x}{}{\# I_n^{T_0,\epsilon} \neq 0 \ \forall n\in \N } >0.
    \end{equation*}
\end{lemma}

We combine the previous result with the interpolation lemma (Lemma \ref{lemma:interpolation}) to obtain the following.

\begin{lemma}\label{cor:WeakInterpolation}
    Fix $\epsilon\in (0,1)$. There exist $T_0=T_0(\epsilon)>0$, $K=K(\epsilon, d)\in \N$, and $ \eta \in (0,1)$, such that
    \begin{equation}\label{eq:lower1}
        \inf_{x\in \R^d} \mathbf{P}_x[ \forall t>KT_0, \ \sup_{v\in \cN_t} \{(X_t(v)-x)\cdot e\}\geq (1-\epsilon)c^*(e)t]\geq \eta>0.
    \end{equation}
\end{lemma}
\begin{proof}
    By Lemma \ref{lemma:lowerestimate} and Lemma \ref{lemma:SurvivalProbability}, there exists $T_0=T_{0}(\epsilon)>0$ such that 
    \begin{equation}\label{e.intermed}
        \inf_{x\in \R^d}\probaaa{x}{}{\# I_n^{T_0,\epsilon/2} \neq 0 \ \forall n\in \N } >0.
    \end{equation}
    For $n\in \N$, let $A_{n}\subset \Omega$ be the event defined by 
    \begin{equation*}
        A_n:=\Bigl\{\sup_{\ell \in [nT_0,(n+1)T_0]} d_\mathrm{H}(\mathcal{X}_{nT_0},\mathcal{X}_\ell)> (\epsilon/2)\cs(e) nT_0 \Bigl \}, 
    \end{equation*}
    where we recall that $\mathcal{X}_{\ell}=\left\{X_{\ell}(v): v\in \mathcal{N}_{\ell}\right\}\subseteq \R^{d}$. By \eqref{eq:generationEmbed}, on the event $A_n^c \cap \{\# I_n^{T_0,\epsilon/2} \neq 0\}$, there exists $v\in \cN_{(n+1)T_0}$ such that, for every $\ell\in [nT_0,(n+1)T_0]$, 
    \begin{equation}\label{eq:lower2}
        \begin{split}
            (X_{\ell}(v)-X_0(v))\cdot e &= ((X_{nT_0} -X_0(v)) + (X_{\ell}(v) - X_{nT_0}))\cdot e \\
            &\geq (1-\epsilon/2)c^*(e) n T_0 - (\epsilon/2)c^*(e)n T_0 \\
            &= (1-\epsilon)c^*(e) n T_0.
        \end{split}
    \end{equation}
    Thus, we have
    \begin{equation}\label{eq:pos1}
        \begin{split}
            \mathbf{P}_x[ \forall t>KT_0, \ \sup_{v\in \cN_t} \{(X_t(v)-x)\cdot e\} \ge (1-\epsilon) c^*(e)t]\geq\probaaa{x}{}{\bigcap_{n\in \N}\{ I_n^{T_0,\epsilon/2} \neq \emptyset\}, \bigcap_{n\ge K} A_{n}^c  }.
        \end{split}
    \end{equation}
    We now claim that there exists $K>0$ sufficiently large such that 
    \begin{equation}\label{eq:pos}
        \probaaa{x}{}{\bigcap_{n\in \N}\{ I_n^{T_0,\epsilon/2} \neq \emptyset\}, \bigcap_{n=K}^{\infty} A_{n}^c  }\geq \frac{1}{4}\inf_{x\in \R^d}\probaaa{x}{}{\bigcap_{n\in \N}\{ I_n^{T_0,\epsilon/2} \neq \emptyset\} }.
    \end{equation}
    To do so, for $k\in \N$, by the same reasoning as for \eqref{eq:interpolatio1} and \eqref{e.gaussbd}, and a union bound, we observe that there exist constants $C_1,C_2>0$, depending on $d$ and $\epsilon$, such that
    \begin{equation*}
        \probaaa{x}{}{\bigcap_{n=k}^{\infty} A_{n}^c } = 1- \probaaa{x}{}{\bigcup_{n=k}^{\infty} A_{n} } \geq 1-\sum_{n=k}^{\infty}  C_1 e^{(n+1)T_0 \|g\|_{\infty}} e^{-C_2 ((\epsilon/2) \cs(e))^2n^2 T_0}.
    \end{equation*}
   By the integral test, the series on the right hand side with $k=0$ converges. This implies that by \eqref{e.intermed}, there exists $K=K(\epsilon, d)\in \N$, independent of $x\in \R^d$, such that 
    \begin{equation*}
        \probaaa{x}{}{\bigcap_{n=K}^{\infty} A_{n}^c } \geq 1-\frac{1}{4}\inf_{x\in \R^d}\probaaa{x}{}{\bigcap_{n\in \N}\{ I_n^{T_0,\epsilon/2} \neq \emptyset\} }.
    \end{equation*}
    To obtain \eqref{eq:pos} we now argue by contradiction. Indeed, let $E:=\bigcap_{n\in \N}\{ I_n^{T_0,\epsilon/2} \neq \emptyset\}$, and $F:=\bigcap_{n=K}^\infty A_{n}^c$. By the prior display, we know that $\mathbf{P}_{x}[F]\geq 1-\tfrac{1}{4}\inf_{x\in \R^{d}} \mathbf{P}_{x}[E]$, and by \eqref{e.intermed}, $\inf_{x\in \R^{d}} \mathbf{P}_{x}[E]>0$. If \eqref{eq:pos} did not hold, this would imply
    \begin{equation*}
        \mathbf{P}_{x}[E\cup F]=\mathbf{P}_{x}[E]+\mathbf{P}_{x}[F]-\mathbf{P}_{x}[E\cap F]>1,  
    \end{equation*}
    and we obtain a contradiction, which proves \eqref{eq:pos}. Plugging \eqref{eq:pos} back into \eqref{eq:pos1}, and since $K=K(\epsilon, d)\in \N$ does not depend on the choice of $x\in \R^d$, the above implies \eqref{eq:lower1}.
\end{proof}

\subsubsection{The cutoff argument}\label{section:cutoffargument}

We now use the branching property of the BBM to improve the bound in Lemma \ref{cor:WeakInterpolation}. This will allow us to conclude Proposition \ref{lemma:DirectionalLowerBound}. We introduce some notation that will be used in this subsection. Fix $e\in S^{d-1}$. Let
\begin{equation}\label{eq:MaxDisplacement}
    M_t:=\max_{v\in \cN_t}\{(X_t(v)-X_0(v))\cdot e\}\ \text{ and } \ M_t^-:=\min_{v\in \cN_t}\{(X_t(v)-X_0(v))\cdot e\}.
\end{equation}
For $L\geq 0$, $v\in \cN_L$, and $s\geq 0$, let $R_s^{v,L}\in \R^d$ be defined by
\begin{equation}\label{eq:MaxDisplacement1}
    R_s^{v,L}:=\text{argmax}_{\{X_{L+s}(u)-X_0(u) \in \R^d:\ u\in \cN_{L+s}, v\leq u\}}\{(X_{L+s}(u)-X_0(u))\cdot e  \}.
\end{equation}
Observe that $R^{v,L}_s\in \R^{d}$ is the recentered position of the child of $v\in \mathcal{N}_{L}$ who travels the maximal displacement in direction $e$ in time $s$. In particular, for the initial particle $\emptyset$, observe that
\begin{equation*}
    R_t^{\emptyset,0}\cdot e = M_t.
\end{equation*}

We are finally ready to prove Proposition \ref{lemma:DirectionalLowerBound}. The proof is motivated by the last step of the proof of \cite[Theorem 1]{MR3127890}. 
\begin{proof}[Proof of Proposition \ref{lemma:DirectionalLowerBound}]
    Let $\delta := (1+\tfrac{1}{2})c^*(-e)$. Let $\kappa>0$ be defined by
    \begin{equation*}
        \kappa := \frac{c^*(e)\epsilon}{4\parentesiss{\delta+\parentesiss{1 - \frac{\epsilon }{2}}c^*(e)}}.
    \end{equation*}
    Let $L := \kappa t$. Observe from the definition of $L$ and $\kappa$ that $$[\delta+(1-\tfrac{\epsilon}{2})c^*(e)]L\leq \tfrac{\epsilon}{2}c^*(e)t,$$
    and that this inequality is equivalent to
    \begin{equation}\label{eq:ineq1}
        \delta L + (1-\epsilon)c^*(e) t \leq (1-\tfrac{\epsilon}{2})c^*(e)(t-L).
    \end{equation}
   Let $M_{t}$ and $M_{L}^{-}$ be as in \eqref{eq:MaxDisplacement}. We have that 
    \begin{equation}\label{eq:5.2.90}
        \begin{split}
            \mathbf{P}_{x}[M_t \leq (1-\epsilon)c^*(e)t] &\leq \proba{x}{}{M_t\leq (1-\epsilon)c^*(e)t,M_L^- \leq - \delta L }\\
            &\  + \proba{x}{}{M_t \leq (1-\epsilon)c^*(e)t,M_L^- > - \delta L }\\
            & \leq \proba{x}{}{M_L^{-}\leq -\delta L} + \proba{x}{}{M_t \leq (1-\epsilon)c^*(e)t,M_L^- > - \delta L }.
        \end{split}
    \end{equation}
    We apply Proposition \ref{lemma:UpperBound} in the direction $-e$ to control
    \begin{equation*}
         \begin{split}
             \proba{x}{}{M_L^{-}\leq -\delta L} &= \proba{x}{}{M_L^{-}\leq -(1+\tfrac{1}{2})c^*(-e)\kappa t}\\
             &= \proba{x}{}{-\max_{v\in \cN_L}\{(X_L(v)-X_0(v))\cdot (-e)\}\leq -(1+\tfrac{1}{2})c^*(-e)\kappa t}\\
             &= \proba{x}{}{\max_{v\in \cN_L}\{(X_L(v)-X_0(v))\cdot (-e)\}\geq (1+\tfrac{1}{2})c^*(-e)\kappa t}\\
             & \leq C e^{-t \frac{\kappa}{2} \gamma(-e,\lambda_{-e})}.
         \end{split}
    \end{equation*}  
    For the second term on the right-hand side of \eqref{eq:5.2.90}, we decompose based on the event $\{\#\cN_L < t\}$, to get
    \begin{equation}\label{eq:5.2.9}
        \begin{aligned}
            \mathbf{P}_{x}&[M_t \leq (1-\epsilon)c^*(e)t,\, M_L^- > - \delta L] \\
            &\quad\leq \proba{x}{}{\#\cN_L \leq  t} + \proba{x}{}{M_t \leq (1-\epsilon)c^*(e)t,M_L^- > - \delta L, \#\cN_L>t }\\
            & \quad \leq \proba{x}{}{\#\cN_L \leq  t} + \cproba{x}{}{M_t \leq (1-\epsilon)c^*(e)t,M_L^- > - \delta L }{\#\cN_L>t}.
        \end{aligned}
    \end{equation}
   Recall that by the definitions introduced in \eqref{eq:MaxDisplacement} and \eqref{eq:MaxDisplacement1}, 
    \begin{equation*}
        \begin{split}
            \mathbf{P}_x[M_t \leq (1-&\epsilon)c^*(e)t,M_L^- > - \delta L  \ | \ \#\cN_L>t]  \\
            &   = \cprobaa{x}{}{\max_{v\in \cN_L}\{ R^{v,L}_{t-L}\cdot e \}\leq (1-\epsilon)c^*(e) t,M_L^- > - \delta L}{\#\cN_L>t}\\
            &   \leq \cprobaa{x}{}{\max_{v\in \cN_L}\{ R^{v,L}_{t-L}\cdot e -M_L^- \}\leq \delta L + (1-\epsilon)c^*(e) t}{\#\cN_L>t}\\
            &   \leq \cprobaa{x}{}{\max_{v\in \cN_L}\{ (R^{v,L}_{t-L}-X_L(v))\cdot e  \}\leq \delta L + (1-\epsilon)c^*(e) t  }{\#\cN_L>t}\\
            & \leq \cprobaa{x}{}{\max_{v\in \cN_L}\{ (R^{v,L}_{t-L}-X_L(v))\cdot e  \}\leq (1-\epsilon/2)c^*(e)(t-L)}{\#\cN_L>t}\\
            & = \cprobaa{x}{}{\forall v\in \cN_L, (R^{v,L}_{t-L}-X_L(v))\cdot e  \leq (1-\epsilon/2)c^*(e)(t-L)}{\#\cN_L>t}.
        \end{split}
    \end{equation*}
    Observe that conditioned on $\cF_L$, by the Markov property and the branching property, the processes $((R^{v,L}_{t-L}-X_L(v))_{t\geq L})_{v\in \cN_L}$ are independent. Hence, by the Markov property and Lemma \ref{cor:WeakInterpolation}, we can find $T_0(\epsilon, d)>0$ and $\eta\in (0,1)$ such that for every $t\geq T_0$, 
    \begin{equation*}
        \begin{split}
            \mathbf{P}_{x}\bigl[\forall v\in \cN_L, (R^{v,L}_{t-L}&-X_L(v))\cdot e  \leq (1-\epsilon/2)c^*(e)(t-L)  \bigl | \#\cN_L>t\big ]\\
            & = \cexpectedd{x}{}{ \prod_{v\in \cN_L}\proba{X_L(v)}{}{(R^{v,L}_{t-L}-X_L(v))\cdot e  \leq (1-\epsilon/2)c^*(e)(t-L)}   }{\#\cN_L>t}\\
            & \leq  \cexpectedd{x}{}{ \parentesiss{\sup_{y\in \R^d}\proba{y}{}{ M_{t-L} \leq (1-\epsilon/2)c^*(e)(t-L)}}^{\#\cN_L}   }{\#\cN_L>t}\\
            &\leq (1-\eta)^t.
        \end{split} 
    \end{equation*}
   To bound the first term of the right-hand side of \eqref{eq:5.2.9}, we compare $\#\cN_{t}$ with the generation size of a discrete-time supercritical Bienaym\'e tree which, by construction, will survive almost surely. To do so, in view of Lemma \ref{lemma:tailBranchingTime}, let us fix $T>0$ sufficiently large, depending on $\|g\|_{\infty}$ and $d$, such that $\sup_{x\in \R^d} \proba{x}{}{\tau_\emptyset >T}\leq 1/2$. It follows that $(\#\mathcal{N}_{nT})_{n\in \N}$ stochastically dominates the generation size process of a Bienaym\'e tree with offspring distribution $p_1=1/2=p_2$. The offspring has mean size $m=3/2$ and the probability of extinction is $q=0$. Thus, by \cite[Lemma 1]{10.1214/aoap/1177004832}, we find $\zeta>1$ and $a\in \N$ such that,
    \begin{equation*}
        \sup_{x\in \R^d}\proba{x}{}{\#\cN_{anT}< \zeta^n }=o(\zeta^{-n}).   
    \end{equation*}
    Let us denote by $\lfloor \cdot \rfloor$ the floor function. Using the fact that $\#\cN_t$ is non-decreasing in $t$, we find a constant $C>0$ such that for $T_0$ sufficiently large, for all $t\geq T_{0}$,
    \begin{equation}\label{ThinnigNumber}
        \sup_{x\in \R^d}\proba{x}{}{\#\cN_{t}< \zeta^{\lfloor t/aT \rfloor } } \leq \sup_{x\in \R^d}\proba{x}{}{\#\cN_{\lfloor t/aT \rfloor aT}< \zeta^{\lfloor t/aT \rfloor } } \leq C\zeta^{-\lfloor t/aT \rfloor } \leq C_1e^{-\log(\zeta)\tfrac{t}{aT}}.
    \end{equation}
    Now, since $\zeta^{\lfloor t/aT \rfloor }$ grows exponentially in $t$, we can further take $T_0$ larger in such a way that, for every $t\geq T_0$, $t < \zeta^{\lfloor t\kappa/aT \rfloor}$. Combining this with the previous bound, we obtain that for $T_0$ sufficiently large, for all $t\geq T_0$,
    \begin{equation*}
         \sup_{x\in \R^d}\proba{x}{}{\#\cN_{\kappa t}< t } \leq \sup_{x\in \R^d}\proba{x}{}{\#\cN_{\kappa t}< \zeta^{\lfloor \kappa t/aT \rfloor } } \leq C_1e^{-\log(\zeta)\tfrac{\kappa t}{aT}}.
    \end{equation*}
    Plugging the previous bounds into \eqref{eq:5.2.9} and then into \eqref{eq:5.2.90}, we conclude that for all $t\geq T_{0}$,
    \begin{equation*}
        \begin{split}
            \mathbf{P}_{x}[M_t \leq (1-\epsilon)c^*(e)t] &\leq C e^{-t\frac{\kappa }{2} \gamma(-e,\lambda_{-e})} +  Ce^{-t\log(\zeta)\tfrac{\kappa}{aT}} + (1-\eta)^t\\
            &\leq Ce^{-t\min\{\frac{\kappa}{2}\gamma(-e,\lambda_{-e}),\ \log(\zeta)\tfrac{\kappa}{aT},\ -\log(1-\eta) \}}\\
            &=Ce^{-t\Gamma(\epsilon,e)},
        \end{split}
    \end{equation*}
    where $\Gamma(\epsilon,e)=\min\left\{\frac{\kappa}{2}\gamma(-e,\lambda_{-e}),\ \log(\zeta)\tfrac{\kappa}{Ta},\ -\log(1-\eta) \right\}$.
\end{proof}

\section{The proof of Theorem \ref{thm:shapeIntroHaussdorf} and Proposition \ref{eq:quantShapeThm}}\label{s.shape}

In what follows, the function $d_{\mathrm{H}}(\cdot, \cdot)$ will denote the Hausdorff distance between subsets of $\R^d$. Equipped with the half-space estimates, we now use Proposition \ref{prop:wulffapprox} to prove a shape theorem for the convex hull of the $g$-BBM. This result will serve as the main tool to prove Theorem \ref{thm:shapeIntroHaussdorf}. Recall that Proposition \ref{prop:wulffapprox} guarantees that for $\mathfrak{c}^{*}$ a bounded function on $S^{d-1}$, for every $\epsilon\in (0,1)$, there are finite sets $\mathcal{R}, \mathcal{Q}\subset S^{d-1}$ such that 
 \begin{enumerate}[(i)]
        \item $\bigcap_{r\in \cR} \cH_{r,\mathfrak{c}^*(r)}^{-}\subset(1+\epsilon)\cW(\mathfrak{c}^*)$, and
        \item for every compact convex set $K\subset B(0,2\sup_{e\in S^{d-1}}\mathfrak{c}^{*}(e))$, if for every $q\in \cQ$, $K\cap \cH^+_{q,\mathfrak{c}^*(q)} \neq \emptyset$, then 
        \begin{equation*}
            (1-\epsilon)\cW(\mathfrak{c}^*)\subset K.
        \end{equation*}
    \end{enumerate}

\begin{proof}[Proof of Proposition \ref{prop:QuantShapeTheorem}]
    Let us denote by $\mathcal{A}_t$ the event in \eqref{eq:quantShapeThm}. We will bound from above the probability of $\mathcal{A}_t^c$. First, let $\epsilon_1\in (0,1)$ be such that $(1+\epsilon_1)^2< (1+\epsilon)$. We take $\mathfrak{c}^*(\cdot)=(1+\epsilon_1)\cs(\cdot)$, in which case it is easy to see that $\cW(\mathfrak{c}^*)=(1+\epsilon_1)\cW(\cs)$. By Proposition \ref{p.c*pos},  $\mathfrak{c}^{*}$ is bounded. Therefore, by Proposition \ref{prop:wulffapprox}(i), there exists a finite set $\cR\subset S^{d-1}$ such that
    \begin{equation*}
        \bigcap_{r\in \cR} \cH_{r,(1+\epsilon_1)\cs(r)}^-=\bigcap_{r\in \cR} \cH_{r,\mathfrak{c}^*(r)}^-\subset(1+\epsilon_1)\cW(\mathfrak{c}^*)= (1+\epsilon_1)^2\cW(\cs)\subset (1+\epsilon)\cW.
    \end{equation*}
    By a union bound and Proposition \ref{lemma:UpperBound}, for every $x\in \R^d$,
    \begin{equation*}
        \begin{split}
            \mathbf{P}_x[ t^{-1}(H_t-x) \not\subset(1+\epsilon)\cW(c^*)]& \leq \mathbf{P}_x[ t^{-1}(H_t-x) \not\subset \cap_{r\in \cR} \cH_{r,(1+\epsilon_1)c^*(r)}^-]\\
            & \leq \sum_{r\in \cR} \mathbf{P}_x[ t^{-1}(\mathcal{X}_t-x)\cap \cH_{r,(1+\epsilon_1)c^*(r)}^+ \neq \emptyset ]\\
            & \leq (\#\cR)\big(\max_{r\in \cR}C(r)\big) e^{-\epsilon_1 t\min_{r\in \cR}\gamma(r,\lambda_r) }.
        \end{split}
    \end{equation*}
    We now fix a new parameter $\epsilon_2>0$ such that $(1-\epsilon)\leq (1-\epsilon_2)^2$. Let $\mathfrak{c}^*(\cdot)=(1-\epsilon_2)\cs(\cdot)$ which has Wulff shape $\cW(\mathfrak{c}^*)=(1-\epsilon_2)\cW(\cs)$. Again, $\mathfrak{c}^{*}$ is bounded and strictly positive by Proposition \ref{p.c*pos}. Let $a^{-1}:=\sup_{e\in S^{d-1}}\mathfrak{c}^{*}\vee(\inf_{e\in S^{d-1}}\mathfrak{c}^{*})^{-1}$. By Proposition \ref{prop:wulffapprox}(ii), there is a finite set $\cQ\subset S^{d-1}$ such that, for every compact convex set $K\subset B(0,2a^{-1})$, if for every $q\in \cQ$, $\cH^+_{q,\mathfrak{c}^*(q)}\cap K\neq \emptyset$, then
    \begin{equation*}
        (1-\epsilon)\cW(\cs)\subset (1-\epsilon_2)^2\cW(\cs)= (1-\epsilon_2)\cW(\mathfrak{c}^*)\subset K.
    \end{equation*}
        
    By Proposition \ref{lemma:DirectionalLowerBound}, for every $x\in \R^d$ and $t>0$ sufficiently large depending on $\epsilon_2, d$ and $\|g\|_{\infty}$,
    \begin{equation*}
        \begin{split}
            \mathbf{P}_x[(1-\epsilon)&\cW(c^*)\not\subset t^{-1}(H_t-x), \ t^{-1}(H_t-x) \subset(1+\epsilon)\cW(c^*)]\\
            &\leq  \mathbf{P}_x[(1-\epsilon)\cW(c^*)\not\subset t^{-1}(H_t-x), \ t^{-1}(H_t-x) \subset B(0, 2a^{-1}) ] \\
            &\leq  \mathbf{P}_x[\exists q\in \cQ \text{ s.t. } t^{-1}(\mathcal{X}_t-x)\cap \cH^+_{q,(1-\epsilon_2)c^*(q)} = \emptyset, \ t^{-1}(H_t-x) \subset B(0, 2a^{-1}) ] \\
            &\leq \sum_{q\in \cQ} \mathbf{P}_x[ t^{-1}(\mathcal{X}_t-x)\cap \cH^+_{q,(1-\epsilon_2)c^*(q)} = \emptyset] \\
            & \leq (\# \cQ) \big(\max_{q\in \cQ} C(q)\big) e^{-t\min_{q\in \cQ}\Gamma(\epsilon_2,q)}.
        \end{split}
    \end{equation*}
    Since these bounds do not depend on the choice of $x\in \R^d$, for all $t>0$ sufficiently large, depending on  $\epsilon, d$ and $\|g\|_{\infty}$, 
    \begin{equation*}
        \sup_{x\in \R^d}\mathbf{P}_x(\mathcal{A}_t^c)\leq (\# \cQ) \big(\max_{q\in \cQ} C(q)\big)e^{-t\min_{q\in \cQ}\Gamma(\epsilon_2,q) } + (\#\cR)\big(\max_{r\in \cR}C(r)\big) e^{-\epsilon_1 t \min_{r\in \cR}\gamma(r,\lambda_r)}.
    \end{equation*}
    This implies the result. 
\end{proof}

\begin{remark}\label{remark:ShapeConvexHull}
    The previous proposition, combined with the Borel-Cantelli lemma and the interpolation lemma (Lemma \ref{lemma:interpolation}), imply that, for all $\epsilon\in (0,1)$ and $x\in \R^d$, 
    \begin{equation*}
        \proba{x}{}{\text{$\exists T>0$ s.t. $\forall t\geq T, t(1-\epsilon)\cW \subset H_t-x\subset t(1+\epsilon)\cW$}} = 1.
    \end{equation*}
\end{remark}

With the shape theorem for the convex hull of the $g$-BBM at hand, Theorem \ref{thm:shapeIntroHaussdorf} will be a consequence of the following fact about convex sets. For a convex set $W$, $x\in W$ is called an \textit{extreme point} of $W$ if it cannot be expressed in the form $x=tz+(1-t)y$ for some $y,z \in W$ and $t\in (0,1)$.

\begin{lemma}\label{lemma:ApproxExtShape}
    Let \( W\subset \R^{d} \) be compact and convex. Fix \( \epsilon \in (0,1) \). There exists \( \delta \in (0,\epsilon) \) such that, for any set \( E\subset \R^{d} \) with 
    \[
    (1 - \delta)W \subset \mathrm{conv}(E) \subset (1 + \delta)W,
    \]
    if $\xi\in W$ is an extreme point, then $B(\xi, \epsilon)\cap E\neq \emptyset$.
\end{lemma}

Let us first prove the following auxiliary lemma.

\begin{lemma}\label{lemma:AuxExtreme}
    Let $W\subset \R^{d}$ be a convex set and let $x$ be an extreme point of $W$. Then for any $\epsilon\in (0,1)$, $x\not \in \mathrm{conv}(W\setminus B(x,\epsilon))$.
\end{lemma}
\begin{proof}
For the purposes of contradiction, assume that $x\in \mathrm{conv}(W\setminus B(x,\epsilon))$. Then, there exist $\alpha_1,...,\alpha_n>0$ with $\sum_{i=1}^n \alpha_i =1$ and distinct $y_1,...,y_n\in W\setminus B(x,\epsilon)$ such that $x=\sum_{i=1}^n \alpha_i y_i$. Since $x\not \in W\setminus B(x,\epsilon)$, it must be the case that $n\geq 2$ and that $\alpha_1,...,\alpha_n < 1$. Notice that $\sum_{i=2}^n \tfrac{\alpha_i}{1-\alpha_1} = 1$, and that $y_2,...,y_n$ all belong to $W$. Since $W$ is convex, $z:=\sum_{i=2}^n \tfrac{\alpha_i}{1-\alpha_1}y_i \in W$. But then 
\begin{equation*}
    x = \sum_{i=1}^n \alpha_i y_i = \alpha_1 y_1 + (1-\alpha_1) z,
\end{equation*}
which contradicts the fact that $x$ is an extreme point of $W$.
\end{proof}

\begin{proof}[Proof of Lemma \ref{lemma:ApproxExtShape}]
    Suppose for the purposes of contradiction that for $\epsilon>0$, there exists a sequence of sets $\left\{E_{n}\right\}_{n\in \mathbb{N}}\subset \R^{d}$ and a sequence of extreme points $\left\{x_{n}\right\}_{n\in \mathbb{N}}$  of $W$ such that for all $n \in \N$, $(1 - \tfrac{1}{n})W \subset \text{conv}(E_n) \subset (1 + \tfrac{1}{n})W$ and $\text{dist}(x_n, E_n) > \tfrac{\epsilon}{2}$.  Note that this implies that \( \text{conv}(E_n) \to W \) in the Hausdorff distance as \( n \to \infty \). Since \( W \) is compact, up to extracting a subsequence, we can assume \( x_n \to x' \in W \), and then, for all \( n \) sufficiently large, $\text{dist}(x', E_n) > \epsilon/4$. Now, let us take $K>0$ sufficiently large such that $|x_K-x'|< \epsilon/8$, then, $\mathrm{dist}(x_K,E_n)>\epsilon/8$ for all $n$ sufficiently large. Let us set $x:=x_K$, which is an extreme point of $W$ satisfying $\mathrm{dist}(x,E_n)>\epsilon/8$ for all $n$ sufficiently large. We will use $x$ to obtain a contradiction.

    First, for $n\in \N$ consider the set $A_n:=\mathrm{conv}((1+\tfrac{1}{n})W\setminus B(x,\epsilon/8))$. We claim that $A_n$ converges to $\mathrm{conv}(W\setminus B(x,\epsilon/8))$ in the Hausdorff distance. Fix $\delta>0$ and observe that $(1+\tfrac{1}{n})W\setminus B(x,\epsilon/8)$ converges to $ W\setminus B(x,\epsilon/8)$ in the Hausdorff distance as $n\to \infty$. Thus, there exists $N>0$ sufficiently large such that, for $n>N$, 
    \begin{equation*}
        d_H\big((1+\tfrac{1}{n})W\setminus B(x,\epsilon/8), W\setminus B(x,\epsilon/8)\big)<\delta.
    \end{equation*}
    Now fix $n>N$, and $y\in A_n$. We can choose $\alpha_1,...,\alpha_m>0$ adding up to 1, and $y_1,...,y_m\in (1+\tfrac{1}{n})W\setminus B(x,\epsilon/8)$ such that $y=\sum_{i=1}^m \alpha_i y_i$. By the previous display, we can find $y'_1,...,y'_m\in W\setminus B(x,\epsilon/8)$ such that $|y_i-y_i'|<\delta$ for every $i\in [m]$. Thus, writing $y':=\sum_{i=1}^m \alpha_i y_i' \in \mathrm{conv}(W\setminus B(x,\epsilon/8))$, we have 
    \begin{equation*}
        |y-y'|=\Big|y - \sum_{i=1}^m \alpha_i y_i' \Big| \leq \sum_{i=1}^m \alpha_i |y_i-y_i'| < \delta\, .
    \end{equation*}
    Since $\delta>0$ and $y\in A_n$ were arbitrary, it follows that for all $\delta>0$ there is $N>0$ such that for all $n>N$,
    \begin{equation*}
     A_n=\mathrm{conv}((1+\tfrac{1}{n})W\setminus B(x,\epsilon/8)) \subset \mathrm{conv}(W\setminus B(x,\epsilon/8)) + B(0,\delta).
    \end{equation*}
    Since $\mathrm{conv}(W\setminus B(x,\epsilon/8))\subset \mathrm{conv}((1+\tfrac{1}{n})W\setminus B(x,\epsilon/8))$, the above implies that $A_n$ converges to $\mathrm{conv}(W\setminus B(x,\epsilon/8))$ in the Hausdorff distance, which proves the claim.

    Finally, to obtain a contradiction, by Lemma \ref{lemma:AuxExtreme}, $x\not\in \mathrm{conv}(W\setminus B(x,\epsilon/8))$, and thus,
    \begin{equation*}
        \kappa:=\mathrm{dist}\big(x,\mathrm{conv}(W\setminus B(x,\epsilon/8)))>0.
    \end{equation*}
    Using the claim on the Hausdorff convergence proved above, we choose $N>0$ sufficiently large such that, for every $n>N$, $d_H(A_n, \mathrm{conv}(W\setminus B(x,\epsilon/8)))< \kappa/2$. Recall that $\mathrm{conv}(E_n) \subset \mathrm{conv}((1+\tfrac{1}{n})W\setminus B(x,\epsilon/8))$. Thus, for $n>N$,  
    \begin{equation*}
        \begin{split}
            \mathrm{dist}(x,\mathrm{conv}(E_n)) & \geq \mathrm{dist}(x,\mathrm{conv}((1+\tfrac{1}{n})W\setminus B(x,\epsilon/8)))\\
            & = \mathrm{dist}(x, A_n)\\
            & \geq \mathrm{dist}(x, \mathrm{conv}(W\setminus B(x,\epsilon/8))) - d_H(A_n, \mathrm{conv}(W\setminus B(x,\epsilon/8))) > \kappa/2,
        \end{split}
    \end{equation*}
    contradicting that $\mathrm{conv}(E_n)$ converges to $W$ in the Hausdorff distance.

\end{proof}

Proposition \ref{prop:QuantShapeTheorem} and the previous lemma tell us that the $g$-BBM approximates the extreme points of $\cW(\cs)$. Now we show that it approximates any point of $\cW(\cs)$.

\begin{proposition}
    For all $\xi\in \cW(c^*)$ and $\epsilon \in (0,1)$, there exist  $C=C(\xi,\epsilon,d), \Gamma=\Gamma(\xi,\epsilon,d)>0$ and $T_0=T_0(\epsilon, \xi, d, \|g\|_{\infty})$ such that, for every $t>T_0$,
    \begin{equation}\label{eq:ApproxIntW}
        \sup_{x\in \R^d} \proba{x}{}{(\mathcal{X}_t-x)\cap tB(\xi,\epsilon)=\emptyset } \leq Ce^{-\Gamma t}.
    \end{equation}
\end{proposition}
\begin{proof}
     Let $\text{Ext}(\cW(c^*))$ denote the set of extreme points of $\cW(c^*)$. Carath\'eodory's theorem states that every point of $\mathcal{W}(c^{*})$ can be written as a convex combination of at most $d+1$ extreme points (see \cite[Theorem 1.3]{MR1242986}), i.e. there exist $d'\in[d+1]$ and  $\alpha_1,\ldots, \alpha_{d'} > 0$ with $\sum_{i=1}^{d'} \alpha_{i}=1$ and $\{\xi_1,...,\xi_{d'}\}\subset \text{Ext}(\cW(c^*))$ such that $\xi=\sum_{i=1}^{d'}\alpha_{i}\xi_{i}$. For $i\in [d']$, let $\beta^{(i)}:= \alpha_1  + ... + \alpha_{i} $ and set $\beta^{(0)}=0$. For each $t>0$, we consider the event
    \begin{equation*}
        \mathcal{E}_t:=\{\exists v\in \mathcal{N}_t: \forall i\in [d'], X_{\beta^{(i)}t}(v)-X_{\beta^{(i-1)}t}(v) \in  t\alpha_i B(\xi_i,\epsilon) \}.
    \end{equation*}
   On the event $\mathcal{E}_{t}$, writing $X_t(v)-X_0(v)$ as a telescoping sum and applying the triangle inequality, we have 
    \begin{equation}
        \begin{split}
            |X_t(v)-X_0(v) - t\xi| \leq  \sum_{i=1}^{d'} \bigl | X_{\beta^{(i)}t}(v)-X_{\beta^{(i-1)}t}(v) -t\alpha_i \xi_i \bigl |\leq \sum_{i=1}^{d'} t\alpha_i \epsilon = t\epsilon.
        \end{split} 
    \end{equation}
    This implies that $\mathcal{E}_t\subset\left\{\exists v\in \mathcal{N}_{t}:  (X_t(v)-x)\in tB(\xi, \epsilon)\right\}=\{(\mathcal{X}_t-x)\cap tB(\xi,\epsilon)\neq \emptyset \}$. Thus, inductively applying the Markov property and the branching property at times $\alpha_i t$, we obtain, for every $x\in \R^d$,
    \begin{equation*}
        \proba{x}{}{(\mathcal{X}_t-x)\cap tB(\xi,\epsilon)\neq \emptyset } \geq \proba{x}{}{\mathcal{E}_t} \geq \prod_{i=1}^{d'} \inf_{y\in \R^d} \proba{y}{}{(\mathcal{X}_{\alpha_i t}-y) \cap t\alpha_i B(\xi_i, \epsilon) \neq \emptyset }.
    \end{equation*}
    By Proposition \ref{prop:QuantShapeTheorem} and Lemma \ref{lemma:ApproxExtShape} applied with $E=\mathcal{X}_{\alpha_{i}t}-y$, there exist $\delta\in (0,\epsilon)$ and positive constants $C=C(\delta)$, $ \Gamma=\Gamma(\delta)$, and $T_1=T_1(\delta, d, \|g\|_{\infty})$ with $Ce^{-T_1\Gamma \min_{i \in [d']} \alpha_i} < 1$ such that, if $\min_{i\in [d']}(t \alpha_i)>T_1$, then
    \begin{equation*}
        \proba{x}{}{(\mathcal{X}_t-x)\cap tB(\xi,\epsilon)\neq \emptyset } \geq \prod_{i=1}^{d'}(1-Ce^{-t\alpha_i\Gamma})\geq 1 - C\sum_{i=1}^{d'} e^{-t\alpha_i\Gamma}\geq 1-d'C e^{-t\Gamma\min_{i\in [d']}\alpha_i},
    \end{equation*}
    where in the second inequality we use that $\prod_{i=1}^{d'} (1-z_i)\geq 1-\sum_{i=1}^{d'} z_i$ for $z_1,...,z_{d'}\in [0,1)$. 
    Taking the complement of the event above, since this bound is independent of the choice of $x\in \R^d$, we conclude \eqref{eq:ApproxIntW}.
    \end{proof}

    \begin{remark}\label{remark:approxEvery}
        By a direct application of the Borel-Cantelli lemma and the interpolation lemma (Lemma \ref{lemma:interpolation}), the previous result gives, for every $\xi \in \cW(\cs)$, $\epsilon \in (0,1)$ and $x\in \R^d$,
    \begin{equation*}
        \proba{x}{}{\exists T_0>0 \text{ s.t } \forall t>T_0, \ (\mathcal{X}_t-x)\cap tB(\xi,\epsilon)\neq \emptyset } =1.
    \end{equation*}
    \end{remark}
    
    \begin{proof}[Proof of Theorem \ref{thm:shapeIntroHaussdorf}]
    By compactness of $\cW(\cs)$, there exists a finite set $\{\xi_1,..., \xi_\ell\}\subset \cW(\cs)$ such that $\cW(\cs)\subset\cup_{i=1}^\ell B(\xi_i,\epsilon/4)$. For $i\in [\ell]$, by Remark \ref{remark:approxEvery}, for every $x\in \R^d$,
    \begin{equation*}
        \mathbf{P}_x[\exists T_i>0 \text{ s.t } \forall t>T_i, \ (\mathcal{X}_t-x)\cap tB(\xi_i,\epsilon/4)\neq \emptyset] =1.
    \end{equation*}
     Intersecting these events over all $i\in [\ell]$, we obtain that $\mathbf{P}_x$-almost surely, for $t>\max_{i\in [\ell]}T_i$, 
    \begin{equation}\label{eq:Hauss1}
        \cW(\cs)\subset\bigcup_{i=1}^\ell B(\xi_i,\epsilon/4) \subset t^{-1}(\mathcal{X}_t-x)+B(0,\epsilon).
    \end{equation}
    On the other hand, since $\cs$ is bounded, and by Remark \ref{remark:ShapeConvexHull}, we can find $\epsilon'>0$ such that
    \begin{equation}\label{eq:Hauss2}
        \begin{split}
            \mathbf{P}_x[\exists T_0>0 \text{ s.t } \forall  t&>T_0,\ t^{-1}(\mathcal{X}_t-x)\subset  \cW(\cs)+B(0,\epsilon)]\\
            &\geq \mathbf{P}_x[\exists T_0>0 \text{ s.t } \forall t>T_0,\ t^{-1}(\mathcal{X}_t-x)\subset (1+\epsilon')\cW(\cs)] = 1
        \end{split}
    \end{equation}
    Finally, intersecting the events in \eqref{eq:Hauss1} and \eqref{eq:Hauss2}, $\mathbf{P}_x$-almost surely, for $t>\max_{i\in \{0\}\cup [\ell]}T_i$,
    \begin{equation*}
         \cW(\cs)\subset t^{-1}(\mathcal{X}_t-x)+B(0,\epsilon)\text{ and } t^{-1}(\mathcal{X}_t-x)\subset  \cW(\cs)+B(0,\epsilon),
    \end{equation*}
    which finishes the proof.
    \end{proof}

\appendix

\section{Appendix A: Survival of the embedded branching process}\label{append:survival}

For the proof of the survival result, Lemma \ref{lemma:SurvivalProbability}, we apply a criterion for the survival of general branching processes (see \cite[Theorem 2.4]{maillard2025generalisedprincipaleigenvaluesglobal}). Applied to our setting, the criterion for survival reads as follows.

\begin{lemma}\cite[Theorem 2.4.2]{maillard2025generalisedprincipaleigenvaluesglobal}\label{lemma:criterion_survival}
    Suppose there exists a bounded measurable function $q:\R^d \to \R$, $\delta \in (0,1)$, and $\rho\geq 1$ such that, for every $x\in \R^d$,
    \begin{equation}\label{eq:criterion_survival}
        \mathbb{E}_x\Bigg[\sum_{v\in I_1^{T_0,\epsilon}}q(X_t(v))\exp\Big( -\delta \sum_{v\in I_1^{T_0,\epsilon}}q(X_t(v)) \Big)\Bigg] \geq \rho q(x).
    \end{equation}
   If $q(x)>0$, then
    \begin{equation*}
         \proba{x}{}{\# I_n^{T_0,\epsilon} \neq 0 \ \forall n\in \N } \geq  \proba{x}{}{\liminf_{n\to \infty}\rho^{-n}\# I_n^{T_0,\epsilon}>0  }\geq 1-e^{-\delta q(x)}>0.
    \end{equation*}
\end{lemma}
\begin{remark}
    The lower bound with the exponential term above is not explicitly given in the statement of \cite[Theorem 2.4.2]{maillard2025generalisedprincipaleigenvaluesglobal}, but it can be deduced from the very last step of the proof.
\end{remark}

\begin{proof}[Proof of Lemma \ref{lemma:SurvivalProbability}]
    We apply Lemma \ref{lemma:criterion_survival} to obtain a uniform lower bound on the probability of survival of the discrete branching process. Let us take $q\equiv 1$, any $\rho>1$, and, by Lemma \ref{lemma:lowerestimate}, $T_0>0$ sufficiently large such that
    \begin{equation}\label{eq:T0}
        \inf_{x\in \R^d}\expected{x}{}{\# I_1^{T_0,\epsilon}}>2e \rho.
    \end{equation}
    To choose $\delta\in (0,1)$, let $\bar{N}_{T_0}$ denote the total number of particles at time $T_0$ of a branching diffusion with constant branching rate $\|g\|_\infty$. Observe that the random variable $\bar{N}_{T_0}$ stochastically dominates $\# I_1^{T_0,\epsilon}$. Since $\bar{N}_{T_0}$ is integrable, and its distribution does not depend on the initial position $x\in \R^d$, let us fix $\delta\in (0,1)$ sufficiently small such that 
    \begin{equation}\label{eq:HomBranching}
        \sup_{x\in \R^d} \mathbb{E}_x\big[  \bar{N}_{T_0}\indc_{\{ \bar{N}_{T_0} > \delta^{-1} \}}\big] \leq e\rho.
    \end{equation}
    With this at hand, observe that
    \begin{equation*}
        \begin{split}
            \mathbb{E}_x\Big[\sum_{v\in I_1^{T_0,\epsilon}}q(X_t(v))\exp\Big( -\delta \sum_{v\in I_1^{T_0,\epsilon}}q(X_t(v)) \Big)\Big]  &=  \mathbb{E}_x\big[ (\# I_1^{T_0,\epsilon})\exp (-\delta (\# I_1^{T_0,\epsilon}) ) \big]\\
            &\geq \mathbb{E}_x\big[ (\# I_1^{T_0,\epsilon})\indc_{\{ \delta (\# I_1^{T_0,\epsilon}) < 1 \}}\exp (-\delta (\# I_1^{T_0,\epsilon}) ) \big]\\
            & \geq e^{-1} \mathbb{E}_x\big[ (\# I_1^{T_0,\epsilon})\indc_{\{ \# I_1^{T_0,\epsilon} < \delta^{-1} \}}\big]\\
            & = e^{-1} \big( \mathbb{E}_x\big[ \# I_1^{T_0,\epsilon}\big]-\mathbb{E}_x\big[ (\# I_1^{T_0,\epsilon})\indc_{\{ \# I_1^{T_0,\epsilon} \geq  \delta^{-1} \}}\big] \big).
        \end{split}
    \end{equation*}
    Finally, $\# I_1^{T_0,\epsilon}$ is stochastically dominated by $\bar{N}_{T_0}$, and thus, by \eqref{eq:T0} and \eqref{eq:HomBranching}, 
    \begin{equation*}
        e^{-1} \big( \mathbb{E}_x\big[ \# I_1^{T_0,\epsilon}\big]-\mathbb{E}_x\big[ (\# I_1^{T_0,\epsilon})\indc_{\{ \# I_1^{T_0,\epsilon} \geq  \delta^{-1} \}}\big] \big) \geq e^{-1}\big( 2e\rho -\mathbb{E}_x\big[  \bar{N}_{T_0}\indc_{\{ \bar{N}_{T_0} > \delta^{-1} \}}\big] \big) \geq \rho = \rho q(x).
    \end{equation*}
    Plugging this back into our previous inequality shows that \eqref{eq:criterion_survival} holds with our choice of $q$, $\delta$, and $\rho$. Thus,
    \begin{equation*}
        \proba{x}{}{\# I_n^{T_0,\epsilon} \neq 0 \ \forall n\in \N } \geq 1-e^{-\delta }>0.
    \end{equation*}
    Since $\delta\in (0,1)$ was chosen independent of the initial position $x\in \R^d$, the conclusion follows.
\end{proof}

\section{Appendix B: The G\"artner-Ellis Theorem}\label{section:gartnerellis}

For a function $I:\R\rightarrow [-\infty,\infty]$, we say $\zeta\in \R$ is an \textit{exposed point} of $I$ if for some $\eta\in \R$ and for all $\lambda \in \R$ with $\lambda\neq \zeta$,
\begin{equation}\label{eq:exposinghyperplane}
     I(\lambda)   >  I(\zeta) + \eta (\lambda-\zeta) .
\end{equation}
The number $\eta\in \R$ that satisfies \eqref{eq:exposinghyperplane} is called a \textit{strictly supporting hyperplane} of $I$. We begin by stating a ``uniform version'' of the G\"artner-Ellis Theorem. 

\begin{theorem}[G\"artner-Ellis]\label{thm:GEgenralizationA}
    Consider a family of probability measures $\{\mu_t^x,t\geq 0\}_{x\in \mathcal{X}}$ on $(\R, \mathcal{B}(\R) )$ indexed by a set $\mathcal{X}$. For $\eta\in \R$, let
    \begin{equation*}
        \Lambda^x_t(\eta):=\log\bigg[ \int_\R e^{\eta\zeta} d\mu_t^x(\zeta) \bigg].
    \end{equation*}
    Assume there exists a function $\Lambda:\R\rightarrow (-\infty,+\infty] $ such that, for every $x\in \mathcal{X}$ and $\eta \in \R$,
    \begin{equation}\label{eq:limlogaritmic}
        \lim_{t\rightarrow \infty}\frac{1}{t}\Lambda_t^x(t\eta)= \Lambda(\eta).
    \end{equation}
    Let $D_\Lambda := \{\eta\in \R: \Lambda(\eta)<\infty\}$ and suppose that $0 \in \mathrm{int}(D_\Lambda)$, and on $D_\Lambda$, the convergence above is uniform in $\mathcal{X}$, meaning that for every $\eta\in D_{\Lambda}$,
    \begin{equation}\label{eq:GEuniformA}
        \lim_{t\rightarrow \infty}\sup_{x\in \mathcal{X}}\ \Big| \frac{1}{t}\Lambda_t^x(t\eta) -  \Lambda(\eta) \Big| =0.
    \end{equation}
    Let $\Lambda^*:\R\rightarrow [-\infty,\infty]$ denote the Legendre transform of $\Lambda$, given by
    \begin{equation}\label{eq:legendre}
        \Lambda^*(\zeta) = \sup_{\eta\in \R} \sqbracket{\eta \zeta-\Lambda(\eta)}.
    \end{equation}
    Let $\mathcal{L}\subset \R$ be the set of exposed points of $\Lambda^*$ for which there exists a strictly supporting hyperplane $\eta \in\mathrm{int}(D_\Lambda)$. Then,
    \begin{enumerate}[(a)]
        \item for every closed set $\mathcal{C}\subset \R$, $\limsup_{t\rightarrow \infty} \sup_{x\in \mathcal{X}}\frac{1}{t}\log \mu_t^x(\mathcal{C})\leq -\inf_{\zeta\in \mathcal{C}}\Lambda^*(\zeta)$,
        \item for every open set $\mathcal{O}\subset \R$, $\liminf_{t\rightarrow \infty}\inf_{x\in \mathcal{X}}\frac{1}{t}\log \mu_t^x(\mathcal{O})\geq -\inf_{\zeta\in \mathcal{O}\cap \mathcal{L}}\Lambda^*(\zeta)$,
    \end{enumerate}
    Moreover, if:
    \begin{enumerate}
        \item $\Lambda$ is lower semicontinuous on $\R$,
        \item $\Lambda$ is differentiable on $\mathrm{int}(D_\Lambda)$,
        \item $\lim_{\eta\rightarrow \partial D_{\Lambda}}| \Lambda^\prime(\eta)|=\infty $,
    \end{enumerate}
    then we can replace the set $O\cap \mathcal{L}$ by $O$ in the lower bound in (b) above.
\end{theorem}

We introduce two useful lemmas from \cite{MR2571413}. We say a function $R:\R \to [0,\infty]$ is a {\em rate function} if it is lower semicontinuous; we say it is a {\em good} rate function if additionally $\{x \in \R: R(x) \le \alpha\}$ is compact for all $\alpha \ge 0$.

\begin{lemma}\cite[Lemma 2.3.9]{MR2571413}.\label{lemma00}
    Any function $\Lambda$ as in \eqref{eq:limlogaritmic} is convex and satisfies $ \Lambda >-\infty$. Moreover, its Legendre transform $\Lambda^*$ is a convex good rate function.
\end{lemma}

We will frequently use the following inequality. Let $N>0$ be a fixed integer. For $i\in [N]:=\{1,...,N\}$, let $ a^i\geq 0$, such that $\sum_{i=1}^{N}a^i>0$. Then,
\begin{equation}\label{eq:lemma1GE}
    \log\bigg(\sum_{i=1}^Na^i \bigg)\leq  \log N + \max_{i\in [N]} \log a^i.
\end{equation}

First, we prove the upper bound (a) for compact sets. The following is proved via a ``uniform version'' of the argument given in 
\cite[Theorem 4.5.3 (b)]{MR2571413}.
\begin{lemma}\label{lemma3.0}
    Under the hypotheses of Theorem \ref{thm:GEgenralizationA}, for any compact set $\mathcal{C}\subset \R$,
    \begin{equation}\label{eq:lemma3.00}
        \limsup_{t\rightarrow \infty } \sup_{x\in \mathcal{X}} \frac{1}{t}\log \mu^x_t(\mathcal{C})\leq - \inf_{\mathcal{C}} \Lambda^*(\zeta),
    \end{equation}
    where $\Lambda^*$ is the Legendre transform of $\Lambda$. 
\end{lemma}
\begin{proof}
    Fix a compact set $\mathcal{C} \subset \R$ and $\delta>0$. Let us introduce the function $I^\delta: \R\rightarrow \R$ defined by 
    \begin{equation*}
        I^\delta(\zeta):=\min\{\Lambda^*(\zeta)-\delta,1/\delta\}.
    \end{equation*}
    We know from Lemma \ref{lemma00} that $\Lambda^*$ is a convex good rate function, and in particular it takes values in $[0,+\infty]$, and the set $\{\Lambda^* = I^\delta +\delta \}=\{ \Lambda^* \leq (1/\delta) + \delta \}$ is compact. It follows that the set $S=\{ \Lambda^* \leq (1/\delta) + \delta \}\cap \mathcal{C}$ is compact, and since $\Lambda^*$ is lower semicontinuous,  $\Lambda^*$ therefore attains the value $\inf_{x \in S}\Lambda^*(x)$ within $S$, provided $S \ne \emptyset$. In particular, if $\inf_\mathcal{C}\Lambda^*(\zeta)<+\infty$, then there exists $\delta\in (0,1)$ such that 
    \begin{equation*}
        \inf _{\zeta\in\mathcal{C}}\Lambda^*(\zeta) = \inf_{\zeta\in\mathcal{C}\cap \{ \Lambda^* \leq (1/\delta) + \delta \}}\Lambda^*(\zeta)=\inf_{\zeta\in\mathcal{C}} I^\delta(\zeta) +\delta.
    \end{equation*}
    On the other hand, if  $\inf_{\zeta\in\mathcal{C}}\Lambda^*(\zeta)=+\infty$, then $\inf_\mathcal{C}I^\delta(\zeta) = 1/\delta$. It follows that 
    \begin{equation}\label{eq:limI}
        \lim_{\delta \rightarrow 0^+}\inf_{\zeta\in\mathcal{C}}I^\delta (\zeta)=\inf_{ \zeta\in\mathcal{C}}\Lambda^*(\zeta).
    \end{equation}
    
    Next, by the definitions of $I^\delta$ and of $\Lambda^*$, for any $\zeta\in \mathcal{C}$ there exists $\eta_\zeta\in \R$ such that 
    \begin{equation}\label{eq:lemma3.01}
        \eta_\zeta  \zeta - \Lambda(\eta_\zeta)\geq I^\delta(\zeta).
    \end{equation}
    Since $I^\delta(\zeta) \ge -\delta$, 
    it follows that $\eta_\zeta \in D_\Lambda$. 
    Fix $x\in \mathcal{X}$ and let $Z$ be a random variable with law $\mu_t^x$, so that $\mu_t^x(B)=\mathbb{P}(Z \in B)$. For any $\eta\in \R$ and $r > 0$, by Markov's inequality,
    \begin{equation*}\label{MI1}
        \begin{split}
            \prob{}{}{Z \in B(\zeta,r)}&\leq \prob{}{}{Z \geq \zeta-r} \leq \probbb{}{}{e^{Z|\eta| }\geq e^{(\zeta-r)|\eta|}} \leq \mathbb{E}{e^{Z|\eta|}}e^{(r-\zeta)|\eta|}=\mathbb{E}{e^{|\eta|(Z-\zeta)}}e^{r|\eta|}.
        \end{split}
    \end{equation*}
    Similarly,
    \begin{equation*}\label{MI2}
        \begin{split}
            \prob{}{}{Z \in B(\zeta,r)}&\leq \prob{}{}{Z \leq \zeta+r} \leq \probbb{}{}{e^{-Z|\eta| }\geq e^{-(\zeta+r)|\eta|}} \leq \mathbb{E}{e^{-Z|\eta|}}e^{(r+\zeta)|\eta|}=\mathbb{E}{e^{-|\eta|(Z-\zeta)}}e^{r|\eta|}.
        \end{split}
    \end{equation*}
  By the above two estimates, we conclude that for any $\eta\in \R$,
    \begin{equation*}
        \prob{}{}{Z \in B(\zeta,r)}\leq \expec{}{}{e^{\eta(Z-\zeta)}}e^{r|\eta|}.
    \end{equation*}
    Letting $\eta=\eta_\zeta t$ and $r=r(\delta,\zeta)=\delta/|\eta_\zeta|$, it follows that 
    \begin{equation}\label{eq:lemma3.02}
        \frac{1}{t} \log \prob{}{}{Z \in B(\zeta,r)} \leq \frac{1}{t}\log \expec{}{}{e^{tZ\eta_\zeta}}  -\eta_\zeta\zeta + r|\eta_\zeta |=  -\bigl(\eta_\zeta\zeta-\frac{1}{t}\Lambda_t^x(t\eta_{\zeta}) \bigl) +\delta.
    \end{equation}
Since $\mathcal{C}$ is compact, we can choose a finite cover $\{B(\zeta_i,r_i),i \in [N]\}$ of $\mathcal{C}$ where each $r_{i}=r_i(\delta,\zeta_i)=\delta/|\eta_{\zeta_i}|$. By \eqref{eq:lemma1GE} and the upper bound \eqref{eq:lemma3.02},
    \begin{equation*}
        \frac{1}{t} \log\mu_t^x(\mathcal{C})=\frac{1}{t} \log \prob{}{}{Z \in \mathcal{C}}\leq \frac{1}{t} \log \bigg(\sum_{i=1}^{N} \prob{}{}{Z \in B(\zeta_i,r_i)}\bigg)  \leq \frac{\log N}{t} + \delta - \min_{i\in [N]} \bigl\{ \eta_{\zeta_i}\zeta_i -\frac{1}{t}\Lambda_t^x(t\eta_{\zeta_i}) \bigl\}.
    \end{equation*}
    Fix $\epsilon>0$. Recall that by \eqref{eq:lemma3.01},  $\Lambda(\eta_{\zeta_i})< +\infty$ for all $i \in [N]$. Hence, by \eqref{eq:GEuniformA}, for every $t>0$ sufficiently large and for all $x\in \mathcal{X}$,  
    \begin{equation*}
        -\min_{i\in [N]} \bigl\{ \eta_{\zeta_i}\zeta_i -\frac{1}{t}\Lambda_t^x(t\eta_{\zeta_i}) \bigl\} \leq  -\min_{i\in [N]}\{ \eta_{\zeta_i}\zeta_i- \Lambda(\eta_{\zeta_i})-\epsilon \}. 
    \end{equation*}
    Hence, since $\eta_{\zeta_i}\zeta_i-\Lambda(\eta_{\zeta_i})\ge I^\delta(\zeta_i)$ by \eqref{eq:lemma3.01}, and since every $\zeta_i\in \mathcal{C}$, it follows from \eqref{eq:limI} that,
    \begin{equation*}
        \limsup_{t\rightarrow \infty} \sup_{x\in \mathcal{X}} \frac{1}{t} \log \mu_t^x( \mathcal{C})\leq \delta + \epsilon - \min_{i\in [N]} I^\delta(\zeta_i)  \leq \delta + \epsilon - \inf_{\zeta\in \mathcal{C}} I^\delta(\zeta)\xrightarrow[]{\delta \rightarrow 0^+} \epsilon - \inf_{\zeta\in \mathcal{C}}\Lambda^*(\zeta).
    \end{equation*}
    Since this inequality holds for every $\epsilon>0$, we conclude \eqref{eq:lemma3.00}.
\end{proof}
 
To show that the bound (a) holds for every closed set $ \mathcal{C}$, we require an auxiliary lemma. We say that a family of measures $\{\mu_t^x\}_{t\geq 0, x\in \mathcal{X}}$ on $(\R,\mathcal{B}(\R))$ is {\em exponentially tight, uniformly in $\mathcal{X}$} to mean that for every $\alpha\in \R$, there exists a compact set $\mathcal{K}_\alpha\subset \R$ such that
\begin{equation}\label{eq:exponentialtight}
    \limsup_{t\rightarrow\infty} \sup_{x\in \mathcal{X}}\frac{1}{t} \log \mu^x_t(\mathcal{K}_\alpha^c)\leq -\alpha.
\end{equation}

\begin{lemma}\label{lemma3.1}
    Under the hypotheses of Theorem \ref{thm:GEgenralizationA}, the family $\{ \mu^x_t,t\geq 0\}_{ x\in \mathcal{X}}$ is exponentially tight, uniformly in $\mathcal{X}$.
\end{lemma}
\begin{proof}
Since $0\in \mathrm{int}(D_\Lambda)$, we can choose $\eta>0$ sufficiently small such that $[\eta,-\eta]\subset \mathrm{int}(D_\Lambda)$. Fix $\rho > 0$ and $x \in \mathcal{X}$, and let $Z$ be a random variable with law $\mu_t^x$. By the Markov inequality,
\begin{align*}
\mathbb{P}(|Z| \ge \rho) &
\le e^{-t\eta \rho} \mathbb{E}e^{t\eta|Z|} \\
& \le e^{-t\eta \rho}(\mathbb{E}e^{t\eta Z}+\mathbb{E}e^{-t\eta Z}) \\
& = 
\exp\left(
-t\eta \rho + \log(\mathbb{E}e^{t\eta Z}+\mathbb{E}e^{-t\eta Z})\right)\\
& \le 
\exp\left(-t\eta \rho
+\log 2 + \max(\Lambda_t^x(t\eta),\Lambda_t^x(-t\eta)
\right)\, ,
\end{align*}
where in the last step we have used \eqref{eq:lemma1GE}. Now fix $\epsilon > 0$; then by \eqref{eq:GEuniformA}, for all $t$ sufficiently large, and for all $x \in \mathcal{X}$, we have $|t^{-1}\Lambda_t^x(t\eta)-\Lambda(\eta)|<\epsilon$ and $|t^{-1}\Lambda_t^x(-t\eta)-\Lambda(-\eta)|<\epsilon$. It follows that for all $t$ sufficiently large, for all $x \in \mathcal{X}$ and $\rho>0$, 
\begin{equation*}
        \frac{1}{t} \log \mu^x_t([-\rho,\rho]^c)  \leq \frac{\log 2}{t}- \eta \rho +\max\{\Lambda(\eta),\Lambda(-\eta) \} +\epsilon.
    \end{equation*}
Provided we choose $\rho$ sufficiently large, the right-hand side is less than $-\alpha$ for all large $t$, so this implies \eqref{eq:exponentialtight}.
\end{proof}

The following corresponds to a uniform version of \cite[Lemma 1.2.18]{MR2571413}.
\begin{lemma}\label{lemma2GE}
    Under the hypotheses of Theorem \ref{thm:GEgenralizationA}, for any closed set $\mathcal{C}\subset \R$,
    \begin{equation}\label{eq:lemma3.0}
        \limsup_{t\rightarrow \infty } \sup_{x\in \mathcal{X}} \frac{1}{t}\log \mu^x_t(\mathcal{C})\leq - \inf_{\mathcal{C}} \Lambda^*(\zeta),
    \end{equation}
    where $\Lambda^*$ is the Legendre transform of $\Lambda$.
\end{lemma}

\begin{proof}
    Fix a closed set $\mathcal{C}\subset \R$. Let $\alpha = \inf_{\zeta \in \mathcal{C}}\Lambda^*(\zeta)$. Recall from Lemma \ref{lemma00} that $\Lambda^*$ takes values in $[0,+\infty]$, and thus $\alpha \in [0,+\infty]$. First, let us assume that $\alpha<+\infty$. By Lemma \ref{lemma3.1}, the family of measures is exponentially tight, uniformly in $\mathcal{X}$. Thus, we can find a compact set $\mathcal{K}_\alpha\subset \R$ such that \eqref{eq:exponentialtight} holds. For $x\in \mathcal{X}$, 
    \begin{equation}\label{eq:lemma2GE1}
        \begin{split}
            \frac{1}{t}\log \mu^x_t(\mathcal{C}) &\leq \frac{1}{t} \log (\mu^x_t(\mathcal{C}\cap \mathcal{K}_\alpha) + \mu^x_t(\mathcal{K}_\alpha^c))\\
            &\leq \frac{\log 2}{t} + \max\left\{\frac{\log \mu^x_t(\mathcal{C}\cap \mathcal{K}_\alpha)}{t},\frac{\log \mu^x_t( \mathcal{K}_\alpha^c) }{t} \right\},
        \end{split}
    \end{equation}
    where for the second inequality we used \eqref{eq:lemma1GE}. Since
     $\mathcal{C}\cap \mathcal{K}_\alpha$ is compact, by Lemma \ref{lemma3.0} we have,
    \begin{equation*}
        \limsup_{t\rightarrow \infty}\sup_{x\in \mathcal{X}}\frac{1}{t}\log\mu_t^x(\mathcal{C}\cap \mathcal{K}_\alpha) \leq -\inf_{ \mathcal{C}\cap \mathcal{K}_\alpha}\Lambda^*(\zeta) \leq -\alpha,
    \end{equation*}
    and from \eqref{eq:lemma2GE1}, the prior inequality, and \eqref{eq:exponentialtight}, we deduce that
    \begin{equation*}
        \limsup_{t\rightarrow\infty} \sup_{x\in \mathcal{X}}\frac{1}{t} \log \mu^x_t(\mathcal{C})\leq -\alpha.
    \end{equation*}

    Finally, if $\inf_\mathcal{C}\Lambda^*(\zeta)=+\infty$, we can replicate the previous argument for every non-negative $\alpha<\inf_{\zeta \in \mathcal{C}}\Lambda^*(\zeta)$, from which the conclusion follows. 
\end{proof}

Before proving Theorem \ref{thm:GEgenralizationA}(b), fix $x\in \mathcal{X}$ and $\eta \in \mathrm{int}(D_\Lambda)$. By \eqref{eq:GEuniformA}, for any $t>0$ sufficiently large, $ \Lambda^x_t(t\eta) < t\Lambda(\eta)+\epsilon t<+\infty.$ Hence, for any fixed $t$ sufficiently large, we can define a probability measure $\Tilde{\mu}_t^x=\Tilde{\mu}_t^x(\cdot;\eta)$ on $(\R,\mathcal{B}(\R))$ by setting
\begin{equation*}\label{eq:lemma5.1}
    \frac{d\Tilde{\mu}^x_t}{d\mu^x_t}(\zeta;\eta) := \exp\bigl(  t\eta \zeta- \Lambda^x_t(t\eta)  \bigr).
\end{equation*}
The logarithmic Laplace transform of this measure is given by
\begin{equation*}\label{eq:lemma5.2}
    \Tilde{\Lambda}_t^x(\lambda;\eta):=\log\bigg[ \int_\R e^{\lambda \zeta} d\Tilde{\mu}_t^x(\zeta;\eta) \bigg].
\end{equation*}
The following holds for $\Tilde{\mu}^x_t$ and $\Tilde{\Lambda}^x_t$. 
\begin{lemma}\label{lemma:ApplicationUpperLarge}
    Assume the hypotheses of Theorem \ref{thm:GEgenralizationA} hold. Fix $\eta \in \mathrm{int}(D_\Lambda)$.
    \begin{enumerate}[(i)]
        \item For any $\lambda\in \R$ such that $\eta+\lambda\in D_\Lambda$,
    \begin{equation}\label{eq:lemma5.3}
        \lim_{t\rightarrow \infty} \sup_{x\in \mathcal{X}} \Big|  \frac{1}{t}\Tilde{\Lambda}^x_t(t\lambda;\eta)- \Tilde{\Lambda}(\lambda;\eta)  \Big| = 0,
    \end{equation}
    where $\Tilde{\Lambda}(\lambda;\eta):=\Lambda(\lambda+\eta)-\Lambda(\eta)\in \R$. Furthermore, the Legendre transform $ \Tilde{\Lambda}^*(\cdot\ ;\eta)$ of $\Tilde{\Lambda}(\cdot;\eta)$ satisfies 
    \begin{equation}\label{eq:lemma5.4}
        \Tilde{\Lambda}^*(\zeta;\eta)=\sup_{\lambda\in \R}\{\lambda \zeta-\Tilde{\Lambda}(\lambda;\eta)\} = \Lambda^*(\zeta)-\eta \zeta + \Lambda(\eta).
    \end{equation}
        \item Fix $\lambda\in \R$. For any $\delta>0$ sufficiently small,
        \begin{equation}\label{eq:lemma5.5}
            \limsup_{t\rightarrow \infty}\sup_{x\in \mathcal{X}} \frac{1}{t} \log \Tilde{\mu}^x_t(B(\lambda,\delta)^c;\eta) \leq -\inf_{ B(\lambda,\delta)^c} \Tilde{\Lambda}^*(\zeta;\eta).
        \end{equation}
        Furthermore, if $\inf_{ B(\lambda,\delta)^c} \Tilde{\Lambda}^*(\zeta)<+\infty$, then there exists $\zeta_0\in B(\lambda, \delta)^{c}$ such that
        \begin{equation}\label{eq:lemma5.56}
            \Tilde{\Lambda}^*(\zeta_0;\eta)=\inf_{ B(\lambda,\delta)^c} \Tilde{\Lambda}^*(\zeta;\eta).
        \end{equation}
    \end{enumerate}
\end{lemma}
\begin{proof}[Proof of Lemma \ref{lemma:ApplicationUpperLarge}(i)]
    Since $\eta\in \mathrm{int}(D_\Lambda)$ is fixed, we suppress the dependence of $\tilde{\mu}$, $\tilde{\Lambda}^{x}_{t}$ and $\tilde{\Lambda}$ on $\eta$, for ease of notation.
    By definition,
        \begin{equation*}
            \begin{split}
                \frac{1}{t}\Tilde{\Lambda}^x_t(t\lambda)- \Tilde{\Lambda}(\lambda) &= \frac{1}{t}\log\bigg[ \int_\R e^{t\lambda \zeta} d\Tilde{\mu}_t^x(\zeta) \bigg]- \Lambda(\lambda + \eta) + \Lambda(\eta)\\
                &= \frac{1}{t}\log\bigg[ \int_\R e^{t\lambda \zeta + t\eta \zeta- \Lambda^x_t(t\eta)} d\mu_t^x(\zeta) \bigg]- \Lambda(\lambda + \eta) + \Lambda(\eta)\\
                &= \frac{1}{t}\Big[\Lambda^x_t(t(\lambda+\eta))-\Lambda^x_t(t\eta) \Big] - \Lambda(\lambda + \eta) + \Lambda(\eta).
            \end{split}
        \end{equation*}
        Since $\lim_{t\rightarrow \infty}\sup_{x\in \mathcal{X}}\ \Big| t^{-1}\Lambda_t^x(t\eta) -  \Lambda(\eta) \Big| =0$ and $\lim_{t\rightarrow \infty}\sup_{x\in \mathcal{X}}\ \Big| t^{-1}\Lambda_t^x(t(\lambda+\eta)) -  \Lambda(\lambda+\eta) \Big| =0$ by \eqref{eq:GEuniformA}, we conclude \eqref{eq:lemma5.3}. To prove \eqref{eq:lemma5.4}, observe that
        \begin{align*}
            \Tilde{\Lambda}^*(\zeta)&=\sup_{\lambda\in \R}\{\lambda \zeta-\Tilde{\Lambda}(\lambda)\}\\
            &=  \sup_{\lambda\in \R}\{(\lambda+
                \eta) \zeta-\Lambda(\lambda+\eta)\}- \eta \zeta + \Lambda(\eta)\\
            &= \Lambda^*(\zeta)-\eta \zeta + \Lambda(\eta). \qedhere
    \end{align*}
    \end{proof}
        
        \begin{proof}[Proof of Lemma \ref{lemma:ApplicationUpperLarge}(ii)]
         Recall that $\eta\in \mathrm{int}(D_\Lambda)$, thus from \eqref{eq:lemma5.3}, it is immediate that the family of measures $\{\Tilde{\mu}_t^x,t\geq 0\}_{x\in \mathcal{X}}$ satisfy the hypotheses of Theorem \ref{thm:GEgenralizationA} $\Lambda$ and $D_{\Lambda}$ replaced by $\Tilde{\Lambda}$ and $D_{\Tilde{\Lambda}}=D_{\Lambda}-\eta$, respectively. By Lemma \ref{lemma2GE}, the uniform upper bound in Theorem \ref{thm:GEgenralizationA}(a) holds for $\tilde{\mu}^{x}_{t}$, and hence 
        \begin{equation*}
            \limsup_{t\rightarrow \infty}\sup_{x\in \mathcal{X}} \frac{1}{t} \log \Tilde{\mu}^x_t(B(\lambda,\delta)^c) \leq -\inf_{ B(\lambda,\delta)^c} \Tilde{\Lambda}^*(\zeta).
        \end{equation*}
        To prove the second claim of (ii), by hypothesis, we suppose that $\inf_{ B(\lambda,\delta)^c} \Tilde{\Lambda}^*(\zeta)<+\infty$. We will prove that this infimum is achieved in $ B(\lambda,\delta)^c$. To do so, let $\alpha>0$ be such that $\inf_{ B(\lambda,\delta)^c} \Tilde{\Lambda}^*(\zeta)<\alpha$. By Lemma \ref{lemma00}, $\tilde{\Lambda}^{*}$ is a convex good rate function, from which it follows that $\{\Tilde{\Lambda}^*\leq \alpha\}$ is compact. Thus, by the lower semicontinuity of $\Tilde{\Lambda}^*$, 
        \begin{equation*}
            \text{argmin}_{\zeta \in  B(\lambda,\delta)^c\cap \{\Tilde{\Lambda}^*\leq \alpha\}}\Tilde{\Lambda}^*(\zeta) \in B(\lambda,\delta)^c\cap\{\Tilde{\Lambda}^*\leq \alpha\}.
        \end{equation*}
        However, since $\inf_{B(\lambda,\delta)^c} \Tilde{\Lambda}^*(\zeta) \leq \alpha$, this implies
        \begin{equation*}
            \min_{ B(\lambda,\delta)^c}\Tilde{\Lambda}^*(\zeta) = \min_{ B(\lambda,\delta)^c\cap \{\Tilde{\Lambda}^*\leq \alpha\}}\Tilde{\Lambda}^*(\zeta).
        \end{equation*}
        This means that the infimum $\inf_{ B(\lambda,\delta)^c} \Tilde{\Lambda}^*(\zeta)$ is achieved in $B(\lambda,\delta)^c$, from which \eqref{eq:lemma5.56}  follows.
\end{proof}

\begin{lemma}\label{lemma:5}
    Suppose the hypotheses of Theorem \ref{thm:GEgenralizationA} hold. Let $\mathcal{L}\subset \R$ be the set formed by the exposed points
    of $\Lambda^*$ for which there exists a strictly supporting hyperplane of $\Lambda^*$ belonging to $\mathrm{int}(D_\Lambda)$. For every $\lambda\in \mathcal{L}$, 
    \begin{equation}\label{eq:lemma5.11}
         \lim_{\delta \rightarrow 0^+}\liminf_{t\rightarrow \infty}\inf_{x\in \mathcal{X}}\frac{1}{t}\log \mu^x_t(B(\lambda,\delta)) \geq - \Lambda^*(\lambda).
    \end{equation}
\end{lemma}

\begin{proof}
    Fix $\lambda\in \mathcal{L}$, and let $\eta_\lambda\in \mathrm{int}(D_\Lambda)$ be a strictly supporting hyperplane of $\Lambda^*$ at $\lambda$. 
    We write $\Tilde{\mu}_t^x(\cdot)=\Tilde{\mu}_t^x(\cdot,\eta_\lambda)$ and $\Tilde{\Lambda}_t^x(\cdot)=\Tilde{\Lambda}_t^x(\cdot,\eta_\lambda)$ for readability.
    Observe that for $\delta >0$,
    \begin{equation*}
        \begin{split}
            \frac{1}{t}\log \mu^x_t(B(\lambda,\delta)) &= \frac{1}{t}\log \bigg[\int_{B(\lambda,\delta)} \exp\bigl( -t\eta_\lambda \zeta + \Lambda_t^x(t\eta_\lambda) \bigr) d\Tilde{\mu}^x_t(\zeta)\bigg]\\
            & = \frac{1}{t}\Lambda_t^x(t\eta_\lambda)+ \frac{1}{t}\log \bigg[\int_{B(\lambda,\delta)} \exp\bigl( -t\eta_\lambda \zeta  \bigr) d\Tilde{\mu}^x_t(\zeta)\bigg]\\
            & = \frac{1}{t}\Lambda_t^x(t\eta_\lambda) - \eta_\lambda \lambda + \frac{1}{t}\log \bigg[\int_{B(\lambda,\delta)} \exp\bigl( t\eta_\lambda(\lambda- \zeta)  \bigr) d\Tilde{\mu}^x_t(\zeta)\bigg]\\
            & \geq \frac{1}{t}\Lambda_t^x(t\eta_\lambda) - \eta_\lambda \lambda - |\eta_\lambda| \delta + \frac{1}{t}\log \Tilde{\mu}^x_t(B(\lambda,\delta)).
        \end{split}
    \end{equation*}
    This implies that
    \begin{equation}\label{eq:lemma5.66}
        \begin{split}
            \inf_{x\in \mathcal{X}}\frac{1}{t}\log \mu^x_t(B(\lambda,\delta)) &\geq \inf_{x\in \mathcal{X}}\bigg( \frac{1}{t}\Lambda_t^x(t\eta_\lambda) - \eta_\lambda \lambda - |\eta_\lambda| \delta + \frac{1}{t}\log \Tilde{\mu}^x_t(B(\lambda,\delta)) \bigg)\\
            & \geq  \inf_{x\in \mathcal{X}} \frac{1}{t}\Lambda_t^x(t\eta_\lambda) - \eta_\lambda \lambda -|\eta_\lambda| \delta + \inf_{x\in \mathcal{X}}\frac{1}{t}\log \Tilde{\mu}^x_t(B(\lambda,\delta)).
        \end{split}
    \end{equation}
    To control the right-hand side, first suppose $\inf_{ B(\lambda,\delta)^c} \Tilde{\Lambda}^*(\zeta)<+\infty$. By \eqref{eq:lemma5.5} and \eqref{eq:lemma5.56}, there exists $\zeta_0\in B(\lambda, \delta)^{c}$  such that
    \begin{equation}\label{eq:lemma5.7}
        \limsup_{t\rightarrow\infty} \sup_{x\in \mathcal{X}} \frac{1}{t} \log \Tilde{\mu}^x_t(B(\lambda,\delta)^c) \leq -\Tilde{\Lambda}^*(\zeta_0).
    \end{equation}
    Next, note that 
    \begin{align*}
    \Tilde{\Lambda}^*(\zeta_0) 
    & = \Lambda^*(\zeta_0) -\eta_\lambda \zeta_0 + \Lambda(\eta_\lambda) \\
    & \geq \parentesiss{\Lambda^*(\zeta_0) -\eta_\lambda \zeta_0}-\parentesiss{\Lambda^*(\lambda) -\eta_\lambda \lambda} > 0\, ,
    \end{align*}
    where the equality holds due to \eqref{eq:lemma5.4}, the first inequality holds since $\Lambda^*(\lambda)=\sup_{\eta \in \R}\{\eta\lambda - \Lambda(\eta)\} \geq  \eta_\lambda\lambda - \Lambda(\eta_\lambda)$, and the second inequality holds since $\lambda$ is an exposed point of $\Lambda^*$ with supporting hyperplane $\eta_\lambda$. 
    By \eqref{eq:lemma5.7}, this yields that
    \begin{equation*}
        \limsup_{t\rightarrow\infty} \sup_{x\in \mathcal{X}} \frac{1}{t} \log \Tilde{\mu}^x_t(B(\lambda,\delta)^c)<0.
    \end{equation*}
    By \eqref{eq:lemma5.5}, the same holds if $\inf_{ B(\lambda,\delta)^c} \Tilde{\Lambda}^*(\zeta)=+\infty$. This implies that there exists $\kappa \in (0,1)$ such that for all $t>0$ sufficiently large and all $x\in \mathcal{X}$, $t^{-1}\log \Tilde{\mu}_t^x(B(\lambda,\delta)^c)\leq -\kappa$. 
    This inequality is equivalent to the inequality $1-e^{-\kappa t}\leq \Tilde{\mu}_t^x(B(\lambda,\delta))$, and hence, for $t$ sufficiently large, and every $x\in \mathcal{X}$, $\tilde{\mu}_x^t(B(\lambda,\delta)) \ge 1/2$, so 
    \begin{equation}\label{eq:asdfg}
        \liminf_{t\rightarrow \infty} \inf_{x\in \mathcal{X}}\frac{1}{t}\log \Tilde{\mu}^x_t(B(\lambda,\delta)) =0. 
    \end{equation}
    Finally, we take a $\liminf$ as $t\rightarrow \infty$ in both sides of \eqref{eq:lemma5.66}. Since $\eta_\lambda \in \mathrm{int}(D_\Lambda)$, by \eqref{eq:GEuniformA} and \eqref{eq:asdfg}, we obtain
    \begin{equation}\label{eq:qwer}
        \liminf_{t\rightarrow \infty} \inf_{x\in \mathcal{X}} \frac{1}{t}\log\mu_t^x(B(\lambda,\delta)) \geq \Lambda(\eta_\lambda) - \eta_\lambda \lambda - |\eta_\lambda|\delta.
    \end{equation}
    By the definition of $\Lambda^*$, we have that  $\eta_\lambda \lambda -\Lambda(\eta_\lambda) \leq \Lambda^*(\lambda)$. Hence, \eqref{eq:qwer} yields that 
    \begin{equation*}
        \liminf_{t\rightarrow \infty} \inf_{x\in \mathcal{X}} \frac{1}{t}\log\mu_t^x(B(\lambda,\delta)) \geq -\Lambda^*(\lambda) - |\eta_\lambda|\delta.
    \end{equation*}
    Taking a limit as $\delta \to 0^+$, this completes the proof. 
\end{proof}

\begin{proof}[\textbf{Proof of Theorem \ref{thm:GEgenralizationA}}(b)]
    This will follow as a consequence of \eqref{eq:lemma5.11}. Fix an open set $\mathcal{O}\subset \R$ and $\lambda\in \mathcal{O}\cap \mathcal{L}$, where $\mathcal{L}$ is the set of exposed points of $\Lambda^{*}$ for which there exists a strictly supporting hyperplane of $\Lambda^*$ belonging to $\mathrm{int}(D_\Lambda)$. Let $\delta\in (0,1)$ sufficiently small such that $B(\lambda,\delta)\subset \mathcal{O}$. This implies that for every $t>0$, 
    \begin{equation*}
        \inf_{x\in \mathcal{X}}\frac{1}{t}\log \mu_t^x(\mathcal{O}) \geq \inf_{x\in \mathcal{X}}\frac{1}{t}\log \mu_t^x(B(\lambda,\delta)).
    \end{equation*}
    By \eqref{eq:lemma5.11}, taking $\liminf$ as $t\rightarrow \infty$ and then letting $\delta\downarrow 0$, we obtain
    \begin{equation*}
        \liminf_{t\rightarrow \infty}\inf_{x\in \mathcal{X}}\frac{1}{t}\log \mu_t^x(\mathcal{O}) \geq -\Lambda^*(\lambda).
    \end{equation*}
    Since $\lambda\in \mathcal{O}\cap \mathcal{L}$ was arbitrary, this proves Theorem \ref{thm:GEgenralizationA}(b). 

    Finally, the last claim of the theorem follows exactly as in the proof of \cite[Theorem 2.3.6(c)]{MR2571413}, because this argument only depends on geometric properties of the function $\Lambda$.
\end{proof}

Finally, equipped with the G\"artner-Ellis framework, we give the proof of Lemma \ref{lemma:propI}, which yields desirable properties of $I_{e}(\zeta)=[\gamma(e, \lambda_{e}+\cdot)-\gamma(e, \lambda_{e})]^{*}(\zeta)$.

\begin{proof}[\textbf{Proof of Lemma \ref{lemma:propI}(i)}]
    Since $e\in S^{d-1}$ is fixed, let us write $\gamma(\lambda)=\gamma(e,\lambda)$. Let $\Lambda:\R\rightarrow\R$ be defined by $$\Lambda(\eta)=\gamma(\lambda_e+\eta)-\gamma(\lambda_e).$$
        By Proposition \ref{Prop:PropertiesGamma}(i),
      $\Lambda$ is finite for each $\eta$, strictly convex, and differentiable. Moreover, by the quadratic growth \eqref{eq:BoundGamma}, it follows that the Legendre transform $\Lambda^{*}(\zeta)=I_{e}(\zeta)$ is finite for every $\zeta\in \R$ (and in particular, the supremum in the definition of the Legendre transform is achieved). 
      
       Next, as in the proof of Proposition \ref{proposition:largeDeviation}, we observe that for any sequence $(\eta_n)_{n\geq 1}$ such that $|\eta_n|\rightarrow +\infty$,
        \begin{equation*}
            \lim_{n\rightarrow \infty}|\Lambda^\prime(\eta_n)|=+\infty.
        \end{equation*}

        Since $I_e$ is the Legrendre transform of $\Lambda$, by \cite[Theorem 26.6]{Rockafellar+1970} and \cite[Lemma 26.7]{Rockafellar+1970}, this implies that $I_e$ is strictly convex and differentiable.
       \end{proof}
       
    \begin{proof}[\textbf{Proof of Lemma \ref{lemma:propI}(ii)}]
        
     For $\zeta\in \R$, let $q_\zeta:\R\rightarrow \R$ be the line with slope $\zeta$ passing through $\gamma(\lambda_{e})$, given by 
        \begin{equation*}
            q_\zeta(\lambda) = (\lambda-\lambda_e)\zeta+\gamma(\lambda_e).
        \end{equation*}
        Then 
        \begin{equation*}
           \begin{split}
                I_e(\zeta)=\sup_{\eta\in \R}\{\eta \zeta-\Lambda(\eta)\}=\sup_{\eta\in \R}\{q_\zeta(\lambda_e+\eta) -\gamma(\lambda_e+\eta)\}=\sup_{\eta\in \R}\{q_\zeta(\eta) -\gamma(\eta)\}.
           \end{split}
        \end{equation*}
        From the identity above, we observe that $I_e(\zeta)\geq q_\zeta(\lambda_e)-\gamma(\lambda_e)=0$. Furthermore, $I_e(\zeta)>0$ if there is at least one $\eta\in \R$ for which $\gamma(\eta)<q_\zeta(\eta)$, and since $\gamma$ is differentiable, that $I_e(\zeta)=0$ if and only if the line $q_\zeta$ is tangent to the graph of $\gamma$ from below.
        
        From Proposition~\ref{prop:PrincipalEigenEx}(i), we know that $\gamma$ is strictly convex and differentiable. Thus, there is only one line passing through the point $(\lambda_e,\gamma(\lambda_e))$ that is tangent to the graph of the function $\gamma$ from below. Proposition~\ref{prop:PrincipalEigenEx}(iii) implies that such a line has slope $c^*(e)$, and hence $I_e(c^*(e))=0$. 
        
        On the other hand,  consider $\zeta\in \R$ such that $\zeta\neq c^*(e)$. From our previous observations, the line $q_\zeta$ is not tangent to the graph of $\gamma$ at $(\lambda_e,\gamma(\lambda_e))$, but $q_\zeta(\lambda_e)=\gamma(\lambda_e)$. By the strict convexity and differentiability of $\gamma$, there exists $\eta_\zeta\in \R$ such that $q_\zeta(\eta_\zeta)>\gamma(\eta_\zeta)$. We deduce that $I_e(\zeta)>0$ for $\zeta\neq c^*(e)$.

To show that $I_e$ is strictly increasing on $(\cs(e),+\infty)$, fix $\zeta_1,\zeta_2\in  (\cs(e),\infty)$ with $\zeta_1<\zeta_2$. By (i), we know that there exists $\zeta_1^* \in \R$ such that $I_e(\zeta_1) =  q_{\zeta_1}(\zeta_1^*) -\gamma(\zeta_1^*)>0$. We claim that $\zeta_1^*>\lambda_e$. To see this, recall from Remark \ref{r:derivativecs} that $\cs(e)$ is the derivative of $\gamma$ at $\lambda_e$. By convexity of $\gamma$, implies that for every $\eta\in \R$, $q_{\cs(e)}(\eta)\leq \gamma(\eta)$. Since $q_{\zeta_1}$ is the line with slope $\zeta_1>c^{*}(e)$ passing through $\gamma(\lambda_e)$, for every $\eta\leq \lambda_e$, $q_{\zeta_1}(\eta)-\gamma(\eta)\leq q_{\cs(e)}(\eta)- \gamma(\eta)\leq 0$. Since $I_e(\zeta_1)>0$, it follows that $\zeta_1^*>\lambda_e$.

Finally, since $\zeta_1^*>\lambda_e$ and $\zeta_2>\zeta_1$, we have
        \begin{equation*}
            \begin{split}
                I_e(\zeta_1) &= q_{\zeta_1}(\zeta_1^*) -\gamma(\zeta_1^*)\\ 
                & = (\zeta_1^*-\lambda_e)\zeta_1+\gamma(\lambda_e) -\gamma(\zeta_1^*)\\
                & = (\zeta_1^*-\lambda_e)(\zeta_1-\zeta_2)+ q_{\zeta_2}(\zeta_1^*) - \gamma(\zeta_1^*)\\
                & < q_{\zeta_2}(\zeta_1^*) - \gamma(\zeta_1^*) \\
                & \leq I_e(\zeta_2).
            \end{split}
        \end{equation*}
       This argument holds for arbitrary $c^{*}(e)<\zeta_1<\zeta_2$, and thus we deduce that $I_e$ is strictly increasing on $(\cs(e),\infty)$. A symmetric argument shows that $I_e$ is strictly decreasing on $(-\infty, \cs(e))$.
    \end{proof}
    \begin{proof}[\textbf{Proof of Lemma \ref{lemma:propI}(iii)}]

       We proceed by contradiction. Suppose there exists $\beta>0$ and a sequence $\{\kappa_n\}_{n\geq 0}$ decreasing to zero such that 
        \begin{equation*}
            I_e(c^*(e)(1-\kappa_n)) \geq \kappa_n \beta.
        \end{equation*}
        Then
        \begin{equation*}
            \frac{I_e(c^*(e)(1-\kappa_n))}{\kappa_nc^*(e)} \geq \frac{\kappa_n \beta}{\kappa_n c^*(e)}= \frac{\beta}{c^*(e)}.
        \end{equation*}
        Since $I_{e}$ is differentiable by (i), we deduce that $-I_e^\prime(c^*(e))\geq \beta/c^*(e)$. This gives a contradiction since we know that~$c^*(e)$ is a minimum of $I_e$, and thus by differentiability of $I_e$, $I_e^\prime(c^*(e))=0$. A symmetric argument proves the second claim.
    
\end{proof}

\section{Appendix C: Approximation of Wulff shapes}\label{append:ApproxWulff}

\begin{proof}[\textbf{Proof of Proposition \ref{prop:wulffapprox}}(i)]
Since $0<a<\min_{e\in S^{d-1}}\mathfrak{c}^*(e)$,
        \begin{equation}
            \bigcup_{e\in S^{d-1}}\cH^+_{e,\mathfrak{c}^*(e) } = \Bigl(\bigcap_{e\in S^{d-1}}\cH^-_{e,\mathfrak{c}^*(e) }\Bigl )^c = \mathcal{W}(\mathfrak{c}^*)^c \supset  \partial(1+\epsilon)\cW(\mathfrak{c}^*).
        \end{equation}
        Thus, $\bigcup_{e\in S^{d-1}}(\cH^+_{e,\mathfrak{c}^*(e) }\cap \partial (1+\epsilon)\cW(\mathfrak{c}^*))$ is an open cover of $\partial (1+\epsilon)\cW(\mathfrak{c}^*)$. Since $\sup_{e\in S^{d-1}} \mathfrak{c}^{*}(e)=b<\infty$, $\partial(1+\epsilon)\mathcal{W}(\mathfrak{c}^{*})$ is compact, so there exists a finite set of directions $\cR\subset S^{d-1}$ such that 
        \begin{equation*}
            \bigcup_{r\in \cR}(\cH^+_{r,\mathfrak{c}^*(r) }\cap \partial (1+\epsilon)\cW(\mathfrak{c}^*))\supset \partial (1+\epsilon)\cW(\mathfrak{c}^*).
        \end{equation*}
        Since $0\in \mathrm{int}(\cW(\mathfrak{c}^*))$, the above implies that
        \begin{equation}\label{eq:ApC1}
            \bigcup_{r\in \cR}\cH^+_{r,\mathfrak{c}^*(r) } \supset  ((1+\epsilon)\cW(\mathfrak{c}^*))^c.
        \end{equation}
        By taking complements, it follows that $\bigcap_{r\in \cR} \cH_{r,\mathfrak{c}^*(r)}^-\subset(1+\epsilon)\cW$, which finishes the proof.
        \end{proof}

We next recall an auxiliary result from convex analysis which we will use in the proof of Proposition \ref{prop:wulffapprox}(ii). This result, proven in \cite[Lemma 4]{MR0400054}, originally applies to convex sets $W$ with smooth boundaries. However, the same argument holds for less regular convex sets, and yields the following version.
\begin{lemma}\cite[Lemma 4]{MR0400054}\label{l:wulff(ii)}
    Let $W\subset \R^d$ be a compact, convex set with $0\in \mathrm{int}(W)$ and $K\subset \R^d$ another compact, convex set. Suppose that for every supporting hyperplane $\cH_{e,\ell}$ of $W$, $\cH^+_{e, \ell}\cap K\neq 0$. Then $W\subset K$.
\end{lemma}

\begin{proof}[\textbf{Proof of Proposition \ref{prop:wulffapprox}}(ii)]
For $e$ in $S^{d-1}$, let
    \begin{equation}
        \mathfrak{c}(e):=\inf\{\ell>0:\cW(\mathfrak{c}^*)\subset \cH^-_{e,\ell}\}.
    \end{equation}
From this definition, it follows that $\cW(\mathfrak{c})=\cW(\mathfrak{c}^*)$, and $\mathfrak{c}^*(e)\geq \mathfrak{c}(e)$ for every $e\in S^{d-1}$. Observe that $\mathfrak{c}:S^{d-1}\to (0,\infty)$ is a continuous function since $\cW(\mathfrak{c}^*)$ is a convex compact set with $0\in \mathrm{int}(\cW(\mathfrak{c}^*))$. Let 
    \begin{equation}
        \mathcal{L}^e:=\{e'\in S^{d-1}: \cH^+_{e,\mathfrak{c}^*(e) }\cap B(0,2a^{-1})\subset\cH^+_{e',(1-\epsilon)\mathfrak{c}(e')}\cap B(0,2a^{-1})\}.
    \end{equation}
    Since $\mathfrak{c}^*(e)\geq \mathfrak{c}(e)$, it follows that for every $\epsilon\in (0,1/2)$, $\mathfrak{c}^*(e)>(1-\epsilon)\mathfrak{c}(e)$, and hence $e\in \mathcal{L}^e$. We claim that in fact,
    \begin{equation}\label{e.lastclaim}
    e\in \mathrm{int}(\mathcal{L}^e).
    \end{equation}
Assuming \eqref{e.lastclaim}, we continue the proof, and return to verify this claim at the end of the argument. 

As a consequence of \eqref{e.lastclaim}, $S^{d-1}\subset \cup_{e\in S^{d-1}}\mathrm{int}(\mathcal{L}^e)$, so by compactness of $S^{d-1}$, there exists a finite set $\cQ\subset S^{d-1}$such that $S^{d-1}\subset \cup_{q\in \cQ}\mathrm{int}(\mathcal{L}^q)$.
 
Let $K$ be a compact, convex set such that $K\subset B(0,2a^{-1})$ and for every $q\in \cQ$, $K\cap \cH^+_{q,\mathfrak{c}^*(q)}\neq \emptyset$. In view of Lemma \ref{l:wulff(ii)}, to prove that $(1-\epsilon)\cW(\mathfrak{c}^*)\subset K$, it is sufficient to show that for every $e\in S^{d-1}$, $K\cap \cH^+_{e,(1-\epsilon)\mathfrak{c}(e)}\neq \emptyset$.

To do so, let us fix an arbitrary $e\in S^{d-1}$. By the choice of $\mathcal{Q}$, there exists $q\in \cQ$ such that $e\in \mathrm{int}(\mathcal{L}^q)$. Thus, $\cH^+_{q,\mathfrak{c}^*(q) }\cap B(0,2a^{-1})\subset\cH^+_{e,(1-\epsilon)\mathfrak{c}(e)}\cap B(0,2a^{-1})$. By hypothesis on $K$, we observe that 
    \begin{equation*}
        \begin{split}
            K\cap \cH^+_{e,(1-\epsilon)\mathfrak{c}(e)} &= K\cap \cH^+_{e,(1-\epsilon)\mathfrak{c}(e)}\cap B(0,2a^{-1})\\
            & \supset K\cap \cH^+_{q,\mathfrak{c}^*(q)}\cap B(0,2a^{-1})\\
            & = K\cap \cH^+_{q,\mathfrak{c}^*(q)}\neq \emptyset.
        \end{split}
    \end{equation*}
    Since $e\in S^{d-1}$ was chosen arbitrarily, we conclude that $K\cap \cH^+_{e,(1-\epsilon)\mathfrak{c}(e)}\neq \emptyset$ for every $e\in S^{d-1}$, which implies that $K\supset (1-\epsilon)\cW(\mathfrak{c}^*)$.
    
    Finally, let us prove \eqref{e.lastclaim}. Fix $e\in S^{d-1}$; it is enough to show that there exists $\theta\in (0,1)$ such that, for every $e'\in S^{d-1}$ with $e\cdot e'\geq 1-\theta$,
    \begin{equation*}
        \cH^+_{e,\mathfrak{c}^*(e) }\cap B(0,2a^{-1})\subset\cH^+_{e',(1-\epsilon)\mathfrak{c}(e')}\cap B(0,2a^{-1}).
    \end{equation*}
    Let $\kappa=\kappa(\epsilon)\in (0,1)$ be sufficiently small such that $(1-\epsilon/4)(1-\epsilon/2)^{-1}>1+\kappa$. Since $\mathfrak{c}$ is continuous, there exists $\theta \in (0,1)$ such that $\mathfrak{c}(e')\in (\mathfrak{c}(e)(1-\kappa),\mathfrak{c}(e)(1+\kappa))$ whenever $e'\cdot e > 1-\theta$. Since $\mathfrak{c}^*\geq \mathfrak{c}$, it follows that for every $e'\in S^{d-1}$ with $e'\cdot e>1-\theta$,
    \begin{equation*}
        \mathfrak{c}^*(e)\geq \mathfrak{c}(e)>(1-\epsilon/4)\mathfrak{c}(e) > \frac{(1-\epsilon/4)}{(1+\kappa)}\mathfrak{c}(e')>(1-\epsilon/2)\mathfrak{c}(e').
    \end{equation*}
    Thus, for $x\in \cH^+_{e,\mathfrak{c}^*(e) }\cap B(0,2a^{-1})$ and $e'\cdot e>1-\theta$,
    \begin{equation*}
        \begin{split}
            x\cdot e' &= x\cdot (e'-e) + x\cdot e\\
            & \geq -\sup_{x\in \cH^+_{e,\mathfrak{c}^*(e)}\cap \overline{B(0,2a^{-1})} }|x\cdot (e'-e)| + \mathfrak{c}^*(e)\\
            & \geq -\sup_{x\in \cH^+_{e,\mathfrak{c}^*(e)}\cap \overline{B(0,2a^{-1})} }|x\cdot (e'-e)| + (1-\epsilon/2)\mathfrak{c}(e').
        \end{split}
    \end{equation*}
    By making $\theta$ even smaller, we can further impose that for all $e'\in S^{d-1}$ with $e'\cdot e > 1-\theta$,
    \begin{equation*}
        \sup_{x\in \cH^+_{e,\mathfrak{c}^*(e)}\cap \overline{B(0,2a^{-1})} }|x\cdot (e'-e)| < (\epsilon/4)\inf_{\nu\in S^{d-1}}\mathfrak{c}(\nu).
    \end{equation*}
    Combining this with the prior display, we obtain that for any $x\in \cH^+_{e,\mathfrak{c}^*(e) }\cap B(0,2a^{-1})$ and for any $e'$ such that $e'\cdot e>1-\theta,$
    \begin{equation*}
        x\cdot e' \geq -(\epsilon/4) \inf_{\nu\in S^{d-1}}\mathfrak{c}(\nu) +  (1-\epsilon/2)\mathfrak{c}(e') > (1-\epsilon)\mathfrak{c}(e').
    \end{equation*}
   This implies that $\cH^+_{e,\mathfrak{c}^*(e) }\cap B(0,2a^{-1})\subset \cH^+_{e,(1-\epsilon)\mathfrak{c}(e) }\cap B(0,2a^{-1})$ for every $e'$ with $e'\cdot e>1-\theta$. We conclude that $e\in \mathrm{int}(\mathcal{L}^e)$.
    
\end{proof}

\section*{Acknowledgements}
LAB acknowledges support from the NSERC Discovery Grants program and from the Canada Research Chairs program. This material is based in part upon work supported by the National Science Foundation under Grant No. DMS-1928930, while LAB was in residence at the Simons Laufer Mathematical Sciences Institute in Berkeley, California, during the semester of Spring 2025. 
AAA gratefully acknowledges the financial support provided by the Fonds de recherche du Qu\'ebec Doctoral research scholarship (DOI: \hyperlink{}{https://doi.org/10.69777/336282}). This work was partially funded by this scholarship, which provided the crucial resources and dedicated time necessary to conduct this study. JL acknowledges support from the NSERC Discovery Grants program and from the Canada Research Chairs program.

\footnotesize
\bibliographystyle{plain}
\bibliography{refs}

\begin{thebibliography}{10}

\bibitem{AW}
D.~G. Aronson and H.~F. Weinberger.
\newblock Multidimensional nonlinear diffusion arising in population genetics.
\newblock {\em Adv. in Math.}, 30(1):33--76, 1978.

\bibitem{MR2839402}
A.~Bensoussan, J.-L. Lions, and G.~Papanicolaou.
\newblock {\em Asymptotic analysis for periodic structures}.
\newblock AMS Chelsea Publishing, Providence, RI, 2011.
\newblock Corrected reprint of the 1978 original [MR0503330].

\bibitem{BHN}
Henri Berestycki, Fran\c{c}ois Hamel, and Nikolai Nadirashvili.
\newblock The speed of propagation for {KPP} type problems. {I}. {P}eriodic
  framework.
\newblock {\em J. Eur. Math. Soc. (JEMS)}, 7(2):173--213, 2005.

\bibitem{MR2155900}
Henri Berestycki, Fran\c{c}ois Hamel, and Lionel Roques.
\newblock Analysis of the periodically fragmented environment model. {II}.
  {B}iological invasions and pulsating travelling fronts.
\newblock {\em J. Math. Pures Appl. (9)}, 84(8):1101--1146, 2005.

\bibitem{MR4493578}
Henri Berestycki and Gr\'egoire Nadin.
\newblock Asymptotic spreading for general heterogeneous {F}isher-{KPP} type
  equations.
\newblock {\em Mem. Amer. Math. Soc.}, 280(1381):vi+100, 2022.

\bibitem{BKLMZ}
Julien Berestycki, Yujin~H. Kim, Eyal Lubetzky, Bastien Mallein, and Ofer
  Zeitouni.
\newblock The extremal point process of branching {B}rownian motion in
  {$\mathbb{R}^d$}.
\newblock {\em Ann. Probab.}, 52(3):955--982, 2024.

\bibitem{MR0518327}
J.~D. Biggins.
\newblock The asymptotic shape of the branching random walk.
\newblock {\em Advances in Appl. Probability}, 10(1):62--84, 1978.

\bibitem{Biggins2}
J.~D. Biggins.
\newblock Spatial spread in branching processes.
\newblock In {\em Biological growth and spread ({P}roc. {C}onf., {H}eidelberg,
  1979)}, volume~38 of {\em Lecture Notes in Biomath.}, pages 57--67. Springer,
  Berlin-New York, 1980.

\bibitem{MR0494541}
Maury~D. Bramson.
\newblock Maximal displacement of branching {B}rownian motion.
\newblock {\em Comm. Pure Appl. Math.}, 31(5):531--581, 1978.

\bibitem{MR0400054}
E.~M. Bron\v{s}te\u{\i}n and L.~D. Ivanov.
\newblock The approximation of convex sets by polyhedra.
\newblock {\em Sibirsk. Mat. \v{Z}.}, 16(5):1110--1112, 1132, 1975.

\bibitem{CCKW}
Xin Chen, Zhen-Qing Chen, Takashi Kumagai, and Jian Wang.
\newblock Periodic homogenization of nonsymmetric {L}\'evy-type processes.
\newblock {\em Ann. Probab.}, 49(6):2874--2921, 2021.

\bibitem{MR2303944}
Francis Comets and Serguei Popov.
\newblock On multidimensional branching random walks in random environment.
\newblock {\em Ann. Probab.}, 35(1):68--114, 2007.

\bibitem{MR2365644}
Francis Comets and Serguei Popov.
\newblock Shape and local growth for multidimensional branching random walks in
  random environment.
\newblock {\em ALEA Lat. Am. J. Probab. Math. Stat.}, 3:273--299, 2007.

\bibitem{MR2571413}
Amir Dembo and Ofer Zeitouni.
\newblock {\em Large deviations techniques and applications}, volume~38 of {\em
  Stochastic Modelling and Applied Probability}.
\newblock Springer-Verlag, Berlin, 2010.
\newblock Corrected reprint of the second (1998) edition.

\bibitem{MR1242986}
J\"urgen Eckhoff.
\newblock Helly, {R}adon, and {C}arath\'eodory type theorems.
\newblock In {\em Handbook of convex geometry, {V}ol.\ {A}, {B}}, pages
  389--448. North-Holland, Amsterdam, 1993.

\bibitem{Fisher}
R.A. Fisher.
\newblock The wave of advance of advantageous genes.
\newblock {\em Ann. of Eugenics}, 7(4):335–369, 1937.

\bibitem{Fpaper}
Mark~I. Freidlin.
\newblock On wavefront propagation in periodic media.
\newblock In {\em Stochastic analysis and applications}, volume~7 of {\em Adv.
  Probab. Related Topics}, pages 147--166. Dekker, New York, 1984.

\bibitem{Freidlin+1985}
Mark~Iosifovich Freidlin.
\newblock {\em Functional Integration and Partial Differential Equations.
  (AM-109), Volume 109}.
\newblock Princeton University Press, Princeton, 1985.

\bibitem{MR553200}
Ju. Gartner and M.~I. Fre\u{\i}dlin.
\newblock The propagation of concentration waves in periodic and random media.
\newblock {\em Dokl. Akad. Nauk SSSR}, 249(3):521--525, 1979.

\bibitem{MR0670523}
J\"urgen G\"artner.
\newblock Location of wave fronts for the multidimensional {KPP} equation and
  {B}rownian first exit densities.
\newblock {\em Math. Nachr.}, 105:317--351, 1982.

\bibitem{MR3043938}
Fran\c{c}ois Hamel, James Nolen, Jean-Michel Roquejoffre, and Lenya Ryzhik.
\newblock A short proof of the logarithmic {B}ramson correction in
  {F}isher-{KPP} equations.
\newblock {\em Netw. Heterog. Media}, 8(1):275--289, 2013.

\bibitem{MR3463416}
Fran\c{c}ois Hamel, James Nolen, Jean-Michel Roquejoffre, and Lenya Ryzhik.
\newblock The logarithmic delay of {KPP} fronts in a periodic medium.
\newblock {\em J. Eur. Math. Soc. (JEMS)}, 18(3):465--505, 2016.

\bibitem{MR3606740}
Simon~C. Harris and Matthew~I. Roberts.
\newblock The many-to-few lemma and multiple spines.
\newblock {\em Ann. Inst. Henri Poincar\'e{} Probab. Stat.}, 53(1):226--242,
  2017.

\bibitem{R-381-PR}
Theodore~Edward Harris.
\newblock {\em The Theory of Branching Process}.
\newblock RAND Corporation, Santa Monica, CA, 1964.

\bibitem{MR4162842}
Pratima Hebbar, Leonid Koralov, and James Nolen.
\newblock Asymptotic behavior of branching diffusion processes in periodic
  media.
\newblock {\em Electron. J. Probab.}, 25:Paper No. 126, 40, 2020.

\bibitem{MR4585411}
Haojie Hou, Yan-Xia Ren, and Renming Song.
\newblock Invariance principle for the maximal position process of branching
  {B}rownian motion in random environment.
\newblock {\em Electron. J. Probab.}, 28:Paper No. 65, 63, 2023.

\bibitem{MR1121940}
Ioannis Karatzas and Steven~E. Shreve.
\newblock {\em Brownian motion and stochastic calculus}, volume 113 of {\em
  Graduate Texts in Mathematics}.
\newblock Springer-Verlag, New York, second edition, 1991.

\bibitem{MR4564433}
Yujin~H. Kim, Eyal Lubetzky, and Ofer Zeitouni.
\newblock The maximum of branching {B}rownian motion in {$\Bbb{R}^d$}.
\newblock {\em Ann. Appl. Probab.}, 33(2):1315--1368, 2023.

\bibitem{MR4998359}
Yujin~H. Kim and Ofer Zeitouni.
\newblock The shape of the front of multidimensional branching {B}rownian
  motion.
\newblock {\em Probab. Theory Related Fields}, 193(3-4):1121--1160, 2025.

\bibitem{KPP}
A.~Kolmogorov, I.~Petrovskii, and N.~Piskounov.
\newblock A study of the diffusion equation with increase in the amount of
  substance, and its application to a biological problem.
\newblock {\em Bull. Moscow Univ., Math. Mech.}, 1(6):1--26, 1937.

\bibitem{MR3497465}
Jean-Fran\c{c}ois Le~Gall.
\newblock {\em Brownian motion, martingales, and stochastic calculus}, volume
  274 of {\em Graduate Texts in Mathematics}.
\newblock Springer, [Cham], french edition, 2016.

\bibitem{linzlatos}
Jessica Lin and Andrej Zlato\v{s}.
\newblock Stochastic homogenization for reaction-diffusion equations.
\newblock {\em Arch. Ration. Mech. Anal.}, 232(2):813--871, 2019.

\bibitem{LS}
Pierre-Louis Lions and Panagiotis~E. Souganidis.
\newblock Homogenization of ``viscous'' {H}amilton-{J}acobi equations in
  stationary ergodic media.
\newblock {\em Comm. Partial Differential Equations}, 30(1-3):335--375, 2005.

\bibitem{MR4492971}
Eyal Lubetzky, Chris Thornett, and Ofer Zeitouni.
\newblock Maximum of branching {B}rownian motion in a periodic environment.
\newblock {\em Ann. Inst. Henri Poincar\'{e} Probab. Stat.}, 58(4):2065--2093,
  2022.

\bibitem{maillard2025generalisedprincipaleigenvaluesglobal}
Pascal Maillard and Oliver Tough.
\newblock Generalised principal eigenvalues and global survival of branching
  markov processes, 2025.

\bibitem{MR3417448}
Bastien Mallein.
\newblock Maximal displacement of {$d$}-dimensional branching {B}rownian
  motion.
\newblock {\em Electron. Commun. Probab.}, 20:no. 76, 12, 2015.

\bibitem{10.1214/aoap/1177004832}
Colin McDiarmid.
\newblock {Minimal Positions in a Branching Random Walk}.
\newblock {\em The Annals of Applied Probability}, 5(1):128 -- 139, 1995.

\bibitem{mckean}
H.~P. McKean.
\newblock Application of {B}rownian motion to the equation of
  {K}olmogorov-{P}etrovskii-{P}iskunov.
\newblock {\em Comm. Pure Appl. Math.}, 28(3):323--331, 1975.

\bibitem{MR2491804}
Gregoire Nadin.
\newblock The principal eigenvalue of a space-time periodic parabolic operator.
\newblock {\em Ann. Mat. Pura Appl. (4)}, 188(2):269--295, 2009.

\bibitem{MR2555178}
Gr\'egoire Nadin.
\newblock Traveling fronts in space-time periodic media.
\newblock {\em J. Math. Pures Appl. (9)}, 92(3):232--262, 2009.

\bibitem{nolen}
James Nolen.
\newblock A central limit theorem for pulled fronts in a random medium.
\newblock {\em Netw. Heterog. Media}, 6(2):167--194, 2011.

\bibitem{MR1326606}
Ross~G. Pinsky.
\newblock {\em Positive harmonic functions and diffusion}, volume~45 of {\em
  Cambridge Studies in Advanced Mathematics}.
\newblock Cambridge University Press, Cambridge, 1995.

\bibitem{MR4596348}
Yan-Xia Ren, Renming Song, and Fan Yang.
\newblock Branching {B}rownian motion in a periodic environment and existence
  of pulsating traveling waves.
\newblock {\em Electron. J. Probab.}, 28:Paper No. 72, 50, 2023.

\bibitem{MR3127890}
Matthew~I. Roberts.
\newblock A simple path to asymptotics for the frontier of a branching
  {B}rownian motion.
\newblock {\em Ann. Probab.}, 41(5):3518--3541, 2013.

\bibitem{MR1881888}
James~C. Robinson.
\newblock {\em Infinite-dimensional dynamical systems}.
\newblock Cambridge Texts in Applied Mathematics. Cambridge University Press,
  Cambridge, 2001.
\newblock An introduction to dissipative parabolic PDEs and the theory of
  global attractors.

\bibitem{Rockafellar+1970}
Ralph~Tyrell Rockafellar.
\newblock {\em Convex Analysis}.
\newblock Princeton University Press, Princeton, 1970.

\bibitem{Roquejoffre2014KPPII}
Jean-Michel Roquejoffre and Lenya Ryzhik.
\newblock {KPP} invasions in periodic media: lecture notes for the {T}oulouse
  {KPP} school, 2014.

\bibitem{MR3682669}
Luca Rossi.
\newblock The {F}reidlin-{G}\"{a}rtner formula for general reaction terms.
\newblock {\em Adv. Math.}, 317:267--298, 2017.

\bibitem{shabani2019logarithmic}
Beniada Shabani.
\newblock Logarithmic bramson correction for multi-dimensional periodic
  fisher-kpp equations.
\newblock {\em arXiv preprint arXiv:1910.08178}, 2019.

\bibitem{MR4291461}
Roman Stasi\'nski, Julien Berestycki, and Bastien Mallein.
\newblock Derivative martingale of the branching {B}rownian motion in dimension
  {$d \geq 1$}.
\newblock {\em Ann. Inst. Henri Poincar\'e{} Probab. Stat.}, 57(3):1786--1810,
  2021.

\bibitem{MR4079432}
Ji\v{r}\'i \v{C}ern\'y and Alexander Drewitz.
\newblock Quenched invariance principles for the maximal particle in branching
  random walk in random environment and the parabolic {A}nderson model.
\newblock {\em Ann. Probab.}, 48(1):94--146, 2020.

\bibitem{Weinberger}
Hans~F. Weinberger.
\newblock On spreading speeds and traveling waves for growth and migration
  models in a periodic habitat.
\newblock {\em J. Math. Biol.}, 45(6):511--548, 2002.

\bibitem{Z2}
Yuming~Paul Zhang and Andrej Zlato\v{s}.
\newblock Homogenization for space-time-dependent {KPP} reaction-diffusion
  equations and {G}-equations.
\newblock {\em Calc. Var. Partial Differential Equations}, 62(9):Paper No. 248,
  22, 2023.

\bibitem{Z1}
Andrej Zlato\v{s}.
\newblock Homogenization for time-periodic {KPP} reactions.
\newblock {\em Nonlinearity}, 36(3):1918--1927, 2023.

\end{thebibliography}

\end{document}